\documentclass[a4wide,
oneside]{article}

\usepackage{amsmath}
\usepackage{amsthm}
                      
\RequirePackage{mathtools}
\RequirePackage{amssymb}                        
\RequirePackage{latexsym}                       
\RequirePackage{stmaryrd}                       

\RequirePackage[all]{xy}                        
\UseTips                                        
\newdir^{c}{{}*!/-3pt/\dir^{(}}
\newdir_{c}{{}*!/-3pt/\dir_{(}}
\newdir^{ (}{{}*!/-3pt/\dir^{(}}
\newdir_{ (}{{}*!/-3pt/\dir_{(}}

\RequirePackage{xspace}                        
\RequirePackage{calc}                           
\RequirePackage{ifthen}                          

\RequirePackage{mathrsfs}
\renewcommand{\mathcal}[1]{\mathscr{#1}}

\usepackage[dvipsnames,usenames]{xcolor}  

\usepackage{colonequals}
\usepackage[english]{babel}      
\usepackage[ansinew]{inputenc}
\usepackage[figuresleft]{rotating}

\usepackage{hyperref}
\hypersetup{
bookmarks,
bookmarksdepth=3,
bookmarksopen,
bookmarksnumbered,
pdfstartview=FitH,
colorlinks,backref,hyperindex,
linkcolor=RawSienna,
anchorcolor=BurntOrange,
citecolor=OliveGreen,
filecolor=BlueViolet,
menucolor=Yellow,
urlcolor=OliveGreen
}
\usepackage{cleveref}

\theoremstyle{plain}

\def\theenumi{(\alph{enumi})}

\def\p@enumii{\theenumi}

\newtheorem{Thm}{Theorem}[section]
\newtheorem{Prop}[Thm]{Proposition}

\newtheorem{Cor}[Thm]{Corollary}
\newtheorem{Lem}[Thm]{Lemma}

\theoremstyle{definition}    

\newtheorem{Def}[Thm]{Definition}

\newtheorem{Facts}[Thm]{}
\newtheorem{Fact}[Thm]{Fact}

\theoremstyle{remark}    

\newtheorem{Not}[Thm]{Notation}
\newtheorem{Rem}[Thm]{Remark}
\newtheorem{Ex}[Thm]{Example}
\newtheorem{Ques}[Thm]{Question}
\newtheorem{Ass}[Thm]{Assumption}
\newtheorem{Cond}[Thm]{Conditions}
\newtheorem{Con}[Thm]{Convention}

\numberwithin{equation}{section}

\newcommand{\CG}{{\cal G}}
\newcommand{\BF}{{\mathbb{F}}\,\!{}}
\newcommand{\SL}{{\rm SL}}
\newcommand{\GL}{{\rm GL}}
\DeclareMathOperator{\id}{{id}}

\newcommand{\wh}[1]{\widehat{#1}}
\newcommand{\Sets}{\mathbf{Sets}}

\newcommand{\End}{\mathop{\rm End}\nolimits}
\newcommand{\Fm}{{\mathfrak{m}}}
\def\eps{\varepsilon}
\newcommand{\Fg}{{\mathfrak{g}}}
\newcommand{\Lie}{\mathop{\rm Lie}\nolimits}
\newcommand{\der}{{\rm der}}
\newcommand{\Gal}{\mathop{\rm Gal}\nolimits}
\def\into{\hookrightarrow}
\newcommand{\kernel}{\mathop{\rm Ker}\nolimits}
\newcommand{\Hom}{\mathop{\rm Hom}\nolimits}
\newcommand{\CO}{{\cal O}}
\newcommand{\Frob}{{\rm Frob}}
\newcommand{\BQ}{{\mathbb{Q}}}
\newcommand{\et}{{\rm et}}
\newcommand{\Spec}{\mathop{{\rm Spec}}\nolimits}
\newcommand{\CE}{{\cal E}}
\newcommand{\PGL}{{\rm PGL}}
\newcommand{\SO}{{\rm SO}}
\newcommand{\notdiv}{\mathopen{\mathchoice
             {\not{|}\,}
             {\not{|}\,}
             {\!\not{\:|}}
             {\not\:{|}\,}
             }}
\newcommand{\CZ}{{\cal Z}}
\newcommand{\Fz}{{\mathfrak{z}}}
\newcommand{\coker}{\mathop{\rm Coker}\nolimits}
\newcommand{\Fh}{{\mathfrak{h}}}
\newcommand{\BG}{{\mathbb{G}}}
\def\longto{\longrightarrow}
\newcommand{\image}{\mathop{{\rm Im}}\nolimits}
\newcommand{\BZ}{{\mathbb{Z}}}
\newcommand{\CT}{{\cal T}}
\newcommand{\Ext}{\mathop{\rm Ext}\nolimits}
\newcommand{\Fa}{{\mathfrak{a}}}
\newcommand{\Fb}{{\mathfrak{b}}}
\newcommand{\Fc}{{\mathfrak{c}}}
\newcommand{\Fs}{{\mathfrak{s}}}
\def\onto{\twoheadrightarrow}
\newcommand{\opp}{{\rm opp}}

\newcommand{\FM}{{\mathfrak{M}}}
\newcommand{\Fl}{{\mathfrak{l}}}
\newcommand{\Fp}{{\mathfrak{p}}}
\newcommand{\CH}{{\cal H}}
\newcommand{\Ft}{{\mathfrak{t}}}
\newcommand{\GSp}{{\rm GSp}}
\DeclareMathOperator{\Char}{Char}
\newcommand{\ab}{{\rm{ab}}}
\newcommand{\CU}{{\cal U}}
\DeclareMathOperator{\Aut}{Aut}
\newcommand{\wt}[1]{\widetilde{#1}}

\newcommand{\nhs}{\hspace{-.5pt}}
\newcommand{\Krull}{\mathrm{Krull}}
\newcommand{\uw}{\underline{w}}
\newcommand{\hyp}{\mathrm{hyp}}

\newcommand{\oBF}{{\overline{\BF}}}
\newcommand{\sco}{\mathrm{sc}}
\newcommand{\flatt}{\mathrm{fl}}

\DeclareMathOperator{\SU}{SU}
\DeclareMathOperator{\Fro}{Fr}
\newcommand{\CAW}{\wh{Ar}_{\WF}}
\DeclareMathOperator{\pdet}{pdet}
\DeclareMathOperator{\Sp}{Sp}
\DeclareMathOperator{\PGU}{PGU}
\DeclareMathOperator{\Fgl}{\Fg\Fl}

\DeclareMathOperator{\Fsl}{\Fs\Fl}
\DeclareMathOperator{\Fpgl}{\Fp\Fg\Fl}

\newcommand{\WF}{W(\BF)}

\newcommand{\barrhoR}{{\bar\rho_R}}
\newcommand{\barrhoH}{{\bar\rho_H}}
\newcommand{\barrhoS}{{\bar\rho_S}}

\newcommand{\ad}{\mathrm{ad}}
\newcommand{\ct}{{\textbf{(ct)}}}
\newcommand{\pf}{{\textbf{(pf)}}}
\newcommand{\lie}{{\textbf{(lie)}}}

\newcommand{\liegen}{\hbox{\textbf{(l-ge)}}}

\newcommand{\liecsc}{{\textbf{(csc)}}}

\newcommand{\lieu}{{\textbf{(l-un)}}}
\newcommand{\liec}{\hbox{\textbf{(l-cl)}}}
\newcommand{\liead}{\hbox{\textbf{(lie-ad)}}}
\newcommand{\liesc}{\hbox{\textbf{(lie-sc)}}}

\newcommand{\sch}{{\textbf{(sch)}}}

\newcommand{\van}{{\textbf{(van)}}}
\newcommand{\ns}{\textbf{(n-s)}}
\newcommand{\SMat}[4]{{{\renewcommand{\arraystretch}{.5} 
\addtolength{\arraycolsep}{-.8\arraycolsep}
\left( \begin{array}{cc} \scriptscriptstyle #1 & \scriptscriptstyle #2 \\ \scriptscriptstyle #3 & \scriptscriptstyle #4   \end{array}  \right)}}}

\DeclareMathOperator{\Sch}{{\mathbf{Sch}}}
\DeclareMathOperator{\Sh}{{Sh}}
\DeclareMathOperator{\Der}{{Der}}
\DeclareMathOperator{\Gps}{{\mathbf{Gps}}}
\newcommand{\fppf}{{\mathrm{fppf}}}
\newcommand{\hw}{{\mathrm{h.w.}}}
\newcommand{\sh}{{\mathrm{sh}}}
\newcommand{\tr}{{\mathrm{tr}}}
\DeclareMathOperator{\socle}{{soc}}
\DeclareMathOperator{\cosocle}{{csoc}}

\begin{document}

\title{Deformation rings and images of Galois representations}
\author{
Gebhard~B\"ockle
\thanks{Universit\"at~\!Heidelberg,~IWR,~Im~\!Neuenheimer~\!Feld~\!368,~69120~\!Heidelberg,~Germany,~boeckle@uni-hd.de},
Sara Arias-de-Reyna
\thanks{Universidad de Sevilla. Facultad de Mathem\'aticas, C/ S. Fernando, 4, C.P. 41004-Sevilla, Espa\~na, sara\underline{ }\underline{ }arias@us.es }}

\date{\today}

\maketitle

\begin{abstract}
Let $\CG$ be a connected reductive almost simple group over the Witt ring $\WF$ for $\BF$ a finite field of characteristic $p$. Let $R$ and $R'$ be complete noetherian local $\WF$-algebras with residue field $\BF$. Under a mild condition on $p$ in relation to structural constants of $\CG$, we show the following results: (1) Every closed subgroup $H$ of $\CG(R)$ with full residual image $\CG(\BF)$ is a conjugate of a group $\CG(A)$ for $A\subset R$ a closed subring that is local and has residue field $\BF$. (2) Every surjective homomorphism $\CG(R)\to\CG(R')$ is, up to conjugation, induced from a ring homomorphism $R\to R'$. (3) The identity map on $\CG(R)$ represents the universal deformation of the representation of the profinite group $\CG(R)$ given by the reduction map $\CG(R)\to\CG(\BF)$. This generalizes results of Dorobisz and Eardley-Manoharmayum and of Mano\-har\-mayum, and in addition provides an abstract classification result for closed subgroups of $\CG(R)$ with residually full image. 

We provide an axiomatic framework to study this type of question, also for slightly more general $\CG$, and we study in the case at hand in great detail what conditions on $\BF$ or on $p$ in relation to $\CG$ are necessary for the above results to hold.
\end{abstract}

\section{Introduction}\label{Section:Introduction}

Let $R$ be a complete noetherian local ring with finite residue field $\BF$. Let $\bar\rho\colon SL_n(R)\to\SL_n(\BF)$ be the induced group homomorphism. It is shown in \cite{Dorobisz,Manoharmayum,EM} that for all but finitely many pairs $(n,\BF)$ the universal deformation of $\bar\rho$ for representations into $\GL_n$, in the sense of \cite{Mazur}, is represented by $\id\colon \SL_n(R)\to\SL_n(R)$. The proofs were based in part on quite long and explicit computations. Our first observation was that these computations could be avoided almost entirely by a systematic use of methods from deformation theory and of largely well-known results on reductive groups over finite fields, their Lie algebras and some related cohomology groups. For a precise list of conditions, we refer the reader to the axiomatic framework introduced in \Cref{Section:Axioms}. This axiomatic viewpoint, allows us to obtain the same result for an arbitrary absolutely simple connected reductive group $\CG$ over the ring of Witt vectors $\WF$ in place of $\SL_n$. In fact, we have extensions to certain non-connected $\CG$ with $\CG^o$ as in the previous line, or to $\CG\subset\CG'$ modeling $\SL_n\subset \GL_n$; for the relevant deformation theory, we refer to~\cite{Tilouine}.

A second insight was that, by extending ideas of Boston from \cite[Appendix]{Boston}, within our axiomatic framework we can obtain general results on closed subgroups of $\SL_n(R)$ (or $\CG(R)$) with full residual image. We recover \cite[Main Theorem]{Manoharmayum} in nearly all cases, and we generalize it to our setting: given an injection $\WF\to R$, closed subgroups of $\CG(R)$ with residual image $\CG(\BF)$ typically contain a conjugate of $\CG(\WF)$. More generally, we show that certain closed subgroups $H$ of $\CG(R)$ with full residual image are simply conjugates of groups $\CG(A)$ with $A$ a closed subring of~$R$.

Let us give a concrete theorem that highlights the main results of the present article. Let $\CG$ be a connected absolutely simple linear algebraic group scheme over the ring of Witt vectors $\WF$ of a finite field $\BF$ of characteristic $p$. Assume that $p\ge5$, that $p$ does not divide $n+1$ if $\CG$ is of type~$A_n$, and that $\BF\neq\BF_5$ if $\CG$ is of type $A_1$ or $C_n$. Let $\CAW$ be the category of complete local $\WF$-algebras that are filtered inverse limits of local Artin $\WF$-algebras $R$ with residue field $\BF$. Denote by $\pi\colon R \to\BF$ the residue homomorphism and by $\CG(\pi)$ the induced homomorphism $\CG(R)\to\CG(\BF)$. Then the following holds:
\begin{Thm}\label{MainThm}
Let $R,R'$ be in $\CAW$.
\begin{enumerate}
\item If $H$ is a closed subgroup of $\CG(R)$ that surjects under $\pi$ onto $\CG(\BF)$ then there exists a closed $\WF$-subalgebra $A\subset R$ in $\CAW$ such that $H$ is conjugate to $\CG(A)$.
\item Let $\phi\colon \CG(R)\to\CG(R')$ be a surjective group homomorphism that on residue fields induces the identity of $\CG(\BF)$. Then a conjugate of $\phi$ is the map $\CG(\alpha)$ for some surjective ring homomorphism $\alpha\colon R\to R'$ in $\CAW$.
\item The functor $\CAW\to\Sets$,
\[A\mapsto\{\CG\hbox{-valued deformations of $\pi\colon\CG(R)\to\CG(\BF)$ to $A$}\}\] 
is representable and a universal deformation is given by the class of $\id_{\CG(R)}$.
\end{enumerate}
\end{Thm}
Parts (a) and (b) mean that all closed subgroups with maximal residual image and all homomorphisms between such must be `linear'. Part (c) is a consequence of (a), (b) and the definition of deformation. We shall in fact prove (a) and (c) from which (b) follows immediately.

\begin{Rem}
\begin{enumerate}
\item Part (a) is an abstract big image theorem under strong residual hypotheses. For $\GL_2$ there are more general results under weaker residual hypotheses in \cite{Bellaiche} or \cite{CLM}. One of their aims is a general description of the image of certain $p$-adic families of automorphic forms as in \cite{HIda}. We plan to apply our results in a similar way in future work.
\item Part (a) extends \cite{Manoharmayum} by Manoharmayum. There it is proved, again for $\CG=\SL_n$ and considered inside $\GL_n$, that under the hypothesis of (a), a conjugate of $H$ contains the subgroup $\CG(\WF_R)$ where $\WF_R$ is the image under the structure morphism $\WF\to R$; it is a quotient ring of $\WF$. Our part (a) implies this particular case for any $\CG$ considered, since clearly $\WF_R=\WF_A\subset A$.
\item Part (c) for $\CG=\SL_n$ and for deformations of the given $\pi$ but into $\GL_n$ is a result due to, independently, Dorobisz and Eardly-Manoharmayum; see \cite{Dorobisz,EM}. Their results are more complete than what we state above and also include $p=2,3$. \cite{Dorobisz} also characterizes the finitely many exceptions for small $p$, by giving counterexamples. In \Cref{Rem:5.4-pUniversal} we completely recover their list from our axiomatic framework. 
\item Part (c) implies in particular that any complete noetherian local ring occurs as a universal deformation ring, a question posed in \cite[Question~1.1]{BCdS}. It was the motivation behind  \cite{Dorobisz} and \cite{EM}, and was answered for $\GL_n$ in these papers.
\item For the most general results in the sense of \Cref{MainThm}(a), we refer to Corollaries~\ref{Rem:6.4-forAbsSimple} and \ref{Rem:6.4-pUniversal}, and for those in the sense of \Cref{MainThm}(c), to Corollaries~\ref{Rem:5.4-forAbsSimple} and~\ref{Rem:5.4-pUniversal}.
\end{enumerate}
\end{Rem}

We now give an outline of the article. In \Cref{Section:Axioms}, we present our basic set-up to be used throughout the remainder of this article and we formulate a number of axiomatic conditions. They have occurred, at least implicitly, in some form in the deformation theory of Galois representations for $\GL_n$. These axiomatic conditions are admittedly somewhat technical. But they match very well with what is needed in our proofs of the main results in Sections~\ref{Section3-new} and~\ref{Section4-new}. Some basic  facts on affine group schemes over general bases used in our axiomatic presentation and throughout this work are collected in \Cref{Appendix}. In \Cref{Section:Axioms}, we also formulate technical versions of \Cref{MainThm}(a) and~(c), based on our axiomatic conditions. Moreover the section contains, before \Cref{DefHPerfection}, the technically important assignment $H\mapsto H^c$ for closed subgroups $H$ of $\CG(R)$ with certain residual images. The group $H^c$ is a variant of the commutator subgroup that preserves the residual image. 

In \Cref{Section:Chevalley} we hope to convince the reader that our axiomatic conditions are natural by proving that they are satisfied for connected absolutely simple reductive groups $\CG$ over $\WF$ with `few' exceptions. \Cref{MainThm} gives a good idea of what `few' might mean, namely that they hold whenever $p$ is `large' in comparison to data coming from $\CG$. \Cref{Section:Chevalley} gives a thorough investigation of when precisely our axiomatic conditions are satisfied for absolutely simple $\CG$, even if $p$ is small! There is a finite list of `obvious' exceptions. But the validity of our conditions depends, for small $p$, also on further invariants of $\CG$, such as its center or the type of its root system; see \Cref{Thm:MainOnStdSetup} for a summary. In addition, in \Cref{Subsec:Puniv} we also study the validity of our axiomatic conditions for a situation $\CG\subset\CG'$ that resembles that of $\SL_n\subset\GL_n$; see \Cref{Thm:StdHypForPuniv}. We hope that our thorough analysis of the small prime situation, which makes up almost half of this article, is also useful to others since, as mentioned earlier, several of the conditions we investigate also occur in deformation theory. Moreover our review of the literature also discovered several untreated cases. It might be a challenge for us or others to eventually resolve these.

\Cref{Section3new} contains preparations for the proofs of our main results. We lay a number of elementary foundations of a group theoretic nature, and we explain how our axiomatic conditions are used to deduce results on group extensions, commutators etc. Some ideas are taken from \cite[Appendix]{Boston} by Boston. The rather elementary proofs of our main results are given in Sections~\ref{Section3-new} and~\ref{Section4-new}. They make essential use of standard results on universal deformations and build on the preparations from \Cref{Section3new}. The proofs are short and contain next to no explicit computations unlike \cite{Manoharmayum,Dorobisz,EM}. The reason is our method of proof. In concrete cases such as \Cref{MainThm}, we can rely on the extensive results gathered in \Cref{Section:Chevalley} that are by now standard results on algebraic groups over finite fields.

Let us end this introduction with a general question in the spirit of  \cite[Questions~1.1 and 1.2]{BCdS} but from a slightly different perspective. Suppose that $\Pi$ is a profinite group, and $\bar\rho\colon\Pi\to \GL_n(\BF)$ is a continuous homomorphism with trivial determinant.\footnote{For simplicity we phrase the question for $\GL_n$ and assume trivial determinant; the generalization to other $\CG$ and more general `determinants' is left to the reader.} Consider $\BF^n$ as being acted on by $\Pi$ via $\bar\rho$ and assume that the natural map $\BF\to\End_\Pi(\BF^n)$ is an isomorphism. Then $\bar\rho$ possesses a universal deformation $\rho^u\colon \Pi\to \GL_n(R^u)$ with trivial determinant (unique up to unique conjugacy) with $R^u\in \CAW$; essentially by \cite{Mazur}. Let $H$ be the closed compact subgroup $\rho^u(\Pi)$ of $\GL_n(R^u)$ and consider the diagram
\[\xymatrix{
\Pi\ar[r]&H \ar@{^{ (}->}[r]^-{\rho_H} \ar[dr]_-{\bar\rho_H}&\GL_n(R^u)\ar[d]\\
&&\GL_n(\BF)
\rlap{,}}\]
where the composed maps from $\Pi$ to $\GL_n(R^u)$ and to $\GL_n(\BF)$ are $\rho^u$ and $\bar\rho$, respectively. It is straightforward to see that 
\begin{Fact}
The universal deformation of $\bar\rho_H$ (with trivial determinant) exists and is represented by~$\rho_H$.
\end{Fact}
In light of \Cref{MainThm} the following seems a natural question to us:
\begin{Ques}
For which subgroups $H_\BF$ of $\SL_n(\BF)$ can one classify closed subgroups of $\GL_n(R)$ with $R\in \CAW$ in a way similar to  \Cref{MainThm}(a)?
\end{Ques}
In cases where the question has a reasonable answer, the image $H$ of a universal deformation $\rho^u$ has a uniform description, in which $R^u$ depends on $\Pi$, but the shape of $H$ depends on $R^u$ and $H_\BF$ only, not on $\Pi$ in any further way. A result  for $\GL_2$, very much in this direction, is \cite[Thm.~7.2.3]{Bellaiche}.

\subsection*{Notation and Conventions}
\begin{itemize}
\item $p$ will denote a prime number, $\BF$ a finite field of characteristic $p$ and $\WF$ its ring of Witt vectors.
\item $\CAW$ is the category of complete local $\WF$-algebras $R$ that are filtered inverse limits of local Artin $\WF$-algebras with residue field $\BF$. The maximal ideal of $R$ will be denoted $\Fm_R$.
\item The ring of dual numbers $\BF[X]/(X^2)\in\CAW$ will be denoted by $\BF[\eps]$. 
\item Algebraic groups will always be denoted by capital script letters such as $\CG$; abstract groups by roman letters such as $H$; Lie algebras by small gothic letters such as $\Fg$; for the Lie algebra of $\CG$ we write $\Lie(\CG)$ but also often simply $\Fg$.
\item The  center of a group $H$ or an algebraic group $\CG$ or a Lie algebra $\Fg$ will be $Z(H)$ or $Z(\CG)$ or $Z(\Fg)$, respectively; see also \ref{App-1}.
\item The socle of a representation is the largest semisimple subrepresentation; the socle of a Lie algebra is the largest semisimple sub Lie algebra. Dually one defines the cosocle as the largest semisimple quotient.
\item We follow the standard convention that the types of classical groups are $(A_n)_{n\ge1}$, $(B_n)_{n\ge2}$, $(C_n)_{n\ge3}$ and $(D_n)_{n\ge4}$.
\end{itemize}

\subsection*{Acknowledgements:} The authors thank A.\ Conti for many comments on a preliminary version of the manuscript. During this work G.B. received support from the DFG within the  FG 1920 and the SPP 1489. S.A.-d.-R. was partially supported by project MTM2016-75027-P, funded by the Ministerio de Econom\'ia y Competitividad of
Spain and project US-1262169, funded by the Fondo Europeo de Desarrollo Regional (FEDER) and the Consejer\'ia de Econom\'ia, Conocimiento, Empresas y Universidad de la Junta de Andaluc\'ia.

\section{An axiomatic framework}\label{Section:Axioms}
Let us begin by introducing some notation. By $\CG$ we denote a smooth group scheme over $\WF$ whose connected component $\CG^o$ is reductive over $W(\BF)$, and such that $\CG/\CG^o$ is a constant group of order prime to $p$, see \ref{App-0}, \ref{App-3}, \ref{App-4} and \ref{App-5} in the Appendix. We write $\CG^\der$ for the commutator subgroup $[\CG^o,\CG^o]$, see \ref{App-6}. The group $\CG^\der$ is a characteristic subgroup of $\CG$ because $\CG^o$ is characteristic in $\CG$ and $\CG^\der$ in $\CG^o$, and in particular it is normal in $\CG$. It is semisimple over $\WF$ and $\CG^o/\CG^\der$ is a torus, see \ref{App-6}. Moreover $\CG/\CG^\der$ exists as a smooth group scheme over $\WF$, and it is an extension of the constant group $\CG/\CG^\der$ by the torus $\CG^o/\CG^\der$, see \ref{App-8}. We write $\CG_\BF$, $\CG^o_\BF$, $\CG^\der_\BF$ for the special fibers of $\CG$, $\CG^o$ and $\CG^\der$, respectively, and note that $(\CG_\BF)^o=\CG^o_\BF$, that $\CG^o_\BF$ is reductive and that $\CG^\der_\BF$ is semisimple; see \ref{App-3} and \ref{App-9}. We denote the Lie algebras of $\CG_\BF$ and $\CG^\der_\BF$ by $\Fg$ and $\Fg^\der$, respectively; cf.~\ref{App-Lie}. Via the adjoint~representation $\Fg$ and $\Fg^\der$ carry an action of $\CG(\BF)$ and $\CG^\der(\BF)$, respectively. Note also that $\Fg=\Lie(\CG^o_\BF)$.

\smallskip

Throughout this article, we fix a pair $(H_\BF,H'_\BF)$ consisting of a subgroup $H_\BF\subset\CG(\BF)$ and a normal subgroup $H'_\BF$ of $H_\BF$ such that $H'_\BF\subset \CG^\der(\BF)$. We make the following
\begin{Ass}[Standard hypothesis]\label{StandardHyp}
The tuple $(\CG,\BF,H_\BF,H'_\BF)$\footnote{We often simply refer to the pair $(H_\BF,H'_\BF)$, the group $\CG$ and the field $\BF$ being implicitly understood.} satisfies
\begin{enumerate}
\item $H_\BF$ surjects onto $(\CG/\CG^o)(\BF)$ and $H_\BF/[H_\BF,H_\BF]$ is of order prime to $p$,
\item there exists a subgroup $M_\BF$ of $H_\BF$ of order prime to $p$ such that $M_\BF H'_\BF=H_\BF$.
\end{enumerate}
\end{Ass}
\begin{Rem}
Note that if $H'_\BF$ is contained in $[H_\BF,H_\BF]$, for instance if $H'_\BF$ is perfect, then the quotient $H_\BF/[H_\BF,H_\BF]$ being of order prime to $p$ is implied by~(b): By hypothesis, $H'_\BF$ is normal in $H_\BF$. If in addition $H'_\BF$ is contained in $[H_\BF,H_\BF]$, then we have a surjective homomorphism $H_\BF/H'_\BF\to H_\BF/[H_\BF,H_\BF]$. Now (b) gives the isomorphism $M_\BF/(M_\BF\cap H'_\BF)\cong H_\BF/H_\BF'$, showing that $H_\BF/H_\BF'$ and hence $H_\BF/[H_\BF,H_\BF]$ is of order prime to~$p$, because $M_\BF$ is so.
\end{Rem}

The possible presence of $\CG/\CG^o$ allows for instance that $\CG\cong \CG^o\rtimes\Gal(L/K)$ where $L/K$ is a finite Galois extension of global fields (of order prime to $p$). 
\begin{Def}\label{Def:}
For $A\in\CAW$ we define $H'_A:=\{g\in \CG^\der(A) \mid g\pmod{\Fm_A}\in  H'_\BF\}$.
\end{Def}
Let $M^o_\BF=M_\BF\cap \CG^o(\BF)$. Since $M_\BF$ is of order prime to $p$, by \Cref{Lem-HRExists} for $A\in\CAW$ there exist subgroups $M^o_A\subset\CG(A)$ and $M_A\subset\CG(A)$ that modulo $\Fm_A$ reduce isomorphically to $M^o_\BF$ and $M_\BF$, respectively, they both normalize $H'_A$, and the group $M^o_AH'_A$ is independent of the choice of $M^o_A$. Our hypotheses imply $M^o_A\subset\CG^o(A)$. Throughout this article we impose
\begin{Con}\label{Conv:MWF}
We fix lifts 
\[M^o_{\WF}\subset M_{\WF}\subset \CG(\WF)\] 
of $M^o_\BF\subset M_\BF$, for which $M_{\WF}^?\!\to\! M^?_\BF$ is an isomorphism under reduction for $?\in\{o,\emptyset\}$. For any $A$ in $\CAW$, we define $M^?_A$ as the image of $M^?_{\WF}$ \hbox{under the structure morphism~$\WF\!\to \!A$.}
\end{Con}

We set 
\begin{equation}\label{Def:HA-HpA}
H_A^o:=M_A^o H'_A,\qquad  H_A:=M_A^{\hbox{\phantom{o}}} H_A'.
\end{equation}

\begin{Cond}\label{Cond:Axioms}
For a tuple $(\CG,\BF,H_\BF,H'_\BF)$ satisfying \Cref{StandardHyp}, we formulate the following conditions:
\begin{itemize}
\item[\pf] the group $H'_\BF$ is perfect;
\end{itemize}
\begin{itemize}
\item[\ct] 
the natural inclusion $\Lie Z(\CG^o) \into H^0(H_\BF, \Fg)$ is an isomorphism, and moreover the schematic center of $\CG$ is smooth;
\end{itemize}
\begin{itemize}
\item[\liegen] 
(i) $\Fg^\der$ is perfect, i.e., $[\Fg^\der,\Fg^\der]=\Fg^\der$ for the commutator Lie subalgebra, the center $Z(\Fg^\der)$ of $\Fg^\der$ is trivial, (ii) $\Fg^\der$ is irreducible and non-trivial as an $\BF_p[H'_\BF]$-module, and (iii) the natural map $\BF\to\End_{\BF_p[H'_\BF]}(\Fg^\der)$ is bijective\footnote{For any $\BF[G]$-module $V$ there is a canonical homomorphism $\BF\to\End_{\BF_p[G]}(V)$.};
\item[\lieu] 
(i) as $\BF_p[H'_\BF]$-module, $[\Fg^\der,\Fg^\der]\subset \Fg^\der$ is non-trivial, and one of the Jordan-H\"older factors of $[\Fg^\der,\Fg^\der]$ is not a Jordan-H\"older factor of $\Fg^\der/[\Fg^\der,\Fg^\der]$ and (ii) the natural map $\BF\to\End_{\BF_p[H'_\BF]}(\Fg^\der)$ is bijective;
\item[\liec]
(i) $\Fg^\der$ is perfect and (ii)  the $\BF_p[H'_\BF]$-cosocle $\overline\Fg^\der$ of $\Fg^\der$ is irreducible and $H_\BF$ acts trivially on $\kernel(\Fg^\der\to\overline\Fg^\der)$;
\item[\liecsc]
the cosocle of $\Fg^\der$ does not contain the trivial $H'_\BF$-module $\BF_p$;
\end{itemize}

\begin{itemize}
\item[\van] 
the cohomology $H^1(H'_\BF,\Fg)$ vanishes;
\end{itemize}
\begin{itemize}
\item[\sch] 
the mod $p$ Schur multiplier group $H^2(H'_\BF,\BF_p)$ vanishes;
\end{itemize}
\begin{itemize}
\item[\ns] 
the extension $1\to \Fg^\der \to H'_{W_2(\BF)}\to H'_\BF\to1$ is non-split. 
\end{itemize}
\end{Cond}

\begin{Rem}
It is straightforward to see that the following conditions are equivalent: (i) condition \liecsc, (ii) $\Hom_{\BF_p[H'_\BF]}(\Fg^\der,\BF_p)=0$, (iii) $H_0(H'_\BF,\Fg^\der)=0$.
\end{Rem}
\begin{Rem}\label{Rem:ConseqOfLieGen}
 Note that \liegen(ii) implies \liecsc, that \liegen(i) and (ii) imply \liec, and that \liegen(ii) and (iii) imply \lieu.
\end{Rem}

In \Cref{Section:Chevalley} we discuss in great detail the case when $\CG$ is connected, $\CG^\der$ is absolutely simple, $H_\BF=\CG^\der(\BF)$ and $H'_\BF\supset [H_\BF,H_\BF]$. Because of the Lie-part of \Cref{Cor:CondsForp5} we like to think that \liegen\ describes the `generic' behavior of $\Fg^\der$. As was just observed, \liegen\ implies \lieu\ and \liec. The latter conditions will be a crucial input about the Lie algebra $\Fg^\der$ in our main theorems on certain {\bf un}iversal deformation rings and on {\bf cl}osed subgroups of $\CG(R)$, for $R\in\CAW$, respectively.
 
 \medskip

We shall now state the main technical results of this work and use them to derive \Cref{MainThm}. For this let $R$ be in $\CAW$ and consider the canonical reduction
\[\barrhoR\colon H_R\to\CG^\der(\BF)\subset\CG(\BF)\] 
as a $\CG(\BF)$-valued representation of $H_R$. A $\CG$-deformation of $\barrhoR$ to a ring $A$ (in $\CAW$) is a $\kernel(\CG(A)\to\CG(\BF))$-conjugacy class of continuous homomorphisms $\rho_A\colon H_R\to \CG(A)$ such that $\rho_A\equiv\barrhoR\pmod{\Fm_A}$. As in \cite{Mazur} or \cite{Tilouine} one shows that
\footnote{Both references require that $R$ is noetherian. Hence one may apply them to all Artin quotients of $R$ and then argue using an inverse limit argument.}
\begin{Lem}\label{Lem-UDefExists}
If \ct\ holds, then the functor 
\[D_{\barrhoR}\colon\CAW\to\Sets, A\mapsto \{\hbox{$\CG$-valued deformations of $\barrhoR$ to $A$}\}\]
is pro-representable within $\CAW$. 
\end{Lem}
\begin{Def}\label{Def:UDefRing}
We denote the universal ring representing $D_\barrhoR$ by $R_{\barrhoR}$ and a representative of the universal deformation by $\rho_\barrhoR\colon H_R\to \CG(R_\barrhoR)$.
\end{Def}

The first main technical result of this article is the following:
\begin{Thm}[\Cref{Thm-Thm1}]
Suppose $(\CG,\BF,H_\BF,H'_\BF)$ satisfies conditions \ct, \van, \ns, \lieu\ and one of \liecsc\ or $\CG=\CG^\der$. Then the canonical inclusion $\iota\colon H_R\to \CG(R)$ represents the universal deformation of $D_\barrhoR$, and in particular $R_\barrhoR= R$.
\end{Thm}
In \Cref{Rem:Dorob-EM}, we shall explain of the relation between \Cref{Thm-Thm1} and the results of Dorobisz and Eardly-Manoharmayum.

\medskip

For the second main result, let $H$ denote any closed subgroup of $\CG(R)$ such that the image of $H$ in $\CG(\BF)$ is equal to $H'_\BF$. In \Cref{Lem-OnCommutator} we shall prove that there exists a unique closed subgroup $H^c\subset H$ that contains the closure $\overline{[H,H]}$ of the commutator subgroup and for which $H^c/\overline{[H,H]}\to H_\BF/[H_\BF,H_\BF]$ is an isomorphism. Note that by definition $H^c$ is a closed subgroup of $H$ that surjects onto $H_\BF$. We shall also show that $H\mapsto H^c$ commutes with passing from $R$ to a quotient.

We define $H^{(0)}:=H$ and, inductively, for any $i\ge1$  we define $H^{(i)}:=(H^{(i-1)})^c$. The $H^{(i)}$ define a descending sequence of closed subgroups of~$H$.
\begin{Def} \label{DefHPerfection}
The group $H$ is called {\em $H_\BF$-perfect} if $H=H^c$. 

For any $H$, its {\em $H_\BF$-perfection} is defined as $H^{(\infty)}:=\bigcap_i H^{(i)}$.
\end{Def}

Note that for any artinian quotient of $R$ the procedure defining $H^{(\infty)}$ stops after finitely many steps. From $R$ being a filtered inverse limit of such rings and from \Cref{Lem-OnCommutator}, it follows that $H^{(\infty)}$ is $H_\BF$-perfect. We call a closed subgroup $H$ of $H_R$ {\em residually full} if under the reduction map $H$ surjects onto~$H_\BF$.

The second main technical result of this article is:
\begin{Thm}[\Cref{Thm-Thm2}]
Let $H\subset H_R$ be a closed subgroup that is residually full. Suppose that \ct, \ns\ and \van\ hold, and that either \liegen\ holds or that \liec\ and \sch\ hold. Then there exists a closed $\WF$-subalgebra $A$ of $R$ such that $H^{(\infty)}$ is conjugate to $H_A\subset\CG(R)$.
\end{Thm}

In \Cref{Rem:RelToMano} we shall explain its relation to a result by Manoharmayum.
\begin{proof}[{Proof of \Cref{MainThm}}]
In \Cref{Cor:CondsForp5} we show that under the hypotheses of \Cref{MainThm} conditions \pf, \ct, \liegen, \sch, \ns\ and \van\ hold. By \Cref{Rem:ConseqOfLieGen}, also condition \lieu\ holds. Hence  \Cref{MainThm}(a) and (c) are immediate from \Cref{Thm-Thm1} and \Cref{Thm-Thm2}. Using the universal deformation property of $\id\colon \CG(R)\to\CG(R)$, the proof of (b) follows from (a), (c) and \Cref{Lem-GpsWithHc}.
\end{proof}

\medskip

We end this section with some remarks on our standard hypotheses from \Cref{StandardHyp} in relation to speculations about compatible systems of Galois representations attached to pure motives with coefficients; cf.\ the foundational reference 
\cite{Serre-Motives} and also \cite{Hui}.

Let $G$ be an affine group scheme over a number field $F$ such that the identity component $G^o$ of $G$ is reductive. Denote by $P_F$ the set of finite places of $F$, by $F_\lambda$ the completion of $F$ at $\lambda\in P_F$, by $\CO_\lambda$ its ring of integers and by $\BF_\lambda$ its residue field. Let $k$ be a number field with absolute Galois group $\Gamma_k$; we write $\Frob_v$ for a Frobenius automorphism of the finite place $v\in P_k$. We write $\ell_v$ and $\ell_\lambda$ for the rational prime below $v$ and $\lambda$, respectively, we set $S_\lambda:=\{v\in P_k\mid \ell_v=\ell_\lambda\}$. Then an $F$-rational $G$-compatible system $\rho_\bullet$ of $\Gamma_k$ consists of a finite set $S$ of places of $P_k$, for each $\lambda\in P_F$ a continuous representation $\rho_\lambda\colon\Gamma_k\to G(F_\lambda)$ that is unramified outside $S\cup S_\lambda$, and for each $v\in P_k\setminus S$ a semisimple $G(\overline \BQ)$-conjugacy class $t_v$ in $G(\overline\BQ)$ such that for all $\lambda\in P_F$ and $v\in P_k\setminus (S\cup S_\lambda)$, the semisimplification of $\rho_\lambda(\Frob_v)$ is conjugate to $t_v$ in $G(\overline F_\lambda)$. 

For any smooth projective variety $X$ over $k$ and fixed $i\ge0$, the $i$-th \'etale $\ell$-adic cohomologies over all primes $\ell$ form a $\BQ$-rational $\GL_n$-compatible system with $n=\dim H^i_\et(X,\BQ_\ell)$. If in addition one has a projector in the sense of Grothendieck motives for the $i$-th \'etale cohomologies of $X$, say defined over a number field $F$, then multiplying with the projector defines an $F$-rational compatible system. Similarly, it is expected that to a $G'$-valued cuspidal automorphic representation with Hecke field $F'$ and Langlands dual $G$ of $G'$ one can attach a an $F'$-rational compatible system over a finite extension $F$ of~$F'$. 

Given a $G$-compatible system, a first expectation is that there exists a reductive subgroup $H$ of $G$ such that, up to conjugation for each $\lambda\in P_F$, the group $\rho_\lambda(\Gamma_k)$ is Zariski dense in $H(F_\lambda)$. So let us assume from now on that $G$ is chosen so that $\rho_\lambda(\Gamma_k)$ is Zariski dense in $G(F_\lambda)$ for all $\lambda$. The group $G$ should then be the motivic Galois group of the $F$-rational compatible system. There is a finite Galois extension $k'$ of $k$ such that $\Gal(k'/k)$ is isomorphic to $G(F_\lambda)/G^o(F_\lambda)$, and a result of Serre on compatible systems says that $k'$ is independent of~$\lambda$. 

By forming the quotient modulo the center of $G^o$, let us next assume that $G^o$ is semisimple of adjoint type. The system remains compatible. Because $G^o$ is of adjoint type, by \cite{Pink-Compact} the compatible system $(\rho_\lambda\colon \Gamma_{k'}\to G^o(F_\lambda))$ thus obtained should arise from an $F_{\tr}$-rational $G'$-compatible system where $F_{\tr}\subset F$ is the field of traces of its adjoint representation and $G=G'\times_{F_\tr}F$. We now make two hypotheses for the remaining discussion:
\begin{enumerate}
\item The group $G^o$ is absolutely simple.
\item The $G$-compatible system, and not only its restriction to $\Gamma_{k'}$, is $F_\tr$-rational.
\end{enumerate}
Condition (a) is an intrinsic condition on the motive giving rise to $(\rho_\lambda)_\lambda$. We assume it in order to fit our context. Concerning (b), observe that the result of Pink guarantees that for each $\lambda\in P_F$ and $\lambda'\in P_{F_\tr}$ below $\lambda$, there is connected reductive group $H_{\lambda'}$ defined over $(F_\tr)_{\lambda'}$ whose base change to $F_\lambda$ is $G^o\times_FF_\lambda$. Pink's result does not guarantee that the groups come from a global group $H$ defined over $F_\tr$. This should be expected and this is one of our requirements in (b). The other requirement is that there is an extension $G'$  of a finite group by $H$ defined over $F_\tr$, such that $G=G'\times_{F_\tr}F$. We want to alert the reader at this point, that $G^o$ is simply connected and of adjoint type, and hence inner-twist like phenomena do not occur.
We think that it is an interesting question to ask if (b) can always be expected, or if there are natural sufficient conditions for it to hold. Jointly with A. Conti, the first author plans to explore this further in some particular cases.

The group $G$ has an integral model $\CG$ over an open non-empty subscheme $U$ of $\Spec \CO_F$. For places $\lambda$ in $|U|\subset P_F$ the group $\CG^o(\CO_\lambda)$ is maximal hyperspecial. An expectation that one has in this context, cf.~\cite{Larsen}, is that  for almost all $\lambda\in |U|$ the condition $(\hyp_\lambda)$ holds: the subgroup $\rho_\lambda(\Gamma_{k'})$ of $G^o(F_\lambda)$ itself is maximal hyperspecial, and hence conjugate to $\CG^o(\CO_\lambda)$. The following result is now a direct consequence of the above discussion and \Cref{Cor:CondsForp5}.

\begin{Prop}\label{Prop:StdHypForCS}
Let $\rho_\bullet$ be an $F$-rational $G$-compatible system of representations of $\Gamma_k$ with $G^o$ absolutely simple and of adjoint type. Suppose condition $(\hyp_\lambda)$ holds for all but finitely many $\lambda\in P_F$.\footnote{In particular, $\rho_\lambda(\Gamma_k)$ is open in $G(F_\lambda)$ and Zariski-dense in $G\times_FF_\lambda$ for all these $\lambda$.} Set $H_\lambda:=\bar\rho_\lambda(\Gamma_k)$ and $H'_\lambda:=H_\lambda\cap \CG^o(\BF_\lambda)$. Then for all but finitely many $\lambda\in P_F$ the tuple $(\CG_{W(\BF_\lambda)}, \BF_\lambda,H_\lambda,H'_\lambda)$ satisfies \Cref{StandardHyp} and conditions \pf, \ct, \liegen, \van, \sch, \ns.
\end{Prop}

\section{Discussion of basic hypotheses for Chevalley groups}
\label{Section:Chevalley}

In this section, until the end of \Cref{Subsection:Ns}, we fix the following set-up.
\begin{Cond}\label{Cond:StdSetup}
\begin{enumerate}\advance\itemsep by -.4em
\item $\CG$ is a connected absolutely simple linear algebraic group~over~$\BF$,
\item $H'_\BF$ is the image of $\CG^\sco(\BF)$ under $\CG^\sco(\BF)\to\CG(\BF)$, and  
\item $H_\BF$ is any subgroup of $\CG(\BF)$ that contains $H'_\BF$. 
\end{enumerate}
\end{Cond}
We shall investigate the validity of \Cref{StandardHyp} and the conditions formulated in \Cref{Cond:Axioms} for the tuple $(\CG,\BF,H_\BF,H'_\BF)$.  We rely on well-known results from the literature. To get them into the precise shape needed, we often need to prove auxiliary results. We shall investigate our basic hypotheses separately, each in its own subsection. In \Cref{Subsec:Puniv} we shall consider a slight variation of this basic~setup. It will be useful to have the following lists that will be used to describe some exceptional behavior:
\begin{equation}\label{eq:Pf-List}
\CE_{\pf}:=\{\SL_2(\BF_2), \ \SL_2(\BF_3), \ \SU_3(\BF_2), \ \Sp_4(\BF_2), \ G_2(\BF_2)\}
\end{equation} 
\begin{equation}\label{eq:Schur-List}
\CE_{\sch}\!:=\!\!\left\{\!\!\!
\renewcommand{\arraystretch}{1.3}
\begin{array}{c}
(A_1,\BF_?)_{?\in\{4,9\}},\nhs (A_2,\BF_?)_{?\in\{2,4\}},\nhs (A_3,\BF_2),(B_2,\BF_2), \nhs(B_3,\BF_?)_{?\in\{2,3\}}, \nhs(C_3,\BF_2) , \\
(D_4,\BF_2), (F_4,\BF_2), (G_2,\BF_?)_{?\in\{3,4\}}, ({}^2A_3,\BF_?)_{?\in\{2,3\}}, ({}^2A_5,\BF_2), ({}^2E_6,\BF_2)
\end{array}\!\!\!\!
\right\}\!\!
\end{equation} 
\begin{equation}
\label{eq:n-s-List}
\CE_{\ns}:=\left\{ \!\!\!
\renewcommand{\arraystretch}{1.3}
\begin{array}{c}
(\SL_2,\BF_?)_{?\in \{2,3\}}, (\PGL_2,\BF_?)_{?\in\{2,3,4\}},  (\SL_3,\BF_2), (\PGL_3,\BF_2), \\
(\SU_3,\BF_2), (\PGU_3,\BF_2), (\PGU_4,\BF_2), (\SO_6,\BF_2)
\end{array}\!\!\!
\right\}
\end{equation} 
The following result combines the results from Subsections~\ref{Subsection:Pf} to~\ref{Subsection:Ns}.
\begin{Thm} \label{Thm:MainOnStdSetup}
Suppose that $(\CG,\BF,H_\BF,H'_\BF)$ satisfies \Cref{Cond:StdSetup}
\begin{enumerate} 
\item
If $\CG^\sco(\BF)$  is not in $\CE_{\pf}$, then \pf\ holds and $(\CG,\BF,H_\BF,H'_\BF)$ satisfies \Cref{StandardHyp}.
\item
If $($type$,\BF)\notin \CE_{\sch}$, then \sch\ holds.
\item
Suppose that $\CG^\sco(\BF)\notin\CE_{\pf}$. \hbox{Then $(\CG, \BF, H_{\BF}, H'_{\BF})$~satisfies:}
\begin{enumerate}
 
\item $\liegen$ if $\CG$ is of type $A_n$ and $p\nmid n+1$, or if $\CG$ is of type $B_n$,  $C_n$, $D_n$, $E_7$ or $F_4$ and $p\not=2$, or if $\CG$ is of type $E_6$ or $G_2$ and $p\not=3$, or if $\CG$ is of type $E_8$.

\item $\liec$ if $\CG$ is Lie-simply connected of type $A_n$, $n\ge2$, $D_n$, $E_n$, or $\CG$ is of type $A_1$, $B_n$,  $C_n$, $F_4$ and $p\neq2$, or of  type $G_2$ and $p\neq3$.\footnote{For types $A_n$, $D_n$, $E_6$, $E_7$, this gives restrictions only if $p|n+1$, $p=2$, $p=3$, $p=2$, respectively.}

\item $\lieu$ unless $\CG$ is Lie-intermediate of type $A_n$ with $p|n+1$ or $D_n$ with $n$ odd and $p=2$, or $\CG$ is of type $B_2$ or $F_4$ and $p=2$, or of type $G_2$ and~$p=3$.

\end{enumerate}
\item If $\CG$ is Lie-simply connected, then \liecsc\ holds.
\item  \ct\ holds $\Longleftrightarrow$  $Z(\Fg)=0$ $\Longleftrightarrow$ $\CG$ is of Lie-adjoint type.
\item
Condition \ns\ holds if and only if $(\CG,\BF)$ is not $\CE_{\ns}$.
\item Condition \van\ holds if $\CG^\sco(\BF)\notin\CE_{\!\pf}$, $($type$,\BF)\notin\CE_{\!\sch}\cup\{(A_1,\BF_5)\}$, \ct\ holds, and if further one of the following holds:
\begin{enumerate}
\item
If type$\ =C_n$, then $|\BF|\notin\{2,3,4,5,9\}$.
\item
If $\CG$ is non-split (and hence of type A, D or E${}_6$), then $|\BF|\ge4$.
\end{enumerate}
\end{enumerate} 
 
\end{Thm}

\begin{Rem}
Observe that \ct\ and \liec\ can hold simultaneously only if \liegen(ii) holds; and in the present situation the latter implies \liegen.
\end{Rem}

\begin{Not}\label{Not:ExcSet}
For the conjunction of $\CG^\sco(\BF)\notin\CE_{\pf}$, $($type $\CG,\BF)\notin \CE_{\sch}$ and $(\CG,\BF)\notin \CE_{\ns}$, we shall simply write $(\CG,\BF)\notin\CE$, and say that $(\CG,\BF)$ is {\em not exceptional}.
\end{Not}
We state an immediate consequence of \Cref{Thm:MainOnStdSetup}.
\begin{Cor}\label{Cor:CondsForp5}
Suppose that $p\ge5$,  that $p\notdiv n+1$ if $\CG$ is of type $A_n$ and that $\BF\neq\BF_5$ if $\CG$ is of type $A_1$ or $C_n$. Then \pf, \sch, \liegen, \ct, \ns\ and \van\ hold.
\end{Cor}

\medskip

We fix the following notation throughout this section for an absolutely simple group $\CG$ over $\BF$. By $\phi^\sco\colon\CG^\sco\to\CG$ and $\phi^\ad\colon\CG\to\CG^\ad$ we denote the central isogenies from the simply connected cover of $\CG$ and to its adjoint group, respectively. We define $\phi:=\phi^\ad\circ\phi^\sco$ and write $\CZ:=\kernel\phi$ for the center of $\CG^\sco$. We set $\CZ':=\kernel \phi^\sco$ and $\CZ'':=\kernel \phi^\ad$, so that there is a short exact sequence $1\to\CZ'\to\CZ\to\CZ''\to1$. Cf.~\ref{App-15}.

We remind the reader of the well-known structure of $\CZ$ given in \Cref{Eqn:Centers}, see \cite[Table~9.2 and Table~24.2]{MalleTesterman}.
\begin{table}
\[
\begin{tabular}{c|c|c|c|c|c|c|c|c|c|c|c|c|} 
type&$A_n$&$B_n$&$C_n$&$D_n$, ($n$ odd)&$D_n$, ($n$ even)
&$E_6$&$E_7$&$E_8$&$F_4$&$G_2$\\
\hline
$\CZ$&$\mu_{n+1}$&$\mu_2$&$\mu_2$&$\mu_4$&$\mu_2\times\mu_2$
&$\mu_3$&$\mu_2$&$\mu_1$&$\mu_1$&$\mu_1$\\
\end{tabular}
\]
\caption{Centers of absolutely simple simply connected $\CG$}
\label{Eqn:Centers}
\end{table}
Here $\mu_n$ is the finite flat group scheme that is the kernel of $\CG_m\to\CG_m,\alpha\mapsto \alpha^n$. In par\-ticular $\CZ$ is \'etale over $\BF$ and $\Lie\CZ=0$ if and only if $p$ does not divide the order~of~$\CZ$. We also recall the following fact for the convenience of the reader, since it is used repeatedly.
\begin{Fact}[{\cite[10.14]{MilneAGS}}]
\label{LieAndKernel} Let $u:\CG\rightarrow \CG'$ be a morphism of affine algebraic groups over a field. Then $\Lie(\ker u)=\ker(\Lie u)$ by \cite[10.14]{MilneAGS}.
\end{Fact}

The homomorphisms on Lie algebras induced from the above morphisms of algebraic groups are denoted $\mathrm{d}\phi^\sco\colon\Fg^\sco\to\Fg$, $\mathrm{d}\phi^\ad\colon\Fg\to\Fg^\ad$ and $\mathrm{d}\phi=\mathrm{d}\phi^\ad\circ\mathrm{d}\phi^\sco$. We set $\Fz:=\kernel \mathrm{d}\phi=\Lie(\CZ)$
and $\Fz^*:=\coker\mathrm{d}\phi$, so that $\dim\Fz=\dim\Fz^*$; see \cite[Prop.~1.11(a)]{Pink-Compact}. We also define $\Fz':=\kernel \mathrm{d}\phi^\sco$ and $\Fz'':=\kernel \mathrm{d}\phi^\ad$, so that $\Fz'=\Lie\CZ'$ and $\Fz''=\Lie\CZ''$. Note that \cite{Pink-Compact} usually requires that $\CG$ be adjoint, so that $\Fz'=\Fz$. The action of $\CG^?$ on $\Fg^?$ for $?\in\{\sco,\emptyset,\ad\}$ is via the adjoint action, and hence it factors via the canonical map to $\CG^\ad$. In particular, all Lie algebras and homomorphisms between them that we just defined are modules under any of the algebraic groups $\CG^?$, and thus also modules for $H_\BF$. For any Lie algebra $\Fh$ over $\BF$, we write $\Fh_{\overline\BF}$ for $\Fh\otimes_\BF\overline\BF$ where $\overline\BF$ is an algebraic closure of $\BF$. We note that if $\Fh$ is any of the Lie algebras above, then $\Fh_{\overline\BF}$ carries a representation of $\CG^\ad(\overline\BF)$. Adapting the notation of \cite{Hogeweij} to our needs, we call $\CG$ {\em Lie-simply connected} if $\mathrm{d}\phi^\sco$ is an isomorphism, {\em Lie-adjoint} if $\mathrm{d}\phi^\ad$ is an isomorphism, and {\em Lie-intermediate} otherwise.

Following \cite{Pink-Compact}, in our notation (!), we let $\overline\Fg$ be the image of $\Fg^\sco$ under $\mathrm{d}\phi$, and we write $\mathrm{d}\phi\colon\Fg^\sco\to\overline\Fg$ and $\mathrm{incl}\colon\overline\Fg\to \Fg^\ad$ for the induced homomorphisms, of Lie algebras and of $\CG^\ad$-representa\-tions. By \cite[Prop.~1.10]{Pink-Compact} there is a Lie algebra and $\CG^\ad$-module $\hat\Fg$ such that one has a pushout as well as a pullback diagram of Lie algebras and $\CG^\ad$-representations
\[\xymatrix{\hat \Fg\ar@{->>}[r]&\Fg^\ad\\
\Fg^\sco\ar@{^{ (}->}[u]\ar@{->>}[r]^{\mathrm{d}\phi}&\overline\Fg\ar@{^{ (}->}[u]_{\mathrm{incl}}\rlap{.}\\
}\]
We will see in \Cref{Subsec:Puniv} that in fact $\hat\Fg$ occurs as the Lie algebra of a certain reductive group constructed from $\CG^\ad$ (or $\CG^\sco$). The usefulness of $\hat\Fg$ can also be seen in Subsections~\ref{Subsection:Lie} and~\ref{Subsection:Van}. 

\subsection{Condition \pf}
\label{Subsection:Pf}

Note again that until the end of \Cref{Subsection:Ns} we assume \Cref{Cond:StdSetup}. 
\begin{Thm}[{Tits, see~\cite[Thm.~24.17]{MalleTesterman}}]
\label{Thm:Pf}
If $\CG^\sco(\BF)$ is not in the list $\CE_{\pf}$ from~\eqref{eq:Pf-List}, then the group $\CG^\sco(\BF)$ is perfect, and its quotient $\CG^\sco(\BF)/Z(\CG^\sco(\BF))$ is simple.
\end{Thm}

\begin{Cor}\label{Cor:Pf}
Write $Z$ for $\CZ'(\BF)$. Then the following hold:
\begin{enumerate}
\item The map $\phi^\sco$ induces short exact sequences $1\to Z\to \CG^\sco(\BF)\to H'_\BF\to 1$ and $1\to H'_\BF\to \CG(\BF)\to Z\to 1$.
\item The group $Z$ is finite abelian of order prime to $p$.
\item Suppose that $\CG^\sco(\BF)$  is not in $\CE_{\pf}$, then $H'_\BF$ is perfect and $H'_\BF=[\CG(\BF),\CG(\BF)]$.
\end{enumerate}
\end{Cor}
\begin{proof}
The argument seems to be well-known, but we could not find a complete reference, so we give some details: The map $\phi^\sco$ induces the short exact sequence $1\to \CZ'\to \CG^\sco\to\CG\to 1$ of group schemes. Applying (non-commutative) flat cohomology yields the $5$-term left exact sequence of pointed sets
\[ 0\to H^0_\flatt(\BF,\CZ')\to H^0_\flatt(\BF,\CG^\sco)\to H^0_\flatt(\BF,\CG)\to H^1_\flatt(\BF,\CZ')\to H^1_\flatt(\BF,\CG^\sco);\]
in fact because $\CZ'$ is central, by \cite[Prop.~3.4.3]{Giraud} the first three non-trivial arrows are group homomorphisms. Moreover for smooth group schemes flat and \'etale cohomology in degrees $0$ and $1$ coincide; for both see~\cite[III.4.5, III.4.7, III.4.8]{MilneEtale}. This yields
\[ 0\to  H^0_\flatt(\BF,\CZ')\to \CG^\sco(\BF)\stackrel{\phi^\sco}{\to} \CG(\BF)\to H^1_\flatt(\BF,\CZ')\to H^1_\et(\BF,\CG^\sco).\]
By a result of Steinberg, we have $H^1_\et(\BF,\CG^\sco)=1$; see \cite[Thm.~1.9]{Steinberg-RegularElements}. The sequence also implies that $H^0_\flatt(\BF,\CZ')\cong \kernel(\phi^\sco)=Z$. Note next that the group scheme  $\CZ'$ 
is a product of group schemes $\mu_n$, and for the latter ones flat cohomology in degrees $0$ and $1$ can be computed via $1\to \mu_n\to \BG_m\to\BG_m\to1$ by the same arguments that were used above, now using Hilbert 90, from the $4$-term exact sequence of groups 
\[1\longto H^0_\flatt(\BF,\mu_n)\longto \BF^\times \stackrel{\alpha\mapsto \alpha^n}\longto \BF^\times \longto H^1_\flatt(\BF,\mu_n)\longto 1.\]
 Because $\BF^\times$ is finite cyclic of order prime to $p$, it follows that $H^0_\flatt(\BF,\mu_n)\cong H^1_\flatt(\BF,\mu_n)$ is finite cyclic of order prime to $p$. Hence $H^1_\flatt(\BF,\CZ')\cong H^0_\flatt(\BF,\CZ')\cong Z$, and $Z$ is finite abelian of order prime to $p$, and this completes the proof of (a) and~(b). 

To prove (c), note first that by \Cref{Thm:Pf} the group $\CG^\sco(\BF)$ is perfect, and hence so is $\image(\phi^\sco)=H'_\BF$. It remains to show that $H'_\BF$ is equal to $[\CG(\BF),\CG(\BF)]$. Since the cokernel of $\phi^\sco$ is abelian, the group $H'_\BF$ contains $[\CG(\BF),\CG(\BF)]$. But $H'_\BF$ cannot be larger, since as a perfect group it has no non-trivial abelian quotients.
\end{proof}

\begin{Rem}
The group $Z$ is cyclic unless $\CG$ is of type $D_n$ with $n$ even, $p>2$ and $\phi^\ad$ is an isomorphism in which case $Z$ is isomorphic to $\BZ/2\times\BZ/2$.
\end{Rem}

\begin{Cor}\label{Cor:StdHyp}
Suppose that $\CG^\sco(\BF)$  is not in $\CE_{\pf}$. Then the pair $(H_\BF,H'_\BF)$ satisfies \Cref{StandardHyp}.
\end{Cor}
\begin{proof}
By \Cref{Cor:Pf} the group $H'_\BF$ is perfect. Therefore it suffices to find a subgroup $M_\BF$ of $H_\BF$ that is of order prime to $p$ and such that $M_\BF$ surjects onto $H_\BF/H'_\BF$. Since $\CG(\BF)/H'_\BF$ is finite abelian, it suffices to construct $M_\BF$ in the case where $H_\BF=\CG(\BF)$: if $M_\BF H'_\BF=\CG(\BF)$, then the product of the kernel of $M_\BF\to \CG(\BF)/H_\BF$ with $H'_\BF$ will be equal to $H_\BF$ for any $H_\BF$ with $H'_\BF\subset H_\BF\subset \CG(\BF)$.

Let $\CT\subset\CG$ be a maximal torus. Then by \ref{App-20} its inverse image $\CT^\sco$ is a maximal torus in $\CG^\sco$ and one has a short exact sequence $1\to\CZ'\to\CT^\sco\to\CT\to1$. Arguing as in the proof of \Cref{Cor:Pf}, we obtain a $4$-term exact sequence
\[1\longto \CZ'(\BF) \longto \CT^\sco(\BF) \stackrel{\phi^\sco}\longto \CT(\BF) \longto \CT(\BF)/ \CT^\sco(\BF) \longto 1.\]
It follows that $\CT(\BF)/ \CT^\sco(\BF)$ has the same cardinality as $ \CZ'(\BF)$, and by \Cref{Cor:Pf}, as $\CG(\BF)/H'_\BF$. Because $\CT^\sco$ is the fiber product of $\CT$ with $\CG^\sco$ over $\CG$, the resulting square of $\BF$-points and elementary group theory yield a natural inclusion
\[\CT(\BF)/ \phi^\sco(\CT^\sco(\BF))\into \CG(\BF)/\phi^\sco(\CG^\sco(\BF))\stackrel{\ref{Cor:Pf}(a)}=\CG(\BF)/H'_\BF.\] By our consideration on cardinalities, it must be an isomorphism. Therefore we can take $M_\BF:=\CT(\BF)$ which is clearly of order prime to~$p$.
\end{proof}
 \begin{Rem}
The groups in the list $\CE_{\pf}$ are discussed in \cite[Rem.~24.18]{MalleTesterman}. Analyzing them in more detail, using \cite[I.3.5]{Atlas}, one can verify that \Cref{StandardHyp} is satisfied for the pairs $(H'_\BF,H'_\BF)$  (any type) and the pair $(G_2(\BF_3),H_\BF)$, but not for any other pair $(H_\BF,H'_\BF)$ with $\CG(\BF)\in\CE_{\pf}$ exceptional. We leave the details to the reader.
\end{Rem}
For later use, we also note the following immediate consequence of \Cref{Cor:Pf}.
\begin{Cor}
\label{Cor:H1Triv}
For $\CG^\sco(\BF)\notin \CE_{\pf}$ one has $H^1(\CG(\BF),\BF_p)=\Hom(\CG,\BF_p)=0$.
\end{Cor}

\subsection{Condition \sch}
\label{Subsection:Sch}

\begin{Thm}[{{\cite[Thm.~1.1]{Steinberg}}}]\label{Thm:Schur}
Let $\CG$ be as in \Cref{Cond:StdSetup}. Then the mod $p$ Schur multiplier $H^2(\CG^\sco(\BF),\BF_p)$ vanishes, unless $($type$,\BF)$ is in the list $\CE_{\sch}$ from~\eqref{eq:Schur-List}.
\end{Thm}

\begin{Rem}\label{Rem:SchurZ}
In fact, the result stated in \cite[Thm.~1.1]{Steinberg} is slightly different. It asserts that the Schur multiplier $H_2(\CG(\BF),\BZ)$ vanishes whenever $($type$,\BF)\notin \CE_{\sch}$, and that for $($type$,\BF)\in \CE_{\sch}$ the group $H_2(\CG(\BF),\BZ)$ is finite and of order a power of~$p$. The relation to \Cref{Thm:Schur} is given by the universal coefficient theorem. It gives the short exact sequence
\[ 0\to \Ext^1(H_1(\CG(\BF),\BZ),A)\to H^2(\CG(\BF),A)\to \Hom(H_2(\CG(\BF),\BZ),A)\to 0.\]
Because $\CG(\BF)$ is finite and $\BZ$ is torsion free, the left hand term is zero, and $A=\BZ/(p)$ gives the above theorem. Another consequence is that all finite central extensions of $\CG(\BF)$ of order prime to $p$ are trivial. 
\end{Rem}

\begin{Rem}\label{Rem:Schur}
\begin{enumerate}
\item  
In the statement of \cite[Theorem 1.1]{Steinberg} there is a typo: the group $A_2(3)$ should be $A_3(2)$ (see (5), $\S 2.3$ of \cite{Steinberg}). Moreover, to obtain the above list from Theorem 1.1 in \cite{Steinberg}, one should take into account that $B_2(2)$ is isomorphic to the symmetric group $S_6$, which has nontrivial mod $p$ Schur multiplier (cf.~\cite[(6) in $\S3.3$]{Steinberg}), as well as the exceptional isomorphism $B_3(2)\simeq C_3(2)$.
\item 
One can find the same list of exceptions in \cite[Table 1]{Griess}. The groups considered in \cite{Griess} are simple groups of Lie type of the form $H'_{\BF}=[\CG(\BF), \CG(\BF)]$, where $\CG$ is of adjoint type.
\end{enumerate}
\end{Rem}

\begin{Cor}\label{Cor:Schur}
Suppose that $($type$,\BF)\notin \CE_{\sch}$. Then \sch\ holds.
\end{Cor}

\begin{proof}
In the simply connected case this result is \Cref{Thm:Schur}, taking into account that $H'_{\BF}=[\CG(\BF), \CG(\BF)]=\CG(\BF)$. In the general case, one applies \Cref{Cor:Pf}(a), (b) to control kernel and cokernel of $\CG^\sco(\BF)\to\CG(\BF)$, and then concludes using the Hochschild-Serre spectral sequence from group cohomology.
\end{proof}

\subsection{Condition \ct}
\label{Subsection:Ct}
We begin with the following result due to Hogeweij on the center:

\begin{Thm}[Hogeweij]\label{Thm:LieCenter}
One has
\[\kernel(\mathrm{d}\phi^\ad\colon\Fg\to\Fg^\ad) \stackrel{\mathrm{Fact}\,\ref{LieAndKernel}}=\Lie Z(\CG)= Z(\Fg).\]
Because $Z(\CG)$ is finite, $Z(\Fg)$ is non-zero if and only if $p$ divides the order of $Z(\CG)$, i.e., if and only if $\CG$ is Lie-adjoint.

More concretely, from \Cref{Eqn:Centers} one has $\Fz=\kernel(\mathrm{d}\phi\colon\Fg^\sco\to\Fg^\ad) \neq0$ if and only if one of the following conditions is satisfied:
\begin{enumerate}
\item $\Fg$ is of type $A_n$ and $p|(n+1)$.
\item $\Fg$ is of type $B_n$, $C_n$ or $D_n$ ($n$ odd) or $E_7$ and $p=2$.
\item $\Fg$ is of type $E_6$, and $p=3$.
\item $\Fg$ is of type $D_n$ ($n$ even), and $p=2$.
\end{enumerate}
In cases (a)--(c) one has $\dim \Fz=1$, in case (d) one has $\dim \Fz =2$. 

For general $\Fg$ one has $Z(\Fg)\neq0$ if and only if $\Fz\neq0$ and either $\mathrm{d}\phi^\sco\colon\Fg^\sco\!\to\Fg$~is bijective in cases (a)--(c), or $\kernel(\mathrm{d}\phi^\sco\colon\Fg^\sco\!\to\Fg)$ has dimension at most $1$ in~case~(d).

Finally, $Z(\Fg)$ is an $\BF[H_\BF]$-submodule of $\Fg$ on which $H_\BF$ acts trivially.
\end{Thm}

\begin{proof}
It is a basic fact that $\Lie Z(\CG)$ is a sub Lie algebra of $Z(\Fg)$; see \cite[Prop.~10.33]{MilneAGS}. To prove the displayed formula, note that $\dim_\BF\Lie \mu_m=1$ if $p|m$, and $\dim_\BF\Lie \mu_m=0$, otherwise. Hence one can read off from \Cref{Eqn:Centers} the dimension of $Z(\CG^\sco)$. 

Next note that $\phi^\sco\mapsto \kernel\phi^\sco$ defines a bijection between the (central) isogenies $\phi^\sco \colon  \CG^\sco\to\CG$ and the subgroups of $Z(\CG^\sco)$. Moreover $Z(\CG)$ is equal to $Z(\CG^\sco)$ modulo $\kernel\phi^\sco$. This allows one to read off $\dim\Lie Z(\CG)$ from \Cref{Eqn:Centers} for all $\CG$ (as a function of $p$ and $\phi^\sco$). One now compares this dimension with that of $Z(\Fg)$ given in \cite[Table 1]{Hogeweij} and observes equality in all cases, so that the stated formula holds. The remaining assertions on $Z(\Fg)$ follow immediately from \cite[Table 1]{Hogeweij}. 

The assertion on the $\BF[H_\BF]$-module structure is straightforward from $Z(\Fg)=\Lie Z(\CG)$: $Z(\CG)\subset\CG$ is a stable subgroup under the adjoint action, and the latter is trivial on $Z(\CG)$. Passing to Lie algebras, $Z(\Fg)\subset \Fg$ is an $H_\BF$-stable subalgebra on which $H_\BF$ acts trivially.
\end{proof}

\begin{Prop}
$(\CG,\BF,H_\BF,H'_\BF)$ satisfies \ct\ if and only if $Z(\Fg)=0$. 
\end{Prop}
\begin{proof}
The centers of the various $\CG^\sco$ are described in \Cref{Eqn:Centers}. For a general $\CG$, the central isogeny $\phi^\sco\colon\CG^\sco\to\CG$, introduced above \Cref{Eqn:Centers}, gives a short exact sequence $1\to  \kernel(\phi^\sco)\to Z(\CG^\sco)\to Z(\CG)\to 1$. From the explicit description in~\Cref{Eqn:Centers}, it follows that $Z(\CG)$ is smooth (over $\WF$) if and only if $p$ does not divide the order of $Z(\CG)$, and the latter is equivalent to $\Lie Z(\CG)$ being trivial. By \Cref{Thm:LieCenter} this is equivalent to $Z(\Fg)$ being zero. It remains to show that $H^0(H'_\BF,\Fg)=0$ provided that $Z(\Fg)=0$.  This can be read off from \cite[Prop.~11.1]{Pink-Compact}.~Note that there one considers $\Fg\otimes_\BF\overline\BF$; but this does not affect the vanishing of $H^0(H'_\BF,\Fg)$. Also, $\Fz$ in \cite{Pink-Compact} is defined as $\kernel(\Fg^\sco\to\Fg^\ad)$; but we have identified this with $Z(\Fg^\sco)$ in \Cref{Thm:LieCenter}.
\end{proof}

\begin{Rem}
If $Z(\Fg)$ is non-zero, we shall obtain in \Cref{Subsec:Puniv} an ``ambient group'' of $\CG$ for which \ct\ holds; see in particular \Cref{Thm:StdHypForPuniv}(a). 
\end{Rem}

\subsection{Condition \lie}
\label{Subsection:Lie}

This subsection is devoted to collecting results on the structure of the Lie algebra $\Fg$ of $\CG$.  Recall the maps $\phi$, $\phi^\sco$ and $\phi^\ad$ in $\CG^\sco\to\CG\to\CG^\ad$, and their induced maps $\mathrm{d}\phi^?$ on Lie algebras. Let in this subsection $G$ be the finite (Chevalley) group that is the image of $\CG^\sco(\BF)$ in $\CG^\ad(\BF)$ under $\phi$. We are going to consider the structure of the Lie algebra $\Fg$ of $\CG$, both as a Lie algebra and as an $\BF_p[G]$-module.  Note right away that the action of $H'_\BF$ on $\Fg$ factors via $G$ as it is trivial on the center of $H'_\BF$, and this explains why assertions for $G$ typically imply the same assertions for $H'_\BF$.
 
 Our main references will be \cite{Hogeweij}, \cite{Hiss}, and \cite{Pink-Compact}. Since the last two references work over an algebraically closed field, we begin with two lemmas to descend from $\overline\BF$ to $\BF$. For an $\BF$-vector space $V$, we write $V_{\overline \BF}$ for the base change  $V\otimes_\BF\overline\BF$.

\begin{Lem}\label{Lem-LieBaseChange}
Let $\Fh$ be a Lie algebra over $\BF$. Then the following hold:
\begin{enumerate}
 \item If $\Fa\subset \Fh$ is an ideal, then $\Fa_{\overline\BF}$ is an ideal of $\Fh_{\overline\BF}$ and the assignment $\Fa\mapsto \Fa_{\overline{\BF}}$ satisfies $\dim_\BF\Fa=\dim_{\overline\BF}\Fa_{\overline\BF}$ and is injective.  In particular, if $\Fh_{\overline\BF}$ is simple, then so is $\Fh$;  and if $\Fh$ splits as $\Fa\oplus \Fb$ for two ideals $\Fa$, $\Fb$, then $\Fh_{\overline\BF}$ splits as $\Fa_{\overline{\BF}}\oplus \Fb_{\overline{\BF}}$.

\item Under $\Fh\mapsto \Fh_{\overline\BF}$, we have $[\Fh_{\overline\BF},\Fh_{\overline\BF}]= [\Fh,\Fh]_{\overline\BF}$ and $Z(\Fh)_{\overline\BF}=Z(\Fh_{\overline\BF})$. In particular $\Fh$ is perfect, i.e., $[\Fh,\Fh]=\Fh$, if and only if $\Fh_{\overline\BF}$ is perfect.
\end{enumerate}
\end{Lem}

\begin{proof} Let $(X_i)_{i\in I}$ be a basis of $\Fh$ such that $(X_i)_{i\in J}$ for a subset $J\subset I$ is a basis for~$\Fa$. 

Because $\Fa$ is an ideal, all brackets $[X_i,X_j]$, for $i\in I$ and $j\in J$, lie in the $\BF$-span of $(X_i)_{i\in J}$. Since $\Fa_{\overline\BF}$ is also spanned by $(X_i)_{i\in J}$, this in turn implies that $\Fa_{\overline\BF}$ is an ideal of $\Fh_{\overline\BF}$. The injectivity of $\Fa\mapsto\Fa_{\overline\BF}$ follows from the injectivity for the corresponding map on $\BF$-vector spaces; the remaining parts are immediate and this completes the proof of~(a).

For (b) note that the $\BF$-span of the brackets $[X_i,X_{i'}]$, $i,i'\in I$, is $[\Fh,\Fh]$; their $\overline\BF$-span is $[\Fh_{\overline\BF},\Fh_{\overline\BF}]$. This implies the first assertion, and also the last. To see the remaining assertion, observe that $Z(\Fh)=\{ Y\in\Fh\mid\forall i\in I: [Y,X_i]=0 \}$. This is a linear system of equations. Hence an 
$\BF$-basis of $Z(\Fh)$ is also an $\overline\BF$-basis of $Z(\Fh_{\overline\BF})$, and thus also (b) is proved.
\end{proof}

\begin{Lem}\label{DescentFG}
Let $V$ be an $\BF[G]$-module and let $N$ be an $\BF$-vector subspace. Suppose that $N_{\overline\BF}\subset V_{\overline\BF}$ is invariant under $G$. Then $N$ is an $\BF[G]$-submodule of~$V$.
\end{Lem}
\begin{proof}
Let $(v_i)_{i\in I}$ be a basis of $V$ over $\BF$ such that there exists $J\subset I$ such that $(v_i)_{i\in J}$ is a basis of $N$. Let $j\in J$. The $G$-invariance of $N_{\overline\BF}$ implies that $gv_j=\sum_{j'\in J}\lambda_{jj'}v_j$ for some $\lambda_{jj'}\in\overline\BF$. The fact that $V$ carries a $G$ action implies that $gv_j=\sum_{i'\in I}\mu_{ji'}v_{i'}$ for suitable $\mu_{j i'}\in\BF$. The basis property of $(v_i)_{i\in I}$ yields $\mu_{jj'}=\lambda_{jj'}$ for $j'\in J$ and $\mu_{ji'}=0$ for $i'\in I\setminus J$. Hence $Gv_j\in N$ for all $j\in J$, and this implies the lemma.
\end{proof}

In many cases, the Lie algebra $\Fg$ of $\CG$ is simple, both as a Lie algebra and as an $\BF[G]$-module. From \cite[Hauptsatz und Korollare]{Hiss}, \cite[Cor.~2.7]{Hogeweij} and Lemma \ref{Lem-LieBaseChange}, the following is immediate:
\begin{Prop}\label{Lie:Simple}
The Lie algebra $\Fg$ is simple if and only if none of the following holds
\begin{enumerate}
\item $\Fg$ is of type $A_n$ and $p$ divides $n+1$.
\item $\Fg$ is of type $B_n$, $C_n$, $D_n$, $E_7$ or $F_4$ and $p=2$.
\item $\Fg$ is of type $E_6$ or $G_2$ and $p=3$.
\end{enumerate}
If $\Fg$ is simple as a Lie algebra, it is simple as an $\BF[G]$-module.
\end{Prop}

The following result is a complete classification of when $\Fg$ is perfect, i.e., when $[\Fg,\Fg]=\Fg$. It follows directly from \cite{Hogeweij} using \Cref{Lem-LieBaseChange}.
\begin{Prop}[\cite{Hogeweij}]\label{Lie:Perfect}
Suppose that $p\neq2$ if $\CG$ is of type $A_1$, $B_2$ or $C_n$. Then $\Fg$ is perfect if and only if $\CG$ is Lie-simply connected. Moreover, the map $\mathrm{d}\phi^\sco$  can fail to be an isomorphism in the following cases only:
\begin{enumerate}
\item $\Fg$ is of type $A_n$ and $p$ divides $n+1$.
\item $\Fg$ is of type $B_n$, $n\ge3$, $D_n$, $n\ge4$, or $E_7$ and $p=2$.
\item $\Fg$ is of type $E_6$ and $p=3$.
\end{enumerate}
\end{Prop}

The next result is again essentially due to Hiss and Hogeweij with some additions by Pink. There is also much overlap with \cite[Thm.~3.10]{Vasiu1}. Recall the exact sequences $0\to\Fz\to\Fg^\sco\to\overline\Fg\to 0$ and $0\to\overline\Fg\to\Fg^\ad\to\Fz^*\to 0$ of $\CG^{\ad}$-representations in which $\Fz\cong\Fz^*$ are the trivial representation; see~\cite[Prop.~1.11(a)]{Pink-Compact}. 

\begin{Thm}
\label{Thm:Lie1}
Suppose that $p\neq2$ if $\CG$ is of type $A_1$, $B_n$, $C_n$ or $F_4$ and that $p\neq3$ if $\CG$ is of type $G_2$. Let $H$ be any group with $G\subset H \subset \CG^\ad(\BF)$.
Then the following hold:
\begin{enumerate} 
\item\label{item-Irred-Gaction} As a Lie algebra and as an $\BF[H]$-module, $\overline\Fg$ is absolutely simple and non-trivial.
\end{enumerate}
Assume from now on that $p$ divides the order of $Z(\CG^\sco)$.
\begin{enumerate}\addtocounter{enumi}{1}
\item\label{Lie-soc} The socle of $\Fg^\sco$ as a Lie algebra and as an $\BF[H]$-module is $\Fz$.
\item\label{Lie-csc} The cosocle of $\Fg^\ad$ as a Lie algebra and as an $\BF[H]$-module is $\Fz^*$.
\item\label{Lie-g=gsc} 
$\CG$ is Lie-simply connected if and only if $\Fz'=0$; it is Lie-adjoint if and~only~if~$\Fz''=0$. 
\item\label{Lie-spcl} If $\Fz'$ and $\Fz''$ are non-zero, then $\dim\Fz'=\dim\Fz''=1$ and either of the following~holds:
\begin{enumerate}
\item $\CG$ is Lie-intermediate of type $A_n$ with $p^2|n+1$ or of type $D_n$ with $n$ odd and $p=2$, and then $\Fg\cong\overline\Fg\oplus\Fz''$ as Lie algebras and as $\BF[H]$-modules.
\item $\CG$ is Lie-intermediate of type $D_n$ with $n$ even and $p=2$, and then $\Fg$ possesses a unique composition series 
\[0\subsetneq \Fz'' \subsetneq[\Fg,\Fg]\subsetneq\Fg\] 
as a Lie algebra and as an $\BF[H]$-module, with $\Fg/[\Fg,\Fg] \cong\Fz'$, and $[\Fg,\Fg]/\Fz''\cong\overline\Fg$.
\end{enumerate}
\end{enumerate}
\end{Thm}
\begin{proof}
It clearly suffices to prove the result for $H=G$. The assertions on Lie algebras are from \cite[Table~1]{Hogeweij} -- strictly speaking the results in \cite{Hogeweij} are results on $\Fg_{\overline\BF}$; but by \Cref{Lem-LieBaseChange} this suffices to deduce the above results over $\BF$ from the corresponding ones over $\overline\BF$. The assertion on the $\BF[G]$-module structure is stated, over $\overline\BF$, in a similar way in \cite[Prop.~1.11]{Pink-Compact}. It completes previous work from \cite{Hiss,Hogeweij}.

\ref{item-Irred-Gaction} 
Consider \cite[Table 1]{Hogeweij}. Since $\overline{\Fg}_{\overline{\BF}}$ only depends on the type of $\CG$, we may assume that $\CG$ is simply connected, i.e., that $\Fg_{\overline{\BF}}$ of universal type in the sense of op.cit. In all cases that we allow, all nontrivial ideals are contained in $Z({\Fg}_{\overline{\BF}})$. Thus $\overline{\Fg}_{\overline{\BF}}\simeq \Fg_{\overline{\BF}}/Z(\Fg_{\overline{\BF}})$ is simple as a Lie algebra. That $\overline{\Fg}_{\overline{\BF}}=\Fg^{\sco}_{\overline{\BF}}/Z(\Fg^{\sco}_{\overline{\BF}})$ is a simple $\overline{\BF}[G]$-module is \cite[Hauptsatz]{Hiss}. Because $\dim_\BF\overline\Fg>1$, the action cannot be trivial.

 \ref{Lie-soc} Note first that under the hypothesis for \ref{Lie-soc}--\ref{Lie-spcl} we have $\Fz\neq0$ by \Cref{Thm:LieCenter}. By \cite[Table 1]{Hogeweij}, the only nontrivial ideals of $\Fg^{\sco}_{\overline{\BF}}$ are contained in the centre $\Fz_{\overline{\BF}}=Z(\Fg^{\sco}_{\overline\BF})$. By Lemma \ref{Lem-LieBaseChange}, the nontrivial ideals of $\Fg^{\sco}$ are thus contained in $\Fz$. By \ref{item-Irred-Gaction}, the socle of $\Fg^{\sco}$ as a Lie algebra is either $\Fz$ or $\Fg^{\sco}$. If it was  $\Fg^{\sco}$, then there should be an ideal $\Fa$ such that $\Fz\oplus \Fa=\Fg^{\sco}$, and after enlarging coefficients we would have $\Fz_{\overline\BF}\oplus \Fa_{\overline\BF}=\Fg^{\sco}_{\overline{\BF}}$. But \cite[Table 1]{Hogeweij} displays no such ideal. 
 
We now consider $\Fg^\sco$ as an $\BF[H]$-module. From \Cref{Thm:LieCenter} we see that $0\neq\Fz\subset\Fg$ is an $\BF[H]$-submodule on which $H$ acts trivially. Hence it is part of the $\BF[H]$-socle of $\Fg^{\sco}$. If the socle was strictly larger, then by \ref{item-Irred-Gaction} it would be all of $\Fg$, and by semisimplicity of the socle there would be a subrepresentation $\Fh\subset \Fg$ that would map isomorphically to $\overline\Fg$. But then $\Fg=\Fh\oplus\Fz$. This contradicts \cite[Hauptsatz]{Hiss}, which asserts also that $\Fg^{\sco}_{\overline{\BF}}$ is indecomposable as an $\overline{\BF}[G]$-module.

\ref{Lie-csc} 
The proof is dual to that of \ref{Lie-soc}. From the ideals displayed in  \cite[Table 1]{Hogeweij}, it follows that $\Fz^*$ is an abelian Lie algebra, and by~\cite[Prop.~1.11(a)]{Pink-Compact} we have that the action of $H$ on $\Fz^*$ is trivial. Moreover from  \cite[Table 1]{Hogeweij} it follows $\overline\Fg=[\Fg^\ad,\Fg^\ad]$ is absolutely simple and contained in all non-trivial Lie ideals of $\Fg^\ad$ and from~\cite[Prop.~1.11(b)]{Pink-Compact} which is based on \cite{Hiss} and our hypothesis on the types, it follows that $\overline\Fg$ is absolutely irreducible and contained in all non-trivial subrepresentations of $\Fg^\ad$. Hence arguing as in \ref{Lie-soc} one deduces \ref{Lie-csc}.

\ref{Lie-g=gsc} Follows from $\dim_\BF\Fg^\sco=\dim_\BF \Fg=\dim_\BF\Fg^\ad$ and the definition of $\Fz'$ and $\Fz''$ as the kernels of $\mathrm{d}\phi^{\sco}$ and of $\mathrm{d}\phi^{\ad}$, respectively.

\ref{Lie-spcl}  If both $\Fz'$ and  $\Fz''$ are non-zero, then neither $\mathrm{d}\phi^{\sco}$ nor $\mathrm{d}\phi^{\ad}$ is an isomorphism, and so $\CG$ is Lie-intermediate. Looking at \cite[Table 1]{Hogeweij}, this can only happen if $\CG$ is of type $A_n$, with $p^2\vert (n+1)$, or $p=2$ and types $D_{2n}$ or $D_{2n+1}$,  and in either case $\dim\Fz''=1$.  Because $\phi^\sco\colon \CG^\sco\to\CG$ is a universal central isogeny, we have the exact sequence $1\to\kernel\phi^\sco\to Z(\CG^\sco)\to Z(\CG)\to1$. Applying the functor $\Lie(\cdot)$ yields the left exact sequence
\vspace*{-.4em}

\begin{equation}\label{Eqn:LieOfPhiSco}
0\rightarrow \Fz'\rightarrow \Fz\stackrel{\mathrm{d}\phi^\sco}\rightarrow \Fz''.
\end{equation}

In case $A_n$, we have $\dim_\BF\Fz=1=\dim_\BF\Fz''$. From the above sequence and $\Fz'\neq0$ we deduce that $\Fz'=\Fz$, and so $\dim_\BF\Fz'=1$ and  $\mathrm{d}\phi^\sco(\Fz)=0$. Next we establish (i) for $A_n$. (The case $D_n$ with $n$ odd is analogous). According to \cite[Table 1]{Hogeweij}, the non-trivial ideals of $\Fg_{\overline{\BF}}$ are $\Fz''_\oBF=Z(\Fg_{\overline{\BF}})$ of dimension 1 and $[\Fg_{\overline{\BF}}, \Fg_{\overline{\BF}}]$ of codimension $1$ of $\Fg_{\oBF}$. It follows from \Cref{Lem-LieBaseChange} that the non-trivial ideals of $\Fg$ are $\Fz''$ of dimension $1$ and $[\Fg,\Fg]$ of codimension $1$. Another codimension $1$ Lie subalgebra of $\Fg$ is the image of $\mathrm{d}\phi^\sco$ which is isomorphic to $\overline\Fg$ and hence simple. Since $\dim_\BF\Fg-1=\dim_\BF\overline\Fg>1=\dim_\BF\Fz''$, we must have $\overline\Fg\cong[\Fg,\Fg]$ and $\Fz''\cap[\Fg,\Fg]=0$, and now (i) follows for $A_n$ and the Lie algebra structure. 

 In both cases $A_n$ and $D_n$ with $n$ odd, we know that $\Fz''$ is a trival $\BF[G]$-module of $\BF$-dimension $1$. Moreover, the image of $\mathrm{d}\phi^{\sco}$ is an $\BF[G]$-submodule of $\Fg$. We have already seen that $\image  \mathrm{d}\phi^{\sco}=[\Fg, \Fg]\cong\overline\Fg$ and $\Fz''$ have trivial intersection, and thus we have $\Fg=\Fz''\oplus [\Fg, \Fg]$ as $\BF[H]$-modules.

Suppose now that $p=2$ and that the type is $D_{n}$ with $n$ even. Then $Z(\CG^\sco)\cong\mu_2\times\mu_2$, and because $\CG$ is neither of simply connected nor of adjoint type, we must have $\kernel\phi^\sco\cong\mu_2\cong Z(\CG)$. It follows that \eqref{Eqn:LieOfPhiSco} is also exact on the right, and moreover that $\dim_\BF\Fz'=\dim_\BF\Fz''=1$. From  \cite[Table 1]{Hogeweij} and using \Cref{Lem-LieBaseChange}, we find again that the non-trivial ideals of $\Fg$ are $\Fz''$ of dimension $1$ and $[\Fg,\Fg]$ of codimension $1$. At the same time $\mathrm{d}\phi^\sco(\Fg^\sco)$ is an ideal of codimension $1$, and as a Lie algebra it is perfect, because this holds for $\Fg^\sco$ by \Cref{Lie:Perfect}. Since $\dim_\BF\Fg>2$, we deduce $\mathrm{d}\phi^\sco(\Fg^\sco)=[\Fg,\Fg]$, and we have $\overline\Fg=\Fg^\sco/\Fz\cong[\Fg, \Fg]/\Fz''$. Thus as a Lie algebra we have the composition series described in (ii), and since $\Fg$ has no other ideals, it is the unique composition series. 

It remains to understand the $\BF[H]$-module structure of $\Fg$. By \ref{Lie-soc} we have that $\Fz\subset\Fg^\sco$ is an $\BF[H]$-submodule on which $H$ acts trivially. Hence $\Fz'$ has the same property. Because $\mathrm{d}\phi^\sco$ is $H$-equivariant, it follows that $\Fz''=\mathrm{d}\phi^\sco(\Fz)$ and $[\Fg,\Fg]=\mathrm{d}\phi^\sco(\Fg^\sco)$ are $\BF[H]$-submodules of $\Fg$. Moreover $H$ acts trivially on $\Fz''$ and by \ref{item-Irred-Gaction} the quotient $[\Fg,\Fg]/\Fz''\cong\overline\Fg$ is absolutely irreducible and non-trivial. Since $\Fz''$ lies in $[\Fg,\Fg]$ and is the kernel of $\mathrm{d}\phi^\ad$, we have an injection $\Fg/[\Fg,\Fg]\into \Fg^\ad/\overline\Fg=\Fz^*$ as $\BF[H]$-modules. This shows that $H$ acts trivially on $\Fg/[\Fg,\Fg]$, and for dimension reasons the latter is isomorphic to $\Fz'$ as an $\BF[H]$-module. 

It remains to prove the uniqueness of the $\BF[H]$ composition series in~(ii). Let $\Fs$ be the socle and $\Fc$ be the cosocle of $\Fg$. We need to show that the canonical inclusion $\Fz''\into \Fs$ and surjection $\Fc\onto\Fg/[\Fg,\Fg]$ are isomorphisms. By \cite[Prop. 1.11-(b)]{Pink-Compact}, we know that $\overline{\Fg}_\oBF$ is the cosocle of $\Fg_\oBF^{\sco}$ as $\oBF[G]$-module, and hence also as $\oBF[H]$-module. Passing to the quotient by $\Fz'_\oBF$ and applying \Cref{Lem-LieBaseChange}, we find that $\Fz''$ is the socle of $[\Fg,\Fg]$, and it follows that $\Fs\cap[\Fg,\Fg]=\Fz''$. Modding out by $\Fz''$, we have $\Fs/\Fz''\oplus \overline\Fg \subset\Fg^\ad$ as $\BF[H]$-modules. However again by \cite[Prop. 1.11-(b)]{Pink-Compact}, the socle of $\Fg^\ad_\oBF$ as an $\BF[G]$-module is $\overline\Fg_\oBF$, and from \Cref{Lem-LieBaseChange}, we deduce $\Fs/\Fz''=0$, which shows the first isomorphism. The argument for $\Fc$ is dual but analogous, and so we omit it.
\end{proof}

\begin{Rem}
 For $p=2$ and $\CG^{\sco}$ of type $A_1$, $[\Fg^{\sco}, \Fg^{\sco}]=Z(\Fg^{\sco})$ (cf.~\cite[Table 1]{Hogeweij}), thus $\overline{\Fg}$ is abelian as a Lie algebra. However, if $q\geq 4$, it is indecomposable as an $\overline{\BF}[G]$-module (cf.~\cite[Hauptsatz]{Hiss}).
\end{Rem}

\begin{Rem}
For those $\CG$  not of type $A_1$ that are not included in \Cref{Thm:Lie1}, the representation $\overline\Fg$ possesses a non-trivial composition series, as an $H$-module and as a Lie algebra. Much of this is related to some non-standard isogenies between types $B_n$ and $C_n$ or from $F_4$ to $F_4$ if $p=2$ and from type $G_2$ to $G_2$ if $p=3$; see \cite[Sect.~1]{Pink-Compact}.
\end{Rem}
\begin{Rem}
\label{Rem:Exc-Center}
Assertions \ref{Lie-soc} and \ref{Lie-csc} of \Cref{Thm:Lie1} are still true in the cases, $p=2$, type $C_n$ with $n$ even, and type $B_n$ with $n\geq 4$ even, as can be directly verified from \cite{Hiss,Hogeweij}. Also, in all cases, i.e., also those excluded in \Cref{Thm:Lie1}, it holds true that the maximal $\BF_p[G]$-submodule of $\Fg^{\sco}$ on which $G$ acts trivially is $\Fz$ (cf.~\cite[Hauptsatz]{Hiss}).
\end{Rem}

\bigskip

Next, we study the canonical map $\BF\to\End_{\BF_p[H'_\BF]}(\overline\Fg)$. For this we need the following result:
\begin{Lem}\label{Lem:OnEndR}
Let $\Gamma$ be a finite group, let $L\supset K$ be any extensions of fields, and let $V$ be a $K[\Gamma]$-module. Then
\begin{enumerate}
\item If $L\otimes_KV$ is completely reducible as an $L[\Gamma]$-module, then $V$ is completely reducible as a $K[\Gamma]$-module.
\item Suppose $V$ is completely reducible as a $K[\Gamma]$-module, $\End_{K[\Gamma]}(V)$ contains a finite field extension $E$ of $K$ and $\dim_L\End_{L[\Gamma]}(L\otimes_K V)=\dim_KE$, then $E=\End_{K[\Gamma]}(V)$ and $V$ is irreducible as a $K[\Gamma]$-module.
\end{enumerate}
\end{Lem}
\begin{proof}
Lacking a reference for the certainly well-known results of the lemma, we give a proof: We deduce (a) from the following result of M. Deuring and E. Noether, cf.~\cite[Theorem (29.11)]{CurtisReinerOld}. Let $W,W'$ be finite-dimensional $K[\Gamma]$-modules such that $L\otimes_KW\cong L\otimes_KW'$ as $L[\Gamma]$-modules. Then $W\cong W'$. 

If now $V_0=0\subsetneq V_1\subsetneq\ldots V_{i-1}\subsetneq V_i=V$ is a composition series of $V$ with simple quotients, then by our hypothesis on $L\otimes_KV$ we have $L\otimes_KV\cong L\otimes_K (\oplus_{j=1}^i V_i/V_{i-1})$. It follows from Deuring-Noether that $V$ is completely reducible.

For part (b) note first that $\End_{K[\Gamma]}(V)\otimes_KL$ injects into $\End_{L[\Gamma]}(L\otimes_KV)$, because for any finite-dimensional $K$-vector space $V$ one has $\End_K(V)\otimes_KL\cong \End_{L}(L\otimes_KV)$. We deduce that $E\otimes_KL$ injects $L$-linearly into $\End_{L[\Gamma]}(L\otimes_K V)$, and using our dimension hypothesis we see that $E\otimes_KL\to \End_{L[\Gamma]}(L\otimes_K V)$ is an isomorphism, and hence the inclusion $E\into \End_{K[\Gamma]}(V)$ must be an isomorphism as well. Because $V$ is semisimple and $E$ is a field, we also deduce that $V$ is a simple $K[\Gamma]$-module.
\end{proof}

\begin{Prop}\label{Prop:HWsubmodule}
Suppose that $\CG^\sco(\BF)\notin\CE_{\pf}$. Denote by $\Fg^\hw$ the $\BF[G]$-subquotient of $\Fg$ of highest weight. Then $\Fg^\hw$ is irreducible as an $\BF_p[H'_\BF]$-module and the canonical map $\BF\to\End_{\BF_p[H'_\BF]}(\Fg^\hw)$ is an isomorphism.
\end{Prop}
The proposition expresses the fact that $\BF$ is the smallest coefficient field over which the absolutely irreducible $\BF[G]$-module $\Fg^\hw$ can be defined. Lacking a reference, we give a proof.
\begin{proof}
In \cite[Prop.~1.11]{Pink-Compact} a composition series of $\Fg_{\overline\BF}$ as an $\overline\BF[G]$-module is given. According to \cite[Table 1]{Hogeweij} there is a filtration of $\BF$-Lie subalgebras of $\Fg$ whose scalar extension to $\overline\BF$ is the decomposition series from \cite{Pink-Compact}, and hence by \Cref{DescentFG}, the filtration deduced  from \cite{Hogeweij} is also one of $\BF[G]$-modules. It follows that the highest weight subquotient of $\Fg_{\overline\BF}$ is defined over $\BF$, i.e., that $\Fg^\hw$ is defined and absolutely irreducible. Hence by \Cref{Lem:OnEndR}(a), the $\BF_p[G]$-module $\Fg^\hw$ is completely reducible. 

We consider $\Fg^\hw\otimes_{\BF_p}\BF$ as a $\BF[H'_\BF]$-module. Using the decomposition
\[\BF\otimes_{\BF_p} \BF\stackrel{\simeq}\longto \!\! \bigoplus_{\sigma\in\Gal(\BF/\BF_p)}  \!\!\BF \ \ , \alpha\otimes\beta\mapsto (\alpha\sigma(\beta))_{\sigma\in\Gal(\BF/\BF_p)},\]
we find $\Fg^\hw\otimes_{\BF_p}\BF\cong \bigoplus_{\sigma\in\Gal(\BF/\BF_p)}  \Fg^\hw\otimes_{\BF}^\sigma\BF$, where each tensor product uses a different Galois automorphism. We now use the Steinberg Tensor Theorem and Steinberg's theory of irreducible representations of finite Chevalley groups to deduce that the $\Fg^\hw\otimes^\sigma_\BF\BF$ are pairwise non-isomorphic. 

Let $\FM$ denote the set of irreducible restricted (algebraic) representations of $\CG^\sco$; their number is $p^\ell$ where $\ell$ denotes the rank of $\CG^\sco$; see \cite[p.~36]{Steinberg-Reps}. Let $n:=[\BF:\BF_p]$. Note that for $n=1$ there is obviously nothing to show. By  \cite[Thm.~7.4 and 9.3]{Steinberg-Reps} every irreducible representation $V$ of $\CG^\sco(\BF)$ can be written as $\otimes_{i=0}^{n-1}V_i^{(i)}$ for a unique choice $(V_0,\ldots,V_{n-1})\in\FM^{n}$, and where the superscript ${^{(i)}}$ denotes the $i$-th Frobenius twist of $V_i$; here the $V_i$ are supposed to be considered as representations of $\CG^\sco(\BF)$ and therefore one has $W^{(n)}=W$ for  $W\in\FM$. Choose the $V_i$ for $\Fg^\hw$, considered as a representation of $\CG^\sco(\BF)$ via the surjection $\CG^\sco(\BF)\to H'_\BF$. Let $\Fro\in\Gal(\BF/\BF_p)$ be the geometric Frobenius automorphism $\alpha\mapsto \alpha^{1/p}$. It is elementary to see that $\Fg^\hw\otimes^{\Fro^i}_\BF\BF=\Fg^{\hw,(i)}$. It follows that $\Fg^\hw\otimes^{\Fro^j}_\BF\BF\cong \otimes_{i=0}^{n-1}V_i^{(i+j)}$. Denote by $\uw:=(w_0,\ldots,w_{n-1})$ the tuple of $p$-restricted weights such that $\sum_{i=0}^{n-1}w_ip^i$ is the highest weight of $\Fg^\hw$. The $p$-restricted weights $w_i$ are linear combinations of the fundamental weights with coefficients in $\{0,\ldots,p-1\}$. Then it remains to show that  no cyclic permutation of $\uw$ except for the identity, fixes the list $\uw$.

By \cite[Planches]{Bourbaki-Lie-4-6}, the highest weight of $\Fg^\hw$ in terms of a basis of fundamental weights $\varpi_1,\ldots,\varpi_\ell$, depending on the type, is given in the following table:
\begin{equation}\label{Eqn:Weights}
\begin{tabular}{r|c|c|c|c|c|c|c|c|c|c|c|c|c|c|c|} 
type & $A_\ell$&$B_2$ & $B_\ell$, ${\scriptstyle \ell\ge3}$ & $C_\ell$, ${\scriptstyle \ell\ge3}$ & $D_\ell$, ${\scriptstyle \ell\ge4}$ & $E_6$ &$E_7$& $E_8$ & $F_4$ & $G_2$\\\hline
$\hw${} & $\varpi_1+\varpi_\ell$ & $2\varpi_2$ & $\varpi_2$ & $2\varpi_1$  & $\varpi_2$ & $\varpi_2$ & $\varpi_1$ & $\varpi_8$ & $\varpi_1$ & $\varpi_2$\\
\end{tabular}
\end{equation}
As a $\BZ$-linear combination in the basis $(\varpi_i)_{i=1,\ldots,\ell}$, only the coefficients $0$, $1$ and $2$ occur, and $2$ only occurs for types $B_\ell$ and $C_\ell$ and $p=2$. Hence apart from this exception the highest weight of $\Fg^\hw$ itself is $p$-restricted, and thus the $\Fg^{\hw,(i)}$, $i=0,\ldots,n-1$ are pairwise non-isomorphic. If we are in the case $B_\ell$ or $C_\ell$ and $p=2$, then $\uw=(0,\varpi_2,0,\ldots,0)$ or $\uw=(0,\varpi_1,0,\ldots,0)$, respectively, and again no non-trivial cyclic permutation fixes $\uw$, and so also in this case the $\Fg^{\hw,(i)}$, $i=0,\ldots,n-1$ are pairwise non-isomorphic.

We now have seen that $\Fg^\hw\otimes_{\BF_p}\BF\cong \bigoplus_{\sigma\in\Gal(\BF/\BF_p)}  \Fg^\hw\otimes_{\BF}^\sigma\BF$ is a decomposition into pairwise non-isomorphic absolutely irreducible modules. This implies
\[ \End_{\BF[H'_\BF]}(\BF\otimes_{\BF_p}\Fg^\hw)\cong \BF\times\ldots\times \BF\]
with $n$ factors on the right, so that $\dim_\BF \End_{\BF[H'_\BF]}(\BF\otimes_{\BF_p}\Fg^\hw)=n =[\BF:\BF_p]$. The proposition now follows from \Cref{Lem:OnEndR}(b).
\end{proof}

\begin{Cor}
Suppose that $p\neq2$ if $\CG$ is of type $B_n$, $C_n$ or $F_4$ and that $p\neq3$ if $\CG$ is of type $G_2$, and suppose that $\CG^\sco(\BF)\notin\CE_{\pf}$, so that $\CG^\sco(\BF)$ is perfect. Then $\overline\Fg$ is irreducible as an $\BF_p[H'_\BF]$-module and the canonical map $\BF\to\End_{\BF_p[H'_\BF]}(\overline\Fg)$ is an isomorphism.
\end{Cor}
\begin{proof}
By our hypotheses and \Cref{Thm:Lie1}, the module $\overline\Fg$ is absolutely irreducible, as a module over $\BF[H'_\BF]$, and it is not difficult to see that $\overline\Fg\cong\Fg^\hw$ in all cases of the corollary. Hence the result follows from \Cref{Prop:HWsubmodule}.
\end{proof}

From what we proved in \Cref{Subsection:Lie} so far, the following result is now immediate.
\begin{Cor}\label{Cor:LieCombined1} If $\CG^\sco(\BF)$ is not in $\CE_{\pf}$, then the triple ($\CG$, $H_{\BF}$, $H'_{\BF})$ satisfies $\liegen$ if $\CG$ is of type $A_n$ and $p\nmid n+1$, or if $\CG$ is of type $B_n$,  $C_n$, $D_n$, $E_7$ or $F_4$ and $p\not=2$, or if $\CG$ is of type $E_6$ or $G_2$ and $p\not=3$, or if $\CG$ is of type $E_8$.
\end{Cor}

For \lieu\ we also need the following result, where for an $\BF[G]$-module $V$ we denote by $\socle(V)$ and $\cosocle(V)$ the socle and cosocle of $V$, respectively:
\begin{Lem}\label{Lem:RedToFgHW}
Let $\overline V$ be an $\BF[G]$-module that is irreducible over $\BF_p[G]$, carries a non-trivial $G$-action and for which the natural ring map $\BF\to\End_{\BF_p[G]} (\overline V)$ is an isomorphism. Suppose that $V$ is an $\BF[G]$-module that satisfies one of the following conditions:
\begin{enumerate}
\item The module $V$ sits in a short exact sequence $0\to \socle(V)\to V\to \cosocle(V)\to 0$ such that $\{\socle(V),\cosocle(V)\}=\{\BF^r,\overline V\}$ for some $r\ge1$.
\item The module $V$ possesses a $3$-step filtration $0\subsetneq V_1\subsetneq V_2\subsetneq V$ such that $V_1=\socle(V)\cong\BF$, $V/V_2=\cosocle(V)\cong \BF$ and $V_2/V_1\cong \overline V$.
\item The module $V$ sits in a short exact sequence $0\to \socle(V)\to V\to \cosocle(V)\to 0$ such that $\{\socle(V),\cosocle(V)\}=\{\BF\oplus \overline V,V'\}$ for some irreducible $\BF[G]$-module $V'$ such that $\BF,V',\overline V$ are pairwise non-isomorphic.
\item The module $V$ possesses a unique filtration $0=V_0\subsetneq V_1 \subsetneq V_2 \subsetneq V_3 \subsetneq V_4=V$ with irreducible $\BF[G]$-subquotients $Q_i=V_i/V_{i-1}$, $i=1,\ldots,4$, and moreover one of the following holds:
\begin{enumerate}
\item $Q_1\cong Q_3\cong\BF$ and $\{Q_2,Q_4\}=\{\overline V,V'\}$ for some irreducible non-trivial $\BF[G]$-module $V'$ such that $\BF,V',\overline V$ are pairwise non-isomorphic.\item $Q_2\cong Q_4\cong\BF$ and $\{Q_1,Q_3\}=\{\overline V,V'\}$ for some irreducible non-trivial $\BF[G]$-module $V'$ such that $\BF,V',\overline V$ are pairwise non-isomorphic.\end{enumerate}
\end{enumerate}
Then the natural map $\BF\to\End_{\BF_p[G]}(V)$ is an isomorphism.
\end{Lem}
\begin{proof}
The proofs are all fairly standard. They use properties of the socle and cosocle of modules over group rings. For instance if $V\to V'$ is an $\BF[G]$-module homomorphism, then it induces maps $\socle(V)\to\socle(V')$ and $\cosocle(V)\to\cosocle(V')$. Case (b) can use case (a) and case (d) can benefit from (a) and sometimes (b). As a sample proof we give one half of the proof in case~(c):

Suppose that $\socle(V)=\BF\oplus\overline V$ and $\cosocle(V)=V'$, and let $\phi\colon V\to V$ be in $\End_{\BF_p[G]}(V)$. Since $\phi$ preserves the socle, it induces a map $\phi|_{\BF\oplus\overline V}$ in $\End_{\BF_p[G]}(\BF\oplus \overline V)$. The latter ring is isomorphic to $\BF\times\BF$ by our hypothesis on $\End_{\BF_p[G]}(\overline V)$ and the fact that $\overline V$ is irreducible and non-isomorphic to $\BF_p$. By replacing $\phi$ by $\phi':=\phi-\lambda\id_V$, we may assume that $\phi'$ restricts to the zero map on $\overline V$. We need to prove that $\phi'=0$.

The map $\phi'$ induces a homomorphism $\bar\phi'\colon V/\overline V\to V$. Suppose first that $\bar\phi'$ is injective. Then $V/\overline V$ is a direct summand of $V$, so that $V\cong \overline V\oplus V/\overline V$, since $\overline V$ and $V'$ are irreducible over $\BF_p[G]$ and of different dimensions and since they are not isomorphic to $\BF_p$. But the $\cosocle(V)$ surjects onto $V'\oplus \overline V$, which contradicts our hypothesis. Hence $\kernel\bar\phi'\neq0$. We note that $V/\overline V$ is a non-trivial extension of $V'$ by $\BF$, since otherwise $\cosocle(V)$ would surject onto $\BF\oplus V'$ which is not allowed. Thus $\kernel\bar\phi'\in\{\BF,V/\overline{V}\}$, and we need to rule out the first case. 

Suppose $\BF=\kernel\bar\phi'$. Then $\phi$ vanishes on $\BF\oplus\overline V$ and induces an injective homomorphism $V'\cong V/(\BF\oplus\overline V)\into V$. But then $V'\subset\socle(V)$, which contradicts the hypotheses of (c). This completes the proof of~(c).
\end{proof}

\begin{Cor}\label{cor:Lie-un-ii}
Suppose that $\CG^\sco(\BF)$ is not in $\CE_{\pf}$ and one of the following conditions~holds:
\begin{enumerate}
\item $\Fg$ is irreducible, i.e., \textbf{(l-ge)}(ii) holds.
\item $\CG$ is Lie-simply connected or Lie-adjoint and of type $A_n$ with $p|n+1$ and not $(\BF,n)=(\BF_2,2)$ or of type $D_n$ with $p=2$, or  $\CG$ is of type $E_7$ with $p=2$, or of type $E_6$ with $p=3$.
\item $\CG$ is Lie-intermediate of type $D_n$ with $p=2$ and $n$ even.
\item $\CG$ is of type $B_n$ or $C_n$, $n\ge3$, and $p=2$.
\end{enumerate}
Then the natural map $\BF\to\End_{\BF_p[H'_\BF]}(\Fg)$ is an isomorphism.
\end{Cor}
\begin{proof}
If \liegen(ii) holds, then $\Fg=\Fg^\hw$ and we can directly apply \Cref{Prop:HWsubmodule}. In all other cases we apply \Cref{Lem:RedToFgHW}. The case (b) here reduces to (a) of \Cref{Lem:RedToFgHW}, the case (c) here reduces to (b), the case (d) here reduces to (c) if $n$ is odd and to  (d) if $n$ is even. To see that one can apply  \Cref{Lem:RedToFgHW}, note first that by \cite[Prop~1.11]{Pink-Compact} we have filtrations of $\Fg_{\overline\BF}$ as an $\overline\BF[G]$-module with the properties described in \Cref{Lem:RedToFgHW}(a)--(d), and that by \cite[Table 1]{Hogeweij} there is a filtration of $\BF$-sub Lie algebras of $\Fg$ whose scalar extension to $\overline\BF$ is that of  \cite[Prop~1.11]{Pink-Compact}. Hence by \Cref{DescentFG}, the filtration descends to one of $\BF[G]$-modules of $\Fg$, as does the uniqueness property stated in \cite[Prop.~1.11]{Pink-Compact}. This establishes the required properties on socles, cosocles and filtrations needed in \Cref{Lem:RedToFgHW}(a)--(d). That $\BF,V',\overline V$ are pairwise non-isomorphic is most easily deduced from \cite[Diagrams in Hauptsatz]{Hiss}, where it is observed that they have pairwise distinct dimensions.
\end{proof}
\begin{Rem}
In the following situations not covered by \Cref{cor:Lie-un-ii}, the canonical map $\BF\to\End_{\BF_p[G]}(\Fg)$ is not an isomorphism:
\begin{enumerate}
\item $\CG$ is of type $B_2$ or $F_4$ and $p=2$ or of type $G_2$ and $p=3$; here $\End_{\BF_p[G]}(\Fg)\cong\BF[\eps]$. These are the cases where the Ree and Suzuki groups occur.
\item $\CG$ is of Lie-intermediate type $A_n$ with $p|n+1$ or $D_n$ with $p=2$ and $n$ odd; here $\End_{\BF_p[G]}(\Fg)\cong\BF\times\BF$.
\end{enumerate}
\end{Rem}
\begin{Lem}
Condition \lieu(i) holds unless $\CG$ is of type $B_2$.
\end{Lem}
\begin{proof}
The commutator Lie subalgebras $[\Fg,\Fg]$ and the quotients $\Fg/[\Fg,\Fg]$ are given in \cite[Table1]{Hogeweij}. Except for types $A_1$ and $p=2$, for types $C_l$, $l>2$, and $\CG$ Lie-universal, and for types $B_2$, the quotient is of the form $\BF^r$ with $r\in\{0,1,2\}$. If the latter is the quotient then $\Fg^\hw$ is a JH-factor of $[\Fg,\Fg]$, and so \lieu(i) holds. For type $A_1$ and $p=2$ we have $[\Fg,\Fg]=\BF$ and the quotient $\Fg/[\Fg,\Fg]$ is  non-trivial irreducible $\BF[G]$-module. For $C_l$ with $l>2$ and $\CG$ Lie-universal, the quotient $\Fg/[\Fg,\Fg]$ is irreducible but not isomorphic to $\Fg^\hw$ by \cite[Hauptsatz]{Hiss}.
\end{proof}

Combining the last two results, we obtain:
\begin{Cor}\label{Cor:LieCombined2} If $\CG^\sco(\BF)\notin\CE_{\pf}$, then $\lieu$ holds unless $\CG$ is Lie-intermediate of type $A_n$ with $p|n+1$ or $D_n$ with $n$ odd and $p=2$, or $\CG$ is of type $B_2$ or $F_4$ and $p=2$, or of type $G_2$ and $p=3$.
\end{Cor}

Lastly we turn to \liec.
\begin{Lem}\label{Cor:LieCombined3}
If $\CG^\sco(\BF)\notin\CE_{\pf}$, then $\liec$ holds if $\CG$ is Lie-simply connected of type $A_n$, $n\ge2$, $D_n$, $E_n$, or $\CG$ is of type $A_1$, $B_n$,  $C_n$, $F_4$ and $p\neq2$, or of  type $G_2$ and $p\neq3$.
\end{Lem}
\begin{proof}
The condition that $\Fg$ is perfect, i.e. \liec(i), is completely described in \Cref{Lie:Perfect}. It requires $\CG$ to be Lie-simply connected and that the type of $\CG$ is not $A_1$, $B_2$ or $C_n$ if $p=2$. Condition \liec(ii), that $\socle(\Fg)\cong\BF^r$ for $r\ge0$ and $\cosocle(\Fg)=\overline\Fg$ (and hence $\cosocle(\Fg)=\Fg^\hw$), can be read off from \cite[Prop.~11.1]{Pink-Compact}. It requires us to further exclude $B_n$, $F_4$ of $p=2$ and $G_2$ if $p=3$. Note also \Cref{Thm:LieCenter} concerning the $H_\BF$-action.
\end{proof}

Let us also state the complete result concerning \liecsc.
\begin{Prop}\label{Prop:csc}
If $\CG$ is Lie-simply connected, then \liecsc\ holds.
\end{Prop}
\begin{proof}
It follows from \cite[Hauptsatz]{Hiss} that in the case where $\CG$ is Lie-simply connected the cosocle of $\Fg_{\overline\BF}$ contains no copy of the trivial $H'_\BF$-module $\overline\BF$. Hence the cosocle of $\Fg$ cannot contain a copy of $\BF$; and this implies \liecsc.
\end{proof}

\subsection{Condition \van}
\label{Subsection:Van}

This subsection collects the known results on the vanishing of $H^1(H'_\BF,\Fg)$. 
We begin with the following lemma that, when combined with results we shall recall later, suggests that it is most natural to expect the vanishing of $H^1(H'_\BF,\hat\Fg)$ in (almost) all cases (similar to \cite[Thm.~9]{TausskyZassenhaus} for  $\CG=\GL_n$). The lemma will be used later in \Cref{Subsec:Puniv}; it is also useful to determine $H^1(H'_\BF,\Fg)$ in some cases. 
\begin{Lem}\label{Gu-van}
Suppose that $(\CG,\BF)$ lies neither in $\CE_{\pf}$ nor in $\CE_{\sch}$. Then the following are equivalent:
\begin{enumerate}
\item $H^1(H'_\BF,\hat\Fg)=0$.
\item $\dim_\BF H^1(H'_\BF,\Fg)=\dim_\BF Z(\Fg)$.
\item $\dim_\BF H^1(H'_\BF,\overline\Fg)=\dim_\BF \Fz$.
\end{enumerate}
\end{Lem}
\begin{Rem}\label{Rem:Van}
The proof shows that if $\mathrm{d}\phi^\sco\colon\Fg^\sco\to\Fg$ is an isomorphism, then (a)$\Leftrightarrow$(b) only requires that $(\CG,\BF)$ is not in $\CE_{\pf}$; the same holds when $\mathrm{d}\phi^\ad\colon\Fg\to \Fg^\ad$ is an isomorphism for (b)$\Leftrightarrow$(c).
\end{Rem}
\begin{proof}
We first assume that $\mathrm{d}\phi^\sco\colon\Fg^\sco\to\Fg$ is an isomorphism. To prove (a)$\Leftrightarrow$(b) consider the short exact sequence $0\to \Fg^\sco\!\to\hat\Fg\to\Fz^*\to 0$ and the resulting long exact sequence
\begin{eqnarray*}
\lefteqn{0\to H^0(H'_\BF,\Fg^\sco)\to H^0(H'_\BF,\hat\Fg)\to\Fz^*\to}\\
&&\to H^1(H'_\BF, \Fg^\sco)\to H^1(H'_\BF,\hat\Fg)\to H^1(H'_\BF,\Fz^*) \to\ldots 
\end{eqnarray*}
of group cohomology. Note that by \Cref{Thm:Lie1} and \Cref{Rem:Exc-Center}, the module $\Fz$ is the maximal one of both $\Fg^\sco\subset\hat\Fg$ on which $H'_\BF$ acts trivially. Hence $H^0(H'_\BF,\Fg^\sco)\to H^0(H'_\BF,\hat\Fg)$ is an isomorphism in the above sequence. Also we have $H^1(H'_\BF,\Fz^*)\cong H^1(H'_\BF,\BF_p)\otimes_{\BF_p}\Fz^*$, because $\Fz^*$ is trivial as a $H'_\BF$-module. Therefore $H^1(H'_\BF,\Fz^*)=0$ by \Cref{Cor:H1Triv}, so that $0\to\Fz^*\to H^1(H'_\BF, \Fg^\sco)\to H^1(H'_\BF,\hat\Fg)\to0$ is a short exact sequence. This implies~(a)$\Leftrightarrow$(b).

To prove (b)$\Leftrightarrow$(c) consider the short exact sequence $0\to \Fz\to \Fg^\sco \to\overline \Fg\to 0$ and the resulting long exact sequence of group cohomology
\begin{eqnarray*}
\lefteqn{0\to \Fz\to H^0(H'_\BF,\Fg^\sco)\to H^0(H'_\BF,\overline\Fg)\to H^1(H'_\BF,\Fz)\to}\\
&&\to H^1(H'_\BF, \Fg^\sco)\to H^1(H'_\BF,\overline\Fg)\to H^2(H'_\BF,\Fz) \to\ldots 
\end{eqnarray*}
We deduce $H^1(H'_\BF,\Fz)=0$ as in the previous paragraph (there for $\Fz^*$ instead of $\Fz$). Moreover we obtain $H^2(H'_\BF,\Fz)=0$ from \Cref{Cor:Schur}. Since $\Fz=Z(\Fg^\sco)$ we showed~(b)$\Leftrightarrow$(c).

From now on we no longer assume that $\mathrm{d}\phi^\sco\colon\Fg^\sco\to\Fg$ is an isomorphism. Instead, we consider the short exact sequences (i) given by $ 0\to \kernel \mathrm{d}\phi^\sco\to \Fg^\sco\to\image \mathrm{d}\phi^\sco\to 0$ and (ii) given by $ 0\to \image \mathrm{d}\phi^\sco\to \Fg\to \coker  \mathrm{d}\phi^\sco\to 0$. We shall prove that (b) for $\Fg^\sco$ and (b) for $\Fg$ are equivalent: Using (i) and $\kernel \mathrm{d}\phi^\sco\subset \Fz$, and arguing as in (b)$\Leftrightarrow$(c), we deduce the isomorphism
\[H^1(H'_\BF, \Fg^\sco)\stackrel\simeq\to H^1(H'_\BF, \image \mathrm{d}\phi^\sco).\]
Note now that by \Cref{Thm:Lie1} and \Cref{Rem:Exc-Center}, the module $Z(\image \mathrm{d}\phi^\sco)$ is the maximal one of both $\image \mathrm{d}\phi^\sco \subset\Fg$ on which $H'_\BF$ acts trivially. Therefore we can now argue as in (a)$\Leftrightarrow$(b) in the first paragraph using the sequence (ii), to deduce the short exact sequence
\[ 0\to \coker  \mathrm{d}\phi^\sco \to H^1(H'_\BF, \image \mathrm{d}\phi^\sco)\to H^1(H'_\BF,\Fg)\to0. \] 
Observing that $\dim_\BF \coker  \mathrm{d}\phi^\sco =\dim_\BF \ker  \mathrm{d}\phi^\sco$, since $\dim_\BF \Fg^\sco = \dim_\BF \Fg$ we may now combine the above with the short exact sequence 
\[0\to  \ker  \mathrm{d}\phi^\sco \to \Fz \to Z(\image \mathrm{d}\phi^\sco)\to 0,\]
to deduce, by a simple dimension count, that (b) for $\Fg^\sco$ and (b) for $\Fg$ are equivalent.
\end{proof}

The following result is essentially due to Cline, Jones, Parshall, Scott for the A-D-E types and due to V\"olk\-lein for the others, as we shall indicate in its proof.
\begin{Thm}
\label{Thm-OnVan}
Suppose that $(\CG,\BF)$ satisfies the following conditions
\begin{enumerate}
\item[(i)] $($type$,\BF)$ is neither $(A_1,\BF_2)$ nor $(A_1,\BF_5)$.
\item[(ii)] If type$\ =C_n$, then $\#\BF\notin\{2,3,4,5,9\}$.
\item[(iii)] If $\CG$ is non-split (and hence of types A, D or E${}_6$), then $|\BF|\ge4$.
\end{enumerate}
Then 
\[ \dim_\BF H^1(H'_\BF,\Fg)=\dim_\BF  Z(\Fg). \]
\end{Thm}
\begin{proof}
Suppose first that $\CG$ is split. Then by \cite[Theorem and Remarks (a)]{Voelklein} conditions (i) and (ii) describe all cases that need to be excluded. Concerning  \cite[Remarks (a)]{Voelklein}, note that the assertion is not quite immediate from \cite[Cor.~2]{Voelklein-2}. Some additional work is required since the module $V$ defined and used in \cite{Voelklein-2} is not equal to $\overline\Fg$ for $\CG$ of type $B_n$, $C_n$, $F_4$ and $G_2$, cf.~\cite[Diagrams in Hauptsatz]{Hiss}. Hence in addition to \cite[Cor.~2]{Voelklein-2} one has to prove the vanishing of $H^1(H'_\BF,W)$ for certain subfactors $W$ of $\Fg$ directly. This is not difficult.

Suppose finally, that $\CG$ is non-split over $\BF$. Since we assume $|\BF|\ge4$ it is shown in \cite[Table 3]{CPS-2}, in all cases claimed, that $H^1(\CG^\sco(\BF),\overline\Fg)=\dim_\BF \Fz$. It follows from \Cref{Gu-van} that $H^1(\CG^\sco(\BF),\wh\Fg)=0$, since all exceptions in that lemma occur for $|\BF|<4$ in the non-split case. By inflation-restriction and \Cref{Cor:Pf}(a) we deduce $H^1(H'_\BF,\wh\Fg)=0$, and then again from \Cref{Gu-van}, the assertion.
\end{proof}

\begin{Rem}
Concerning the excluded $A_1$ cases, a direct computation shows that
\begin{eqnarray*}
H^1(\SL_2(\BF_2),\Fsl_2)\cong\BF_2\cong Z(\Fsl_2)&&H^1(\SL_2(\BF_5),\Fsl_2)\cong\BF_5\not\cong Z(\Fsl_2)=0\\
H^1(\PGL_2(\BF_2),\Fpgl_2)\cong\BF_2\not\cong Z(\Fpgl_2)=0&&H^1(\PGL_2(\BF_5),\Fpgl_2)=0=Z(\Fpgl_2)
\end{eqnarray*}
The results on the Lie algebra centers are straightforward. For the cohomology computations over $\BF_2$ one can simply use the Hochschild-Serre spectral sequence; over $\BF_5$, the result for $\SL_2$ can be found in \cite[Table~4.5]{CPS} and that for $\PGL_2$ can be easily derived via Hochschild-Serre from \cite[Lem.~1.2]{Flach} which states $H^1(\GL_2(\BF_5),\Fgl_2)=0$. 
\end{Rem}
Combining \Cref{Gu-van} and \Cref{Thm-OnVan}, we obtain.
\begin{Cor}
Condition \van\ holds if $\CG^\sco(\BF)\notin\CE_\pf$, $($type$,\BF)\notin\CE_{\sch}\cup\{(A_1,\BF_5)\}$, \ct\ holds, and if one of the following holds:
\begin{enumerate}
\item[(ii)] 
If type$\ =C_n$, then $\#\BF\notin\{2,3,4,5,9\}$.
\item[(iii)] 
If $\CG$ is non-split (and hence of types A, D or E${}_6$), then $|\BF|\ge4$.
\end{enumerate}
\end{Cor}

\subsection{Condition \ns}
\label{Subsection:Ns}

In this subsection we recall two results from \cite{Vasiu1} on \ns. Let $\CG$ be an absolutely simple semisimple group scheme defined over $W(\BF)$, and consider the exact sequence of $\BF_p[\CG(\BF)]$-modules
\begin{equation}\label{eq:ns} 0\rightarrow \Fg\rightarrow \CG(W_2(\BF))\stackrel{\pi_\CG}\rightarrow \CG(\BF)\rightarrow 0\end{equation}
The most comprehensive study of the splitness of such a sequence (and also for general semisimple $\CG$) is to our knowledge carried out in~\cite{Vasiu1}. 
\begin{Thm}[Vasiu]\label{Thm:Vasiu}
Suppose $\CG$ over $\WF$ is absolutely simple. Then \eqref{eq:ns} is non-split if $|\BF|\ge5$, or in the following cases:
\begin{enumerate}
\item $\BF=\BF_2$, $\CG$ is of adjoint type and the type is not in $\{A_1,A_2,{}^2A_2,{}^2A_3\}$, or $\CG$ is simply connected and the type is not in $\{A_1,A_2,{}^2A_2\}$.
\item $\BF=\BF_3$, $\CG$ is of adjoint type or simply connected and the type is not $A_1$.
\item $\BF=\BF_4$, $\CG$ is of adjoint type  and the type is not $A_1$ or $\CG$ is simply connected.
\end{enumerate}
Moreover for the exceptions listed in (a)--(c) it is known that \eqref{eq:ns} is split.

In addition, \eqref{eq:ns} is split if $\CG$ is of type ${}^2A_3$ but neither simply connected nor of adjoint~type.
\end{Thm}
\begin{proof}
The case $|\BF|\ge5$ follows from \cite[Prop.~4.4.1]{Vasiu1}. Cases (a)--(c) and the assertion thereafter follow from \cite[Thm.~4.5]{Vasiu1}; note that there is an omission in his result: the reduction $\SL_2(\BZ_2)\to\SL_2(\BF_2)$ has a splitting; see \cite[Prop.~20]{Dorobisz}. The last line, can be deduced from the second to last paragraph of the proof of \cite[Thm.~4.5]{Vasiu1}: As described there, one has a degree $2$ cover $\SO_6^-\to\PGU_4$ (note $D_3=A_3$ as Dynkin diagrams), which can be found in \cite[p.~26]{Atlas} or in \cite[Table 22.1]{MalleTesterman}. Vasiu deduces from \cite[p.~26]{Atlas} a splitting of $\CG(\BZ_2)\to\CG(\BF_2)$ for $\CG$ of type $\SO_6^-$.
\end{proof}

To have a more complete result, we clarify the relation between the splitness of \eqref{eq:ns} for $\CG$ and for its universal cover $\phi^\sco\colon \CG^\sco\to\CG$. The first important point to notice is that the proof of \Cref{Cor:Pf}(a) holds over the base ring $W_2(\BF)$ instead of $\BF$ with next to no changes. Recalling $\CZ'=\kernel\phi^\sco$, $Z':=\CZ'(\BF)$, we obtain short exact sequences
\[ 1\to \Fz'\times Z'\to \CG^\sco(W_2(\BF))\to E_2\to 1 \hbox{ and }1\to E_2 \to \CG(W_2(\BF))\to \Fz'\times Z'\to 1,\]
in which $E_2$ is defined as the image of $\CG^\sco(W_2(\BF))\to\CG(W_2(\BF))$ under $\phi^\sco$.

Let us now consider the following diagram of abstract groups, in which the group $E_3$ is defined as the pullback from row 4, and the $\pi_j$ are the maps they label
\[
\xymatrix@R-.5pc{
1\ar[r]& \Fg^\sco \ar[r] \ar@{=}[d]&\ar[d]^{\mathrm{mod\,}\Fz'\times Z'}E_1:=\CG^\sco(W_2(\BF))\ar[r]^-{\pi_{\CG^\sco}}&\ar[d]^{\mathrm{mod\,}Z'}\CG^\sco(\BF)\ar[r]&1\\
1\ar[r]& \ar@{^{ (}->}[d]\Fg_2:=\Fg^\sco/\Fz' \ar[r] &\ar@{^{ (}->}[d]E_2\ar[r]^{\pi_2}&H'_\BF\ar[r]\ar@{=}[d]&1\\
1\ar[r]& \Fg_3:=\Fg \ar@{=}[d]\ar[r] &\ar@{^{ (}->}[d]E_3\ar[r]^{\pi_3}& H'_\BF\ar[r]\ar@{^{ (}->}[d]&1\\
1\ar[r]& \Fg \ar[r] &E_4:=\CG(W_2(\BF))\ar[r]^-{\pi_4:=\pi_\CG}&\CG(\BF)\ar[r]&1\rlap{.}\\
}
\]
Note that each of the extensions in rows $1,\ldots,4$ defines a corresponding class $\gamma_i$ in second group cohomology such that row $i$ is split if and only $\gamma_i=0$: $\gamma_1\in H^2(\CG^\sco(\BF),\Fg^\sco)$, $\gamma_4\in H^2(\CG(\BF),\Fg)$, and $\gamma_i\in H^2(H'_\BF,\Fg_i)$ for $i=2,3$. One has the following result:
\begin{Lem}\label{Lem:N-S-and-Isogenies}
If \ns\ holds for $\CG$ it holds for $\CG^\sco$. If conversely \ns\ holds for $\CG^\sco$ and if $\CG^\sco$ is not in $\CE_{\pf}$ or in $\CE_{\sch}$, then \ns\ holds for $\CG$.
\end{Lem}
\begin{proof}
It is clear that if row 1 splits, then so does row 2, by passing to quotients. Also it is trivial that if row 2 splits, then so does row 3. Suppose that row 3 is split. Let $s\colon \CG(\BF)\to E_4$ be a set-theoretic splitting that defines the class $\gamma_4$. Then its restriction to $H'_\BF$ defines a set-theoretic splitting $H'_\BF\to E_3$, i.e., a representative of $\gamma_3$. By hypothesis, $\gamma_3=0$. Now the Hochschild-Serre spectral sequence for $1\to H_\BF\to \CG(\BF)\to Z'\to1$ degenerates with $\Fg$ coefficients, because $Z'$ is of order prime to $p$. Hence restriction is an isomorphism $H^2(\CG(\BF),\Fg)\to H^2(H'_\BF,\Fg)^{Z'}$, and it follows that $\gamma_4=0$. 

We now go in the opposite direction. It is clear that a splitting of row 4 restricts to a splitting of row 3. Suppose now that row 3 is split via some $s\colon H'_\BF\to E_3$. Composing this with the surjection $E_3\to E_3/E_2$ induces a homomorphism $H'_\BF\to E_3/E_2\cong \Fz'$. Note that the target is an elementary abelian $p$-group. As $\CG^\sco$ is not in $\CE_\pf$,  the map $H'_\BF\to E_3/E_2$ is zero by \Cref{Cor:H1Triv}. Hence $s$ takes values in~$E_2$.

Suppose finally that row 2 is split, via some homomorphism $s\colon H'_\BF\to E_2$. The splitting in row 2 gives a commuting diagram
\[\xymatrix@C+1.5pc@R-1pc{
&\Fz'\times Z'\ar[d]\\
\CG^\sco(\BF)\ar[dr]_-{s\circ(\mathrm{mod\,}\Fz'\times Z')\hbox{\phantom{mm}}}\ar@{-->}[r]& E_1\ar[d]\\
& E_2\rlap{.}
}\]
We first wish to lift to the dotted homomorphism. The obstruction lies in $H^2(\CG^\sco(\BF),\Fz'\times Z')$. Here $Z'$ is finite abelian and of order prime to $p$ and $\Fz'$ is finite abelian and $p$-torsion. As we assume that $\CG^\sco$ lies not in $\CE_{\sch}$, it follows from \Cref{Rem:SchurZ} that the obstruction vanishes. Let $s'\colon \CG^\sco(\BF)\to E_1$ be a homomorphism lifting $s$. Consider the map $\psi\colon \CG^\sco(\BF)\to \Fz'\times Z',g\mapsto s'(g)g^{-1}$. Since it takes values in the center $\Fz\times Z'$ of $\CG^\sco(W_2(\BF))$, it is easily seen to be a homomorphism. For the same reason $\psi^{-1}\circ s'\colon \CG^\sco(\BF)\to E_1$ is a homomorphism. It clearly lifts $s$ and hence row 1 is split. This completes the proof.
\end{proof}
The following result is immediate from \Cref{Thm:Vasiu} and \Cref{Lem:N-S-and-Isogenies}.
\begin{Cor}\label{Cor:n-s}
Suppose $\CG$ over $\WF$ is absolutely simple. Then \ns\ holds if and only if $(\CG,\BF)$ is not in the list $\CE_{\ns}$ given in \eqref{eq:n-s-List}.
\end{Cor}

\subsection{$z$-Lie-balanced groups}
\label{Subsec:Puniv}

In this subsection we introduce a particular kind of reductive group over $\WF$ that is frequently encountered in applications and 
for which the analog \Cref{Thm:StdHypForPuniv} of \Cref{Thm:MainOnStdSetup} has fewer exceptions, for instances for \liec\ or \lieu.

Let us begin by a construction inspired from \cite[p.~474]{Arthur} which uses ideas from $z$-extensions introduced by Langlands and Kottwitz. Let $\CH$ be a connected reductive group over $\WF$ whose special fiber $\CH_\BF$ is absolutely simple and such that the natural map $\Fh=\Lie\CH_\BF\to \Fh^\ad$ is an isomorphism. The kernel of the central isogeny $\CH^\sco\to\CH$ is described in \Cref{Eqn:Centers}. 
Let $\pi\colon \CH'\to\CH$ be a central isogeny such that $\CH^\sco\to\CH'$ is \'etale. This implies that the natural map $\Fh^\sco\to\Fh'$ is an isomorphism and that $Z(\CH)^o$ lies in $\kernel\pi$. Choose a homomorphism $\kernel \pi \into \CT$ over $\WF$ into a torus whose induced map $Z(\Fh') \to\Ft=\Lie(\CT_\BF)$ on Lie algebras is an isomorphism. Define $\CG':=(\CH'\times \CT)/\kernel\pi$ with $\kernel\pi$ embedded diagonally. Then one has a short exact sequence of reductive groups
\begin{equation}\label{eq:ses-LieUniv}
1\to \CT \to \CG' \to \CH'\to 1
\end{equation}
with the following properties:
\begin{enumerate}\advance\itemsep by -.4em
\item the extension is central,
\item the natural map $\CH'\to\CG^{\prime,\der}$ is an isomorphism and $\CH'$ is Lie-simply connected,
\item the image of $\CT$ in $\CG'$ is the identity component of the center of $\CG'$,
\item the Lie algebra $\Fg'=\Lie\CG'$ is isomorphic to $\wh\Fh$.
\item the induced sequence of Lie algebras $0\to \Ft\to\Fg'\to\Fh\to0$ is exact.
\end{enumerate}

\begin{Def}\label{Def:Puniv}
We call a connected reductive group $\CG'$ over $\WF$

\begin{enumerate}
\item {\em Lie-balanced} if it arises as a $\CG'$ via the above construction.
\item {\em $z$-Lie-balanced} if there exists a central homomorphism $\CG'\to\CG''$ to a Lie-balanced $\CG'$ with smooth kernel of multiplicative type such that the induced morphism $\CG^{\prime,\der}\to\CG^{\prime\prime,\der}$ is a central isogeny.
\end{enumerate}
\end{Def}
\begin{Ex}
The groups $\GL_n$ and $\GSp_n$ are $z$-Lie-balanced over $W(\BF)$ for any finite field $\BF$. They are Lie-balanced if $p$ divides $n$, for $\GL_n$, or if $p=2$, for $\GSp_n$. 

If $\CG$ is almost simple over $W(\BF)$ and if $\Fg$ is simple, cf.~\Cref{Lie:Simple}, then $\CG$ is Lie-balanced.
\end{Ex}

\begin{Lem}\label{Lem:UniveZsmooth}
If $\CG'$ is $z$-Lie-balanced, then the following hold:
\begin{enumerate}
\item The identity component $Z(\CG')^o$ of $Z(\CG')$ is a torus, and hence $Z(\CG')$ is smooth.
\item One has a short exact sequence $0\to \Ft\to \Lie\CG'\to \wh\Fh\to 0$ as Lie algebras and as $H_\BF$-modules for some $\Ft\subset\Lie\CG'$ which is an abelian sub Lie algebra and which carries a trivial $H_\BF$-action.
\end{enumerate}
\end{Lem}
\begin{proof}
Let $\pi\colon\CG'\to\CG''$ be a homomorphism as in the definition of $z$-Lie-balanced with $\CG''$ Lie-balanced. Because $\pi$ is central and surjective with smooth kernel of multiplicative type, we have a short exact sequence of centers $1\to \kernel\pi\to Z(\CG')\to Z(\CG'')\to 1$ with $\kernel\pi$ smooth of multiplicative type. By the hypothesis on $\CG''$ the group $Z^o(\CG'')$ is a torus. Therefore $Z(\CG')$ is smooth and of multiplicative type. This implies~(a).

To see (b) set $\Ft:=\Lie\kernel \pi$. Because $\pi$ is central the claimed properties on $\Ft$ are clear. Also the displayed sequence is clearly left exact. The right exactness follows from smoothness and dimension considerations: One clearly has $\dim_\Krull \CG'=\dim_\Krull\CG''+\dim_\Krull\kernel\pi$. Passing to tangent spaces at the identity and using the smoothness, we obtain $\dim_\BF\Lie\CG'=\dim_\BF\Lie\CG''+\dim_\BF\Ft$, and using $\Lie\CG''\cong\wh\Fh$, we are done.
\end{proof}

\begin{Cond}\label{Cond-Puniv}
The algebraic group $\CG'$ is $z$-Lie-balanced over $\BF$, $H'_\BF:=[\CG'(\BF),\CG'(\BF)]$, and $H_\BF$ is a subgroup of $\CG'(\BF)$ that contains $H'_\BF$. 
\end{Cond}

\begin{Thm}\label{Thm:StdHypForPuniv}
Suppose that $(\CG',\BF,H_\BF,H'_\BF)$ satisfies \Cref{Cond-Puniv}. Let $\CG:=\CG^{\prime,\der}$. Then the tuple also satisfies  \Cref{StandardHyp}, and the following hold:
\begin{enumerate}
\item 
\pf, \ct, \liecsc\ and \ns\ hold unless $($type$,\BF)$ for the group $\CG$ is in the list
\[ 
(A_1,\BF_?)_{?\in\{2,3\}},   ({}^2A_2,\BF_2),  ({}^2A_3,\BF_2),  (B_2,\BF_2), (G_2,\BF_2)
 .\]
\item \van\ holds unless $($type$,\BF)$ for $\CG$ is in the list
\[ (A_1,\BF_?)_{?\in\{2,3,5\}}, ({}^2A_2,\BF_2),  (C_n,\BF_?)_{n\ge2,?\in \{2,3,4,5,9\} },  (G_2,\BF_2)
\]
or $\CG^\der$ is non-split, and hence of types A, D or E${}_6$, and $|\BF|<4$.
\item Assuming $\CG^\sco(\BF)\notin \CE_{\pf}$, concerning \textbf{(lie-?)} the following hold:
\begin{enumerate}
\item  $\liegen$, if $\CG$ is of type $A_n$ and $p\nmid n+1$, or of type $B_n$,  $C_n$, $D_n$, $E_7$ or $F_4$ and $p\not=2$, or if $\CG$ is of type $E_6$ or $G_2$ and $p\not=3$, or if $\CG$ is of type~$E_8$.
\item $\lieu$, unless $\CG$ is of type $B_2$ or $F_4$ and $p=2$, or of type $G_2$ and $p=3$.
\item $\liec$, if $\CG$ is of type $A_n$, $n\ge2$, $D_n$ or $E_n$, or $\CG$ is of type $A_1$, $B_n$,  $C_n$, $F_4$ and $p\neq2$, or of  type $G_2$ and $p\neq3$.
\end{enumerate}
\end{enumerate}
\end{Thm}
\begin{proof} Part (a) of \Cref{StandardHyp} is clear, since $\CG'$ is connected reductive. To see \Cref{StandardHyp}(b) note that $H_\BF$ is generated by $H'_\BF=\CG(\BF)$ together with $Z(\CG')(\BF)$ and that the latter group is of order prime to~$p$.

(a): Conditions \pf, \liecsc\ and \ns\ only depend on $H'_\BF$ and $\CG$. Hence we can directly apply \Cref{Cor:Pf}(c), \Cref{Prop:csc} and \Cref{Cor:n-s}, noting that here the map $\Fg^\sco\to\Fg$ is an isomorphism. The list in (a) is the union of the relevant sublists of the exceptional lists $\CE_{\pf}$ and $\CE_{\ns}$ that occur in the two corollaries. It remains to see that \ct\ can only fail for groups in that list. That $Z(\CG')$ is smooth follows from \Cref{Lem:UniveZsmooth}(a). To see that $Z(\Lie\CG')\to H^0(H_\BF,\Fg')$ is an isomorphism, we consider the commutative diagram induced from \Cref{Lem:UniveZsmooth}(b)
\[ \xymatrix@R-1pc{
0\ar[r]& \Lie (\kernel\pi)\ar[r]\ar[d]&\Lie Z(\CG')\ar[r]\ar[d]&\Lie Z(\CG'')\ar[d]\\
0\ar[r]& H^0(H_\BF, \Ft)\ar[r]&H^0(H_\BF,\Fg')\ar[r]&H^0(H_\BF,\Fg'')\rlap{,}
}\] 
where $\pi\colon\CG'\to\CG''$ is a surjection as in \Cref{Def:Puniv}. The top row is exact by the proof of \Cref{Lem:UniveZsmooth}. The arrow on the left is an isomorphism, again by \Cref{Lem:UniveZsmooth}(b), and so it suffices to show that the arrow on the right is an isomorphism, i.e., we may assume that $\CG'=\CG''$. But then $\Lie\CG''=\wh\Fh$ for $\Fh=\Lie\CG^\der$, and now the isomorphism property follows from  \Cref{Thm:Lie1}(b) and the last line of \Cref{Rem:Exc-Center}.

(b): Using the sequence in \Cref{Lem:UniveZsmooth}(b) together with \Cref{Cor:H1Triv}, we may assume that $\CG'=\CG''$ for the proof; the exceptions ruled out in  \Cref{Cor:H1Triv} are part of the list in (b). From the definition of Lie-balanced, we know that $\Fg^\sco\to\Fg$ is an isomorphism. By \Cref{Rem:Van}, for \Cref{Gu-van}(a)$\Leftrightarrow$(b) we only need $\CG(\BF)\notin\CE_{\pf}$. For \Cref{Gu-van}(b) to hold, we need to avoid the pairs $(\CG,\BF)$ listed as exceptions in \Cref{Thm-OnVan}. The combined list of exceptions is that given~in~(b).

(c) The conditions \textbf{(lie-?)} only depend on $\CG$, which by construction is Lie-simply connected. Therefore (c) follows from  \Cref{Cor:LieCombined1},  \Cref{Cor:LieCombined2} and  \Cref{Cor:LieCombined3}.
\end{proof}
\begin{Rem}
It is possible to construct a sequence as in \eqref{eq:ses-LieUniv} also for cases where $\CH'$ is Lie-intermediate (and not only for Lie-simply connected cases, as done here). However in this situation the analog of \Cref{Thm:StdHypForPuniv}(b) on \van\ does not hold. Since this condition is crucial for our later applications, we did not pursue this here.
\end{Rem}
\begin{Cor}\label{Cor:BasicHypForGLn}
The tuple $(\GL_n,\BF,H_\BF,\SL_n(\BF))$ for $H_\BF$ a subgroup of $\GL_n(\BF)$ that contains $\SL_n(\BF)$ satisfies  \Cref{StandardHyp}, and the following hold:
\begin{enumerate}
\item 
\pf, \ct, \van, \liecsc\ and \ns\ hold unless $(n,\BF)$ is in the list $(2,\BF_?)_{?\in\{2,3,5\}}$.
\item {}\lieu\ holds unconditionally, \liec\ holds if and only if $(n,\Char\BF)\neq(2,2)$.
\end{enumerate}
\end{Cor}

\section{Preparations}
\label{Section3new}

Suppose throughout this section that \Cref{StandardHyp} holds for $(\CG,\BF,H_\BF,H'_\BF)$.

\smallskip

Our first result clarifies the definition of $H_R$ from \eqref{Def:HA-HpA} based on \Cref{Conv:MWF}.
\begin{Lem}\label{Lem-HRExists}
Let $R$ be in $\CAW$. Then the following hold:
\begin{enumerate}
\item there exist subgroups $M_R^o\subset\CG^o(R)$ and  $M_R\subset\CG(R)$ that under reduction modulo $\Fm_R$ map isomorphically to $M^o_\BF$ and $M_\BF$, respectively;
\item the possible $M^o_R$ from (a) form a conjugacy class under $\kernel\big(\CG^o(R)\to\CG^o(\BF)\big)$, and so do the $M_R$;
\item the groups $M^o_R$ and $M_R$ normalize $H'_R$;
\item \label{HRExists-PartD}the group $H^o_R:=H'_R M^o_R\subset \CG^o(R)$ is independent of the choice of~$M^o_R$;
\item if $\CG/\CG^o$ acts trivially on $\CG^o/\CG^\der$, then also $H_R=H'_R M_R\subset \CG(R)$ is independent of any choices.
\end{enumerate} 
\end{Lem}
\begin{proof}
Parts (a) and (b) follow from the profinite version of the Schur-Zassenhaus Theorem given in \cite[Prop.~2.1]{Boston-Explicit}:
Since $M^o_\BF$ has order prime to $p$, the cohomology groups $H^i(M^o_\BF,\Fg)$, $i\ge1$, vanish. Now the kernel of $\CG^o(R)\to\CG^o(\BF)$ is a pro-$p$ group. Therefore an inductive argument using $H^2(M^o_\BF,\Fg)=0$ shows the existence of a lift as in (a); its uniqueness up to conjugation, i.e.~(b), follows from $H^1(M^o_\BF,\Fg)=0$, again by induction. The same argument works for $M_R$ as well. Since $\CG^\der(R)$ is characteristic in $\CG(R)$, see \Cref{App-8}, it is normalized by $M_R^o$ and $M_R$. Because $M^o_\BF$ and $M_\BF$ normalize $H'_\BF$, part (c) follows.

For (d) note that by construction $H^o_R\cap \CG^\der(R)=H'_R$. Hence it suffices to show that 
\[H^o_R /H'_R\cong  \CG^\der(R) H^o_R/ \CG^\der(R)=\CG^\der(R)M_R^o/\CG^\der(R)\subset \CG^o(R)/\CG^\der(R)\] 
is independent of any choices. However $\CG^o(R)/\CG^\der(R)$ is abelian and the conjugation action of $M_R^o$ induced from (b) is thus trivial. This shows (d). The proof of (e) is analogous, using the hypotheses of (e) and again assertion~(b).
\end{proof}

The following lemma introduces in our setting an important substitute $H^c$ for the commutator subgroup $[H,H]$ for a closed subgroup $H$ of $H_R$ that surjects onto $H_\BF$ under reduction modulo~$\Fm_R$, and that takes the residual image $H_\BF$ into account.
\begin{Lem}\label{Lem-OnCommutator}
Let $R$ be in $\CAW$ and let $H\subset \CG(R)$ be a closed subgroup that surjects onto $H_\BF$. Then there exists a unique closed subgroup $H^c\subset H$ that contains the closure $\overline{[H,H]}$ of the commutator subgroup and for which $H^c/\overline{[H,H]}\to H_\BF/[H_\BF,H_\BF]$ is an isomorphism. 

Moreover if $I$ is an ideal of $R$ and if $\bar H$ denotes the image of $H$ in $H_{R/I}$, which is a closed subgroup, then under reduction modulo $I$ the group $H^c$ maps onto $\bar H^c$. In particular, the reduction of $H^c$ modulo $\Fm_R$ is $H_\BF$.
\end{Lem}
\begin{proof}
The group $H^\ab:=H/\overline{[H,H]}$ surjects onto $H^\ab_\BF:=H_\BF/[H_\BF,H_\BF]$ and the kernel is a pro-$p$ group. Since $H^\ab_\BF$ is of order prime to $p$, there exists a lift $H^{\ab,c}$ of $H^\ab_\BF$ to $H^\ab$ that is unique up to conjugation. Because $H^\ab$ is abelian, in fact $H^{\ab,c}$ is unique. Clearly $H^c$ must be the inverse image of $H^{\ab,c}$ under the canonical surjection $H\to H^\ab$. 

It remains to prove the second assertion. It is immediate that the reduction of $[H,H]$ modulo $I$ equals $[\bar H,\bar H]$. Consider the reduction maps
\[ H^\ab \longto \bar H^\ab:=\bar H/\overline{[\bar H,\bar H]} \longto H^\ab_\BF. \]
The above construction defines subgroups $H^{\ab,c}\subset H^\ab$ and $\bar H^{\ab,c}\subset  \bar H^\ab$. Since the image of  $H^{\ab,c}$ in $\bar H^\ab$ satisfies the properties required for $\bar H^{\ab,c}$, the uniqueness of $\bar H^{\ab,c}$ shows that $H^{\ab,c}$ maps isomorphically to $\bar H^{\ab,c}$ under reduction. This implies~(b).
\end{proof}

\begin{Lem}\label{Lem-GpsWithHc}
Let $R$ be in $\CAW$ and let $H\subset H_R$ be a closed subgroup that surjects onto $H_\BF$. If either \liegen(ii) holds, or if $H=H_R$ and \liecsc\ holds, then $H^c=H$.
\end{Lem}
\begin{proof}
By an inverse limit argument and \Cref{Lem-OnCommutator}, it suffices to prove the lemma for Artinian $R$. Here, by an inductive argument, and again by \Cref{Lem-OnCommutator}, it will suffice to prove the following: Let $I\subset R$ be an ideal such that $I\cong\BF$ as an $R$-algebra, let $\bar H$ be the image of $H$ in $H_{R/I}$. Then $\bar H^c=\bar H$ implies that $H^c=H$. To see this consider the diagram of short exact sequences
\[
\xymatrix@C-1pc{
0\ar[r]&N\ar[rr]\ar@{^{ (}->}[d]&&H\ar[rr]\ar@{^{ (}->}[d]^\iota&&\bar H\ar@{^{ (}->}[d]\ar[rr]&&0\\
0\ar[r]&\Fg^\der\otimes I\cong\Fg^\der\ar[rr]&&H_R\ar[rr]&\hbox{\phantom{m}}&H_{R/I}\ar[rr]&&0\rlap{,}\\}
\]
where $N$ is the kernel of $H\to \bar H$. If $N$ is the trivial group, then $H\to \bar H$ is an isomorphism, and the result is clear. Suppose now that $N$ is non-trivial. Then either by \liegen(ii), or by hypotheses \liecsc\ if $H=H_R$, we must have $N=\Fg^\der$. In either case, $H_0(H'_\BF,\Fg^\der)=0$, i.e., $\Fg^\der$ is the $\BF_p$-span of $\{gX-X\mid X\in \Fg^\der, g\in H'_\BF\}$. This implies $N=[N,H]\subset [H,H]$. Thus $H^c$ contains $N$. Since $\bar H^c=\bar H$ by hypothesis, we deduce $H^c=H$.
\end{proof}
\begin{Rem}
Suppose that as an $\BF_p[H'_\BF]$-module no Jordan-H\"older factor of $\Fg^\der$ is trivial, i.e., that $\Hom_{\BF_p[H'_\BF]}(N,\BF_p)=0$ for all submodules $N$ of $\Fg^\der$. Then the above argument shows that $H=H^c$ for any closed subgroup $H\subset H_R$. 

If on the other hand there exists an $\BF_p[H'_\BF]$-submodule $N$ with a non-zero $\BF_p[H'_\BF]$-homomor\-phism $N\to\BF_p$, then $N$ is a normal subgroup of $H_{\BF[\eps]}$, and $H:=NH_{\BF}$ is a subgroup of $H_{\BF[\eps]}$ with $H^c\subsetneq H$.
\end{Rem}

\begin{Lem}\label{BostonLemmaExt-New}
Consider a diagram with exact rows
\begin{equation}\label{Diag:BostonExt}
\xymatrix{
0\ar[r]&N\ar[r]\ar@{^{ (}->}[d]&H\ar[r]\ar@{^{ (}->}[d]^\iota&H_\BF\ar[r]\ar@{=}[d]&0\\
0\ar[r]&\Fg^\der\ar[r]&H_{\BF[\eps]}\ar[r]&\ar@/_1pc/[l]_\sigma H_\BF\ar[r]&0\\}
\end{equation}
with $\iota$ an inclusion and $\sigma$ the canonical splitting. Then the following hold:
\begin{enumerate}
\item Suppose $N=0$. If \van\ holds,  then $\iota$ is conjugate to $\sigma$ via an element of~$\Fg^\der$.
\item Suppose $N\neq0$. Suppose either (1) that \liegen(ii) holds, or (2) that \liec(ii) and \sch\ hold, and that $H=H^c$. Then  $H=H_{\BF[\eps]}$.
\end{enumerate}
\end{Lem}
\begin{proof}
If $N=0$, then $\iota$ is a splitting of the bottom sequence and by \van\ all such splittings are conjugate by an element in $\Fg^\der$. Thus it remains to prove (b), and we assume $N\neq0$. Under condition \liegen(ii), we deduce $N=\Fg^\der$, and we are done.

Suppose now in (b) that (2) holds. Because of \liec(ii), we have to rule out that $N$ lies in $\kernel(\Fg^\der\to\overline\Fg^\der)$. Assume on the contrary that $N$ lies in this kernel, so that $N\cong\BF_p^r$ for some $r\ge1$. As $H_\BF/H'_\BF$ is of order prime to $p$,  the Hochschild-Serre spectral sequence yields $H^2(H_\BF,N)\cong H^2(H'_\BF,N)^{H_\BF/H'_\BF}$, and the latter is zero by \sch. Hence $H$ is a semidirect product $N\rtimes H_\BF$. By hypothesis, $H_\BF$ acts trivially on $N$, and thus $H=N\times H_\BF$. One computes $[H,H]\subset\sigma(H_\BF)\subsetneq H$, and this contradicts $H=H^c$. Thus we must have $H=H_{\BF[\eps]}$.
\end{proof}

\begin{Lem}\label{Lem-BostonForW2}
Consider the diagram with exact rows
\[
\xymatrix{
0\ar[r]&N\ar[r]\ar@{^{ (}->}[d]&H\ar[r]\ar@{^{ (}->}[d]^\iota&H_\BF\ar[r]\ar@{=}[d]&0\\
0\ar[r]&\Fg^\der\ar[r]&H_{W_2(\BF)}\ar[r]&H_\BF\ar[r]&0\\}
\]
with $\iota$ the inclusion. Assume \ns\ and either \liegen(ii) or \liec(ii) and \sch. Then $H=H_{W_2(\BF)}$.
\end{Lem}
\begin{proof}
Because of \ns\ the subgroup $N$ must be non-trivial. If in addition \liegen(ii) holds, then $N=\Fg^\der$ because $\Fg^\der$ is irreducible as an $\BF_p[H'_\BF]$-module, and thus $\iota$ is an isomorphism. 

Suppose now that \liec(ii) and \sch\ hold. We immediately deduce from \sch\ and \ns\ that $N$ cannot be a submodule of $\kernel(\Fg^\der\to\overline\Fg^\der)$, and then from \liec(ii) that $N$ surjects onto the cosocle of $\Fg^\der$. But as a submodule of $\Fg^\der$ that surjects onto the cocoscle of $\Fg^\der$, we must have $N=\Fg^\der$. This completes the proof.
\end{proof}

\begin{Ex}\label{Ex-PGL}
If $\Fg^\der$ possesses a non-trivial $\BF_p[H'_\BF]$-homomorphism to $\BF_p$, then one cannot expect the conclusion of \Cref{BostonLemmaExt-New}(b) or of \Cref{Lem-BostonForW2} to hold. For a concrete example, let $\CG=\PGL_p$ and $H_\BF=\PGL_p(\BF)$ with $p=\Char\BF$ and $\#\BF\notin\{2,3,4,5,9\}$, so that \sch\ and \ns\ hold. One has a surjective homomorphism $\pdet\colon\PGL_p(R)\to R^\times/R^{\times p}$ induced from $\det$. Let $R\in\{W_2(\BF),\BF[\eps]\}$ and define $H\subset  \PGL_p(\BF[\eps])$ as the kernel of $\pdet\colon\PGL_p(R)\to R^\times/R^{\times p}\cong(\BF,+)$. Then $H$ surjects onto $H_\BF$. But $H$ is properly contained in $\PGL_p(R)$, it is a proper extension of $H_\BF$, and it satisfies $H=H_\BF[H,H]$.
\end{Ex}

For the following result, let $\CG^{[i]}(R)$ denote the kernel of $\CG^\der(R)\to\CG^\der(R/\Fm^i_R)$ for $i\ge1$ and $R\in\CAW$, and recall from \cite[Sect.~6]{Pink-Compact} that one has natural exponential maps giving isomorphisms $\Fg^\der\otimes_\BF\Fm_R^{i}/\Fm_R^{i+1}\stackrel\simeq\to \CG^{[i]}(R)/\CG^{[i+1]}(R)$ for any $i\ge1$.
\begin{Lem}[{cf.~\cite[Appendix, Prop.~2]{Boston}}]\label{Lem-BostonFromM2}
Let $R$ be an Artinian ring in $\CAW$. Assume that $\Fm_R^{n+1}=0$ and $\Fm_R^{n}\neq0$ for some $n\ge2$. Consider a diagram 
\[
\xymatrix@C-1pc{
0\ar[r]&N\ar[rr]\ar@{^{ (}->}[d]&&H\ar[rr]\ar@{^{ (}->}[d]^\iota&&H_{R/\Fm_R^n}\ar[rr]\ar@{=}[d]&&0\\
0\ar[r]&\Fg^\der\otimes_\BF \Fm_R^n\ar[rr]&&H_R\ar[rr]&\hbox{\phantom{m}}&H_{R/\Fm_R^n}\ar[rr]&&0\\}
\]
with exact rows and $\iota$ the inclusion. Then the following hold:
\begin{enumerate}
\item The intersection of $N$ with the commutator subgroup $[\CG^{[n-1]}(R)\cap H,\CG^{[1]}(R)\cap H]$ contains $[\Fg^\der,\Fg^\der]\otimes_\BF\Fm_R^n$.
\item If \liec(i) holds, then $H=H_R$.
\end{enumerate}
\end{Lem}
\begin{proof}
Clearly the multiplication map $\Fm_R^{n-1}\otimes \Fm_R \to \Fm_R^{n}$ factors via \[\Fm_R^{n-1}/\Fm_R^n\otimes \Fm_R/\Fm_R^2\to \Fm_R^{n}/\Fm_R^{n+1}.\] It follows that the commutator map
\[ \CG^{[n-1]}(R)\times\CG^{[1]}(R) \to\CG^{[n]}(R), (g,h)\mapsto ghg^{-1}h^{-1}\]
factors via 
\[\CG^{[n-1]}(R)/\CG^{[n]}(R)\times\CG^{[1]}(R)/\CG^{[2]}(R)\cong \Fg^\der\otimes_\BF\Fm_R^{n-1}/\Fm_R^n\times\Fg^\der\otimes_\BF\Fm_R^{1}/\Fm_R^2 \to\CG^{[n]}(R) .\]
For $X,Y\in\Fg^\der$ and $f\in \Fm_R^{n-1}$, $g\in \Fm_R$, the last map is given by 
\[(X\otimes f,Y\otimes g)\mapsto [X,Y]\otimes fg\in \Fg^\der\otimes_\BF\Fm_R^n \cong \CG^{[n]}(R).\] 
Now $H\cap \CG^{[i]}/H\cap\CG^{[i+1]}$ is naturally isomorphic to $\CG^{[i]}/\CG^{[i+1]}$ for $i=1,\ldots,n-1$, since $H$ surjects onto $H_{R/\Fm_R^n}$. Hence forming (products of) commutators in $H$, we see that the commutator subgroup $[\CG^{[n-1]}(R)\cap H,\CG^{[1]}(R)\cap H]$ must contain $[\Fg^\der,\Fg^\der]\otimes_\BF\Fm_R^n\subset \Fg^\der\otimes_\BF\Fm_R^n \cong \CG^{[n]}(R)$. This proves (a). Part (b) is obvious since \liec(i) asserts that $[\Fg^\der,\Fg^\der]=\Fg^\der$.
\end{proof}

For any $R\in\CAW$ denote by $\pi_R$ the canonical reduction $H_R\to H_\BF$, and for any ideal $I$ of $R$ by $\pi_{R,I}$ the canonical reduction $H_R\to H_{R/I}$.
\begin{Cor}\label{Cor-BostonGenl}
Let $R$ in $\CAW$ be Artinian and let $H\subset H_R$ be a subgroup such that $\pi_R(H)=H_\BF$. Suppose that \liec(i) holds. Then $H=H_R$ assuming one of the following:
\begin{enumerate}
\item one has $\pi_{R,\Fm_R^2}(H)=H_{R/\Fm_R^2}$;
\item one has $\pi_{R,(p,\Fm_R^2)}(H)=H_{R/(p,\Fm_R^2)}$ and condition \ns\ and either \liegen\ or \liec(ii) and \sch\ hold. 
\end{enumerate}
\end{Cor}
\begin{proof}
Let us first show that (b) implies the condition in (a). For this, for simplicity of notation, we may assume that $\Fm^2_R=0$, and we need to show that then $H=H_R$. If $p=0$ in $R$ then there is nothing to show. Otherwise $R\cong W_2(\BF)[x_1,\ldots,x_n]/(p,x_1,\ldots,x_n)^2$ for some $n$, and there are surjective ring homomorphisms onto $R/(p,\Fm_R^2)$ and onto $W_2(\BF)$. The kernels of the induced group homomorphisms $\pi_{R,(p)}\colon H_{R}\to H_{R/(p,\Fm_R^2)}$ and $\pi_{R,(x_1,\ldots,x_n)}\colon H_{R}\to H_{W_2(\BF)}$ intersect trivially. Now the restriction of $\pi_{R,(p)}$ to $H$ is surjective by hypothesis and that of $\pi_{R,(x_1,\ldots,x_n)}$ to $H$ by \Cref{Lem-BostonForW2}, using the hypotheses in (b). This implies $H=H_R$, as had to be shown.

To deduce $H=H_R$ from (a) one proceeds by induction over rings $R$ such that $\Fm_R^n=0$, starting at $n=2$, and one applies \Cref{Lem-BostonFromM2}(b) in the induction step.
\end{proof}
\begin{Cor}\label{Cor-FepsFactorsViaRmtwo}
Let $R$ be an Artinian ring in $\CAW$. Let $\phi\colon H_R\to H_{\BF[\eps]}$ be any group homomorphism such that $\pi_{\BF[\eps]}\circ\phi=\pi_R$. Suppose that \lieu\ holds. Then $\phi$ factors via $\pi_{R,\Fm_R^2}\colon H_R\to H_{R/\Fm_R^2}$.

If in addition condition \ns\ holds, then $\phi$ factors via $\pi_{R,(p,\Fm_R^2)}\colon H_R\to H_{R/(p,\Fm_R^2)}$.
\end{Cor}
\begin{proof}
For the first part we induct on $n\ge2$ and rings $R$ such that $\Fm_R^n=0$. 
The case $n=2$ is clear by hypothesis. In the induction step we assume that $\Fm_R^{n+1}=0$. We claim that $\phi$ factors via $\pi_{R,\Fm_R^n}\colon H_R\to H_{R/\Fm_R^n}$, and by induction this implies the result. We have the commutative diagram
\[
\xymatrix{
0\ar[r]&\CG^{[1]}(R)\ar[r]\ar[d]&H_R\ar[r]\ar[d]^\phi&H_\BF\ar[r]\ar@{=}[d]&0\\
0\ar[r]&\Fg^\der \ar[r]&H_{\BF[\eps]}\ar[r]&H_{\BF}\ar[r]&0\rlap{.}\\}
\]
Because $\Fg^\der$ is abelian as a group under $+$, the group $[\CG^{[1]}(R),\CG^{[1]}(R)]$ lies in the kernel of $\phi$. Hence by \Cref{Lem-BostonFromM2}(a), the restriction of $\phi$ to $\Fg^\der\otimes_\BF\Fm_R^n \cong \CG^{[n]}(R)$ contains $[\Fg^\der,\Fg^\der]\otimes_\BF\Fm_R^n$ in its kernel. By \lieu(ii), we know that any $\BF_p[H'_\BF]$-module homomorphism $\Fg^\der\otimes_\BF\Fm_R^n\to\Fg^\der$ is either surjective or trivial. If we apply this to $\phi$ and observe that \lieu(i) holds, then we see that $\phi$ restricted to $\CG^{[n]}(R)$ is zero, and this proves the claim.

For the second part, we may assume $\Fm_R^2=0$ and in addition that $p\notin\Fm_R^2$, since otherwise there is nothing to show. The embedding $W_2(\BF)\to R$ gives the subgroup $H_{W_2(\BF)}\subset H_R$. Now assume that $\phi(\kernel\pi_{W_2(\BF)})\neq0$, else we are done. Then $\phi$ restricted to $H_{W_2(\BF)}$ defines a non-zero homomorphism from $\Fg^\der\cong \kernel\pi_{W_2(\BF)}$ to $\Fg^\der\cong \kernel\pi_{\BF[\eps]}$. By \lieu(ii), it follows that any such homomorphism is multiplication by a scalar in $\BF$, and since the map is non-zero the scalar must be non-trivial. This implies that the restriction of $\phi$ to $H_{W_2(\BF)}$ defines an isomorphism $H_{W_2(\BF)}\to H_{\BF[\eps]}$. But this is absurd since the extension $H_{\BF[\eps]}\to H_\BF$ is split while $H_{W_2(\BF)}\to H_\BF$ is non-split by~\ns.
\end{proof}
\begin{Rem}
If there is a splitting of $H_{W_2(\BF)}\to H_\BF$ and if \van\ holds, then all split extensions of $H_\BF$ by $\Fg^\der$ are equivalent, and in particular $H_{W_2(\BF)}\cong H_{\BF[\eps]}$.
\end{Rem}

\section{Rings as universal deformation rings}
\label{Section3-new}

Throughout this section, we assume that $\CG$ is connected and that $(\CG,\BF,H_\BF,H'_\BF)$ satisfies \Cref{StandardHyp}, which in the present case means that $H_\BF=M_\BF H'_\BF$ with $M_\BF\subset\CG(\BF)$ of order prime to $p$. From now on we fix some $R\in \CAW$. Then $H_R=H_R^o=M_RH'_R$ from \Cref{Lem-HRExists} is independent of any choices. 

Recall $\bar\rho_R$, $D_{\barrhoR}$ and $R_{\barrhoR}$ from \Cref{Def:UDefRing} and the paragraphs preceding it. Concerning the image of $\rho_A\in D_{\barrhoR}(A)$, let us first observe the following, which is trivial if $\CG=\CG^\der$.
\begin{Lem}\label{Lem-DRimageInHprimeR}
For an $H_\BF$-perfect closed subgroup $H$ of $H_R$, an $A\in\CAW$ and $[\rho_A]\in D_\barrhoR(A)$ with representative $\rho_A$, one has $\rho_A(H)\subset H_A$. If moreover \liecsc\ holds, then $\rho_A(H_R)\subset H_A$.
\end{Lem}
\begin{proof}
Since $\CG=\CG^o$, the composition of $\rho_A$ with the canonical surjection $\pi\colon\CG(A)\to \CG(A)/\CG^\der(A)$ has abelian image. Therefore $H\to \pi\circ\rho_A(H)$ factors via $H/\overline{[H,H]}$, which by the construction of $H^c$ in \Cref{Lem-OnCommutator} and our hypothesis $H=H^c$ is isomorphic to $H_\BF/[H_\BF,H_\BF]$. By \Cref{StandardHyp}, the latter group is of order prime to $p$. Since $H'_\BF$ maps to $\CG^\der(\BF)$ under $\bar\rho$, the restriction of $ \pi\circ\rho_A$ to $H$ will factor via $M_\BF/M_\BF\cap \CG^\der(\BF)$. As in the proof of  \Cref{Lem-HRExists}\ref{HRExists-PartD} we find that $\pi\circ\rho_A(H)$ lies in the unique lift of $M_\BF/M_\BF\cap \CG^\der(\BF)$ to $\CG(R)/\CG^\der(R)$. Since $H$ maps to $H_\BF$ under $\barrhoR$, we deduce $\rho_A(H)\subset H_A$. For the last claim note that under \liecsc\ we have $H_R^c=H_R$ by \Cref{Lem-GpsWithHc}.
\end{proof}

We have the following generalization of \cite{EM,Dorobisz} from $\GL_n$ to arbitrary $\CG$, under the same basic hypotheses; cf.~\Cref{Rem:Dorob-EM}. 
\begin{Thm}\label{Thm-Thm1}
Suppose $(\CG,\BF,H_\BF,H'_\BF)$ satisfies conditions \ct, \van, \ns, \lieu\ and one of \liecsc\ or $\CG=\CG^\der$. Then the canonical inclusion $\iota\colon H_R\to \CG(R)$ represents the universal deformation of $D_\barrhoR$, and in particular $R_\barrhoR= R$.
\end{Thm}
Recall that \lieu\ and \liecsc\ are implied by \liegen(ii) and~(iii). 
\begin{Rem}
For tuples $(\CG,\BF,H_\BF,H'_\BF)$ as in \Cref{Thm:MainOnStdSetup},  with $\CG$ absolutely simple, condition \van\ and \ct\ can only hold under \liegen\ and \liead. For tuples $(\CG,\BF,H_\BF,H'_\BF)$ as in \Cref{Thm:StdHypForPuniv}, where $\CG$ is connected reductive but $\CG^\der$ is absolutely simple, conditions \van\ and \ct\ are compatible with \liesc.
\end{Rem}

From \Cref{Thm-Thm1}  and \Cref{Thm:MainOnStdSetup} we deduce:
\begin{Cor}\label{Rem:5.4-forAbsSimple}
If $(\CG,\BF,H_\BF,H'_\BF)$ satisfies \Cref{Cond:StdSetup}, then $\iota\colon H_R\to \CG(R)$ represents the universal deformation of $D_\barrhoR$ under the conjunction of the following conditions:
\begin{enumerate}
\item $(\CG,\BF)$ is not exceptional in the sense of \Cref{Not:ExcSet}, and $($type $\CG,\BF)\notin\{(A_1,\BF_5)\}$.
\item $\CG$ is of Lie-adjoint type.
\item If type $\CG =C_n$, then $|\BF|\notin\{2,3,4,5,9\}$,
\item If $\CG$ is non-split (and hence of types A, D or E${}_6$), then $|\BF|\ge4$.
\item If $\CG$ is of type $B_2$ or $F_4$ then $p\neq2$, if $\CG$ is of type $G_2$, then $p\neq3$.
\end{enumerate}
\end{Cor}

From \Cref{Thm-Thm1}  and  \Cref{Thm:StdHypForPuniv} we deduce:
\begin{Cor}\label{Rem:5.4-pUniversal}
If $(\CG,\BF,H_\BF,H'_\BF)$ satisfies \Cref{Cond-Puniv}, then $\iota\colon H_R\to \CG(R)$ represents the universal deformation of $D_\barrhoR$ under the conjunction of the following conditions:
\begin{enumerate}
\item $($type $\CG,\BF)\notin\{(A_1,\BF_?)_{?\in\{2,3,5\}}, (C_n,\BF_?)_{n\ge2,?\in \{2,3,4,5,9\} }\}$.
\item If $\CG^\der$ is non-split, then $|\BF|\ge4$.
\item If $\CG$ is of type $B_2$ or $F_4$ then $p\neq2$, if $\CG$ is of type $G_2$, then $p\neq3$.
\end{enumerate}
\end{Cor}

\begin{Rem}\label{Rem:Dorob-EM}
For $\CG=\GL_n$ and $H_\BF=H'_\BF=\SL_n(\BF)$, \Cref{Rem:5.4-pUniversal} shows that \Cref{Thm-Thm1} completely recovers \cite[Theorem~1]{Dorobisz} and \cite[Main Theorem]{EM}. 
\end{Rem}

The proof of \Cref{Thm-Thm1} consists of the following main steps. First we compute the dimension of the mod $p$ tangent space of $R_\barrhoR$. From this deformation theory implies \Cref{Thm-Thm1} in an elementary way in the case where $R$ is formally smooth over $\WF$. Finally we reduce the case of general $R$ to that of formally smooth~$R$.

\begin{Lem}\label{Cor-TangSpaceDim}
Suppose $(\CG,\BF,H_\BF,H'_\BF)$ satisfies conditions \ct, \van, \ns, \lieu\ and one of \liecsc\ or $\CG=\CG^\der$. Then the mod $p$ tangent spaces of $R$ and $ R_\barrhoR$ have the same dimension.
\end{Lem}

\begin{proof}
Define $d:= \dim_\BF \Fm_R/(p,\Fm_R^2)$, and denote by $d'$ the dimension of the mod $p$ tangent space of $R_\barrhoR$. From \Cref{Cor-FepsFactorsViaRmtwo} and also \Cref{Lem-DRimageInHprimeR}, which requires  \ns\ and \lieu\ and \liecsc\ if $\CG\neq\CG^\der$, we deduce that any deformation $[\rho_{\BF[\eps]}]$ in $D_\barrhoR(\BF[\eps])$ factors via $H_{R/(p,\Fm_R^2)}$, and hence that $d'=\dim_\BF H^1(H_{R/(p,\Fm_R^2)}, \Fg)$. Now the inflation restriction sequence of group cohomology yields
\begin{equation}\label{Eq:Inf-Res-Dim}
0\to H^1(H_\BF,\Fg)\to H^1(H_{R/(p,\Fm_R^2)}, \Fg)\to H^1(\CG^{\der,[1]}(R/(p,\Fm_R^2)),\Fg)^{H_\BF}
\end{equation}
Because $M_\BF$ is of order prime to $p$, and again by inflation restriction, the left term $H^1(H_\BF,\Fg)$ is isomorphic to $H^1(H'_\BF,\Fg)^{M_\BF/(M_\BF\cap H'_\BF)}$, and thus zero by \van. Moreover the abelian group $\CG^{\der,[1]}(R/(p,\Fm_R^2))$, which as an $H_\BF$-module is isomorphic to $\Fg^\der\otimes\Fm_R/(p,\Fm_R^2)$, acts trivially on $\Fg$, and so the right hand side can be identified with
\[  \Hom_{\BF_p[H_\BF]} ( (\Fg^\der)^d,\Fg)=\Hom_{\BF_p[H_\BF]} ( \Fg^\der,\Fg)^d.\]
If $\CG=\CG^\der$, then $\Fg=\Fg^\der$. Otherwise \liecsc\ yields $ \Hom_{\BF_p[H'_\BF]}(\Fg^\der,\BF_p)=0$, so that we have $\Hom_{\BF_p[H'_\BF]}(\Fg^\der,\Fg)=\Hom_{\BF_p[H'_\BF]}(\Fg^\der,\Fg^\der)$. We conclude $d'\le d$ from \lieu(ii).

For the converse inequality, consider the diagram 
\begin{equation}
\label{AlphaDiag}
\xymatrix{ H_R \ar[dr]_\iota \ar[rr]^{\rho_\barrhoR}&&\CG(R_\barrhoR) \ar[dl]^\alpha\\
&\CG(R)&\\}
\end{equation}
provided by the universality of $R_\barrhoR$, that commutes up to conjugation by $\kernel(\CG(R)\to\CG(\BF))$, and for some unique homomorphism $\alpha\colon R_\barrhoR\to R$ in $\CAW$. Recall that $R$ is the ring used to define $H_R$. Denote by $\alpha_2$ the homomorphism $R_\barrhoR/(p,\Fm_{R_\barrhoR}^2) \longto R/(p,\Fm_R^2)$  induced from $\alpha$. Then $H_{R/(p,\Fm_R^2)}$ is the image of $H_R$ in $\CG(R/(p,\Fm_R^2))$, and thus it must lie in the image of the homomorphism
\[ \CG(R_\barrhoR/(p,\Fm_{R_\barrhoR}^2)) \longto \CG(R/(p,\Fm_R^2))\]
induced from $\alpha_2$. Since $\pi_{R,(p,\Fm_R^2)}\colon H_R\to H_{R/(p,\Fm_R^2)}$ is surjective, it follows that $\alpha_2$ must be surjective. This yields $d'\ge d$ by comparing cardinalities.
\end{proof}

\begin{Cor}\label{Cor-Thm1ForFree}
\Cref{Thm-Thm1} holds for $R=\WF[[x_1,\ldots,x_d]]$ for any integer $d\ge0$.
\end{Cor}

\begin{proof}
Consider diagram \Cref{AlphaDiag}. It implies that the homomorphism 
\[ \alpha\colon   R_\barrhoR \longto R=\WF[[x_1,\ldots,x_d]]\]
is surjective, as it is surjective on mod $p$ tangent spaces. At the same time, we know from  \Cref{Cor-TangSpaceDim} that $d$ is the dimension of the mod $p$ tangent space of $R_\barrhoR$. Hence by the Cohen structure theorem and Nakayama's Lemma we have a surjective homomorphism 
\[ \beta \colon  R=\WF[[x_1,\ldots,x_d]]  \longto   R_\barrhoR.\]
For dimension reasons it follows that the composite $\alpha\circ\beta$ must have trivial kernel, so that $\beta$ is an isomorphism. But then the same argument shows that $\alpha$ is an isomorphism. This completes the proof of \Cref{Cor-Thm1ForFree}.
\end{proof}

\begin{proof}[Proof of \Cref{Thm-Thm1}]
Choose a surjective ring homomorphism 
\[\pi\colon S:=\WF[[x_1,\ldots,x_d]]\longto R\] 
in $\CAW$, denote the induced map $\CG(S)\longto \CG( R )$ by $\CG(\pi)$. Let now $A\in \CAW$ and let $\rho_A$ represent a $\CG$-valued deformation of $\barrhoR$ to $A$. Then $\rho_A\circ \CG(\pi)$ is a $\CG$-valued deformation of $\barrhoS:=\barrhoR\circ \CG(\pi)\colon H_S\to \CG(\BF)$. We consider the diagram
\begin{equation}
\label{AlphaGSDiag}
\xymatrix{ H_S \ar[dr]_{\CG(\alpha)} \ar[rr]^{\CG(\pi)}&&H_R \ar[dl]^{\rho_A}\\
&\CG(A)&\\}
\end{equation}
where $\CG(\alpha)\colon H_S\longto \CG(A )$ is induced from a unique homomorphism $\alpha\colon S\to A$ in $\CAW$ using the universality of the inclusion $H_S\to \CG(S)$ for $\CG$-deformations of $\barrhoS$ established in \Cref{Cor-Thm1ForFree}. Note that a priori, the diagram commutes only up to strict equivalence, i.a., up to conjugation by an element of $\CG(A)$ that surjects to the identity in $\CG(\BF)$. However by replacing the group homomorphism $\rho_A$ by a suitable conjugate, we can assume from now on that \Cref{AlphaGSDiag} commutes.

A priori, $\rho_A$ is only a group homomorphism. The main point we need to establish is that it is induced from a unique ring homomorphism $R\to A$ in $\CAW$. We first show that $\ker\pi\subset \ker\alpha$. For this consider $s\in S\setminus \kernel \alpha$. We need to show that $\pi(s)\neq0$. Let $\xi$ be a root of $\CG$ with root group $\CU_\xi\subset\CG$ and let $x_\xi\colon \BG_a\to \CU_\xi$ be an isomorphism, all defined over $\WF$. Then $\CG(\alpha)(x_\xi(s))=x_\xi(\alpha(s))$ is non-zero. This implies that $\CG(\pi)(x_\xi(s))=x_\xi(\pi(s))$ is non-zero. 

By the previous paragraph we have $\ker\pi\subset \ker\alpha$. This gives a factorization $\alpha=\beta\pi$ for a unique ring homomorphism $\beta\colon R\to A$, by the homomorphism theorem for rings. Hence $\CG(\alpha)=\CG(\beta)\circ\CG(\pi)$. Because $\CG(\pi)\colon H_S\to H_R$ is a surjective homomorphism of groups, by the homomorphism theorem for groups, we have $\rho_A=\CG(\beta)$, and this can only hold for a unique ring homomorphism $\beta$ in $\CAW$. Thus we have directly established the universal property of~$R$, i.e., that $R\cong R_{\bar\rho_R}$. 
\end{proof}

\begin{Rem}
Suppose that \lieu\ and \liecsc\ hold on the structure of $\Fg$ as a Lie algebra and as an $H'_\BF$-module. Suppose also that $(\CG,\BF,H_\BF,H'_\BF)$ satisfies either \Cref{Cond:StdSetup} or \Cref{Cond-Puniv}, and that $H_\BF=H'_\BF$. If then the assertion of \Cref{Cor-Thm1ForFree} holds for $d=0$, then one can deduce from \eqref{Eq:Inf-Res-Dim}, running backward the argument of \Cref{Cor-TangSpaceDim}, that $H^1(H_\BF',\Fg)=0$. This means that the direct proofs of \Cref{Thm-Thm1} for $\CG=\GL_n$ given in \cite{Dorobisz} and \cite{EM} essentially also reprove, without making this explicit, the vanishing of $H^1(\CG^\der(\BF),\Fg)$ for $\CG=\GL_n$ -- except for some small values of $\#\BF$ and~$n$.
\end{Rem}

\section{Closed subgroups of $H_R\subset\CG(R)$}
\label{Section4-new}

Throughout this section, we assume that $(\CG,\BF,H_\BF,H'_\BF)$ satisfies \Cref{StandardHyp}. By $R$ we denote a ring in $\CAW$ and by $H\subset H_R$ a closed subgroup that under $\pi_R$ surjects onto $H_\BF$. We advice the reader to recall the $H_\BF$-perfection of $H$ defined in \Cref{DefHPerfection} and the discussion around it. Note in addition that if \liegen(ii) holds, then $H=H^{(\infty)}$ by \Cref{Lem-GpsWithHc}.

The following result was motivated by \cite{Manoharmayum} and \cite{Boston}. It sheds some light on the structure of certain closed subgroups of $\CG(R)$ for $R\in\CAW$, which might be useful for studying images of Galois representations attached to automorphic forms. 

\begin{Thm}\label{Thm-Thm2}
Let $H\subset H_R$ be a closed subgroup that surjects onto $H_\BF$. Suppose that \ct, \ns\ and \van\ hold, and that either \liegen\ holds or that \liec\ and \sch\ hold. Then there exists a closed $\WF$-subalgebra $A$ of $R$ such that $H^{(\infty)}$ is conjugate to $H_A\subset\CG(R)$.
\end{Thm}
Observe that if $R$ is noetherian, then $A$ need not share this property. Its tangent space could be infinite dimensional. However if $H$ surjects onto $H_\BF$ and is open in $H_R$, so that $H$ contains $\CG^{[i]}$ for some $i\ge1$, then following the proof of \Cref{Lem-GpsWithHc}, $H^c$ still contains $\CG^{[i]}$, and hence so does $H^{(\infty)}$. From this one easily deduces that in this case $A$ is noetherian if $R$ is so, and hence that $A\in \CAW$.

\medskip

The main idea for the proof of \Cref{Thm-Thm2} is to let deformation theory determine the sought for ring $A$. For this we may, and from now on will, assume that $H$ is $H_\BF$-perfect. Since the inclusion $\iota\colon H\subset \CG(R)$ factors via the universal pair, we have a diagram
\begin{equation}
\label{AlphaHDiag}
\xymatrix{ H \ar[dr]_\iota \ar[rr]^{\rho_\barrhoH}&&\CG(R_\barrhoH)\ar[dl]^{\CG(\alpha)}\\
&\CG(R)&\\}
\end{equation}
that commutes up to strict equivalence, for  a unique ring homomorphism $\alpha\colon R_\barrhoH\to R$. After conjugation of $\iota$ by an element of $\CG(R)$ that reduces to the identity in $\CG(\BF)$, we may assume the diagram commutes. This means that we replace $H$ by a conjugate; this may necessitate a new choice for $M_{\WF}$ inside $\CG(\WF)$ (and thus of $M_R$ and $H_R$). Note that $M_R\subset H$, since $H$ surjects onto $H_\BF\subset M_\BF$.

So from now on, we assume that $H$ lies in $\CG(A)$ for $A:=\alpha(R_\barrhoH)\subset R$, so that $R_\barrhoH\to A$ is surjective.\footnote{The non-noetherian context may be avoided below by replacing $A$ by any of its Artinian quotients.} We let $M_A$ be the image of $M_R\subset H$ under $\CG(\alpha)\circ\rho_{\barrhoH}$, i.e., it is simply equal to $M_R$ under $\iota$. Then $H_A=H_R\cap \CG(A)=M_AH'_R\cap \CG(A)$ and $H_A$ contains $\iota(H)$. 

If we compose $\iota$ with the canonical map $\CG(A)\longto \CG(A/(p,\Fm_A^2))$, we obtain a homomorphism $H\to H_{A/(p,\Fm_A^2)}$ such that $R_\barrhoH\to A/(p,\Fm_A^2)$ is again surjective.
\begin{Lem}\label{HtoTangSurj}
The homomorphism $H\to H_{A/(p,\Fm_A^2)}$ is surjective. 
\end{Lem}
\begin{proof}
Let $A\to \BF[\eps]$ be any surjection in $\CAW$. Since $A$ is a quotient of $R_{\barrhoH}$, the induced deformation $\big[\,\rho_{\BF[\eps]}\colon H\to\CG(\BF[\eps])\,\big]$ is non-trivial. Note also that the image $\bar H$ of $H$ in $H_{\BF[\eps]}$  satisfies $\bar H^c=\bar H$ because of \Cref{Lem-OnCommutator}(b). Since we assume that \van\ holds and that either \liegen(ii) holds  or that \liec\ and \sch\ hold, we deduce from \Cref{BostonLemmaExt-New} that $\rho_{\BF[\eps]}$ is surjective. Since $A\to \BF[\eps]$ was arbitrary, the lemma is proved.
\end{proof}

\begin{proof}[Proof of  \Cref{Thm-Thm2}]
Because of the previous lemma, we are in a position to apply \Cref{Cor-BostonGenl}(b). Its hypotheses are satisfies, since we assume that \ns\ holds and either \liegen\ or \liec\ and \sch\ hold. This immediately yields $H=H_A$.
\end{proof}

From \Cref{Thm-Thm2}  and \Cref{Thm:MainOnStdSetup} we deduce:
\begin{Cor}\label{Rem:6.4-forAbsSimple}
Suppose that $(\CG,\BF,H_\BF,H'_\BF)$ satisfies \Cref{Cond:StdSetup}. Then under the conjunction of the following conditions, any closed subgroup $H\subset H_R$ that is residually full is a conjugate of $H_A\subset\CG(R)$ for a closed $\WF$-subalgebra $A\in\CAW$ of $R$:
\begin{enumerate}
\item $(\CG,\BF)$ is not exceptional in the sense of \Cref{Not:ExcSet}, and $($type $\CG,\BF)\notin\{(A_1,\BF_5)\}$.
\item $\CG$ is of type $A_n$ and $p\nmid n+1$, or $\CG$ is of type $B_n$,  $C_n$, $D_n$, $E_7$ or $F_4$ and $p\not=2$, or $\CG$ is of type $E_6$ or $G_2$ and $p\not=3$, or $\CG$ is of type $E_8$; i.e., \liegen\ holds.
\item If type $\CG =C_n$, then $|\BF|\notin\{3,5,9\}$.
\item If $\CG$ is non-split (and hence of types A, D or E${}_6$), then $|\BF|\ge4$.
\end{enumerate}
\end{Cor}

From \Cref{Thm-Thm2}  and  \Cref{Thm:StdHypForPuniv} we deduce:
\begin{Cor}\label{Rem:6.4-pUniversal}
Suppose that $(\CG,\BF,H_\BF,H'_\BF)$ satisfies \Cref{Cond-Puniv}. Then under the conjunction of the following conditions, for any residually full closed subgroup $H$ of $H_R$, there exists a closed $\WF$-subalgebra $A\in\CAW$ of $R$ such that $H^{(\infty)}$ is conjugate to $H_A\subset\CG(R)$.
\begin{enumerate}
\item $($type $\CG,\BF)\notin\{(A_1,\BF_?)_{?\in\{3,5\}}, (C_n,\BF_?)_{n\ge2,?\in \{3,5,9\} }\}$.
\item $\CG$ is of type $A_n$, $n\ge2$, $D_n$ or $E_n$, or $\CG$ is of type $A_1$, $B_n$,  $C_n$, $F_4$ and $p\neq2$, or $\CG$ is of type $G_2$ and $p\neq3$.
\item If $\CG^\der$ is non-split, then $|\BF|\ge4$.
\end{enumerate}
\end{Cor}

\begin{Rem}
If \liegen(ii) or if \liec(ii) and \sch\ are not satisfied, then \Cref{Ex-PGL} shows that the conclusion of \Cref{Thm-Thm2} need not hold.
\end{Rem}

The proof of \Cref{Thm-Thm2} is built via  \Cref{Cor-BostonGenl}(b) on $[\Fg^\der,\Fg^\der]=\Fg^\der$. This condition is satisfied in the setup of \Cref{Cor:BasicHypForGLn}. However in the particular case $\CG'=\GL_2$ and $\Char\BF=2$ the Lie group $\Fg^\der$ is not perfect -- for all other pairs $(n,\Char\BF)$ it is. The following example shows that the conclusion of \Cref{Thm-Thm2} does not hold for $(n,\Char\BF)=(2,2)$:

\begin{Ex} \label{Ex-SL2}
Let $(\CG,\Char\BF)=(\GL_2,2)$, and let $R=\BF[x_1,\ldots,x_d]/(x_1,\ldots,x_d)^3$ be in $\CAW$. Let $\Fm'\subset\Fm_R$ be the $\BF$-linear span of $\{x_1,\ldots,x_d,x_1^2,\ldots,x_d^2\}$. Define $H\subset \SL_2(R)$ as the subgroup generated by $H_1\cup\ldots\cup H_4$ for 
\[ H_1=\SL_2(\BF),\ \!H_2= \{ {\scriptstyle \SMat{1+a+a^2}{0}{0}{1-a}}\mid  a\in\Fm_R \},\ \!H_3= \{ \SMat{1}{b}{0}{1}\mid  b \in\Fm'\},\ \!H_4= \{ \SMat{1}{0}{c}{1}\mid  c \in\Fm'\} . \]
Then $H$ is a subgroup of $\SL_2(R)$ which surjects onto $\SL_2(R/\Fm_R^2)$. One can verify that any element in $H$ can be written in a unique way as a product $\gamma_1\cdot\ldots\cdot\gamma_4$ with $\gamma_i\in H_i$.  Then the order of $\SL_2(R)$ divided by the order of $H$ is $\#\BF^{2\dim_\BF \Fm_R/\Fm'}=\#\BF^{d(d-1)}$. Hence $H$ is a proper subgroup of $\SL_2(R)$ unless $d=1$. In particular, for $d>1$ and $|\BF|>4$, we see that \Cref{Thm-Thm2} requires condition~\liec.
\end{Ex}

The above counterexample is optimal in the following sense.
\begin{Prop}\label{Prop-SpclThm2}
Let $(\CG,\Char\BF)=(\GL_2,2)$ and suppose that $\BF\neq\BF_2$. Let $R$ in $\CAW$ satisfy $\dim_\BF \Fm_R/\Fm_R^2=1$. Then $H=\SL_2(R)$ for any closed subgroup $H\subset \SL_2(R)$ that surjects onto $\SL_2(R/(p,\Fm_R^2))$.
\end{Prop}
\begin{proof}
If $R$ is a quotient of $\BF[[X]]$, the result follows from  \cite[Thm.~3.6]{DevicPink}; this is a rather technical proof. 
If $R$ is a quotient of $\WF$, the result is a special case of \cite[Main Theorem]{Manoharmayum}.\footnote{There appears to be a small error in \cite[Thm.~3.5]{Manoharmayum}; for $\SL_2(\BF_4)$ the mod $2$ Schur multiplier is non-trivial.} However the second case also has a simple natural direct proof; cf.~e.g.~\cite{Vasiu1}. Our conditions imply that \ns, \sch\ and \liec(ii) hold. Then \Cref{Lem-BostonForW2} implies that $H_R=H_{W_2(\BF)}$ for $R=W_2(\BF)$. Denote by $\Gamma_n$ the kernel of $\SL_2(W_{n+1}(\BF))\to\SL_2(W_{n}(\BF))$, for $n\ge1$, with the map being the canonical reduction. Then the $p$-power map $X\mapsto X^p$ induces an isomorphism $\Gamma_n\to \Gamma_{n+1}$. From this one deduces easily that for any quotient $W_n(\BF)$ of $R$ the group $H_R$ surjects onto $H_{W_n(\BF)}$; the $p$-power map replaces the use commutators in the proof of   \Cref{Cor-BostonGenl}(b).
\end{proof}

\begin{Rem}\label{Rem:RelToMano}
The content of \cite[Main Theorem]{Manoharmayum} is the following result. Let $R$ be in $\CAW$. Let $H$ be a closed subgroup of $\GL_n(R)$ whose image in $\GL_n(\BF)$ contains $\SL_n(\BF)$. Suppose that $|\BF|\ge4$ and $\BF\neq\BF_4$ if $n=3$ and $\BF\neq\BF_5$ if $n=2$. Denote by $W_R$ the image of structure map $\WF\to R$.  Then $H$ contains a $\GL_n(R)$- conjugate of $\SL_n(W_R)$.

Our \Cref{Thm-Thm2} generalizes \cite[Main Theorem]{Manoharmayum} in all cases, except for $(n,p)=(2,2)$. \Cref{Ex-SL2} shows that our theorem cannot be expected to hold for $(n,p)=(2,2)$. 

It is possible to reduce the statement of  \cite[Main Theorem]{Manoharmayum} to the methods treated here, by reducing it to \Cref{Prop-SpclThm2} in the following way: Let $R$ be in $\CAW$. Choose a descending filtration by ideals $I_m$ of $R$ with $I_1=\Fm_R$ such that $I_m/I_{m+1}\cong\BF$ for $m\ge1$ and so that $\bigcap_m W_R+I_m=W_R$. Then $W_R+I_m/pW_R+I_{m+1}\cong\BF[\eps]$ for all $m\ge1$. And by induction on $n$ and {\bf up to conjugation}, one can find a descending sequence of closed subgroups $H_m\subset H$ such that $H_\infty:=\bigcap_m H_m\subset \SL_n(W_R)$ and $H_\infty$ surjects onto $\SL_n(\BF)$ under reduction. Now \Cref{Prop-SpclThm2} implies \cite[Main Theorem]{Manoharmayum}.
\end{Rem}

\appendix

\section{Appendix. Primer on affine group schemes over a base}
\label{Appendix}

In this appendix, we gather definitions and results, frequently used in this article, on various types of affine group schemes over arbitrary base schemes. For further details we refer to \cite{ConradSGA3}. We assume familiarity with the theory of affine algebraic groups over a field as in \cite{Borel,MilneAGS,Springer}.

Throughout this appendix we fix an arbitrary base scheme $S$ and an affine group scheme $\CG$ over $S$. We write $\pi$ for the structure morphism $\CG\to S$. In the main body of this work, $S$ will typically be the spectrum of a complete discrete valuation ring with finite residue field. 

\begin{Facts}\label{App-0}
As the map $\pi\colon\CG\to S$ is assumed to be affine, it is separated and quasi-compact; see \cite[Prop.-Def.~12.1]{GoertzWedhorn}. If $\pi$ is furthermore smooth, then it is also flat and locally of finite presentation, and hence of finite presentation; see \cite[Def.~10.34]{GoertzWedhorn}.
\end{Facts} 

\begin{Facts}[{\cite[II.4, in part 1.1, 1.2, 1.4, 4.8]{Bible} or \cite[{II.3. and II.4}]{SGA3I}.}]\label{App-Lie}
The {\em Lie Algebra} $\Lie(\CG/S)$ of $\CG$ over $S$ is the sheaf of $\CO_S$-modules which on affine $S$-schemes $\Spec R$ takes the value \[\Lie(\CG/S)(\Spec R)=\kernel(\CG(R[\eps])\to \CG(R)),\]
where $R[\eps]$ is the ring of dual numbers over $R$. This defines a functor from affine $S$-schemes of finite type to coherent $\CO_S$-modules. Moreover for any affine $S$-scheme $\Spec R$ there is an obvious action of the abstract group $\CG(R)$ on the abelian subgroup $\Lie(\CG/S)(\Spec R)$ of $\CG(R[\eps])$, and this induces the adjoint action $\mathrm{Ad}\colon \CG\to\Aut_S(\Lie(\CG/S))$.
Passing to Lie algebras defines a morphism 
\[\Lie(\CG/S)\to \End(\Lie(\CG/S)), X\mapsto [X,\cdot].\]
In fact the latter defines a Lie bracket on $\Lie(\CG/S)$ by $(X,Y)\mapsto [X,Y]$. One can also identify $\Lie(\CG/S)$ with the $\CO_S$-module of invariant derivations in the relative tangent sheaf $\Der_{\CO_S}(\CO_\CG,\CO_\CG)$, and then the Lie bracket is identified with the commutator bracket in $\Der_{\CO_S}(\CO_\CG,\CO_\CG)$. If $S=\Spec k$, we write $\Lie(\CG)$ for $\Lie(\CG/S)$. 

If $\Lie(\CG/S)$ is locally free of finite rank, the formation of $\Lie(\CG/S)$ commutes with any base change; otherwise one has to require that $S\to S'$ is flat. If $\CG$ is smooth over $S$, then $\Lie(\CG/S)$ is a locally free $\CO_S$-module of rank the relative dimension of $\CG$ over~$S$.
\end{Facts} 

A finitely presented closed subgroup scheme $\CH\subset\CG$ over $S$ is called {\em normal}, if for all $S'\in\Sch_S$ the subgroup $\CH(S')\subset \CG(S')$ is normal.
\begin{Prop}[{Identity component, \cite[15.6.5]{EGA4.3} and \cite[bottom p.~81]{ConradSGA3}}] \label{App-3}
Suppose $\CG$ is smooth over $S$. Then there exists a unique open subgroup scheme $\CG^o \subset \CG$ such that $(\CG^o)_s$ is the identity component of $\CG_s$ for all $s \in  S$. The subgroup scheme $\CG^o$ is normal in $\CG$. The formation $\CG\mapsto \CG^o$ commutes with any base change. 
\end{Prop}
\begin{Def}[Identity component] \label{App-3.5}
If $\CG$ is smooth over $S$, it is called {\em connected} over $S$, and $\pi$ is called {\em connected}, if $\CG=\CG^o$.
\end{Def}

For a finitely generated $\BZ$-module $(M,+,0_M)$, let $\BZ[M]$ be the Hopf algebra with multiplication defined by $m_1\otimes m_2\mapsto m_1+m_2$, comultiplication by $m\mapsto m\otimes m$, counit $1_\BZ\mapsto 0_M$, coinverse $m\mapsto -m$, and let $D(M)$ be the corresponding affine group scheme over $\BZ$; it is flat and of finite type.
\begin{Def}[{Multiplicative type and tori, \cite[App.~B]{ConradSGA3}}]\label{App-0.5} \leavevmode
\item
\begin{enumerate}
\item
The group scheme $\CG$ is called of {\em of multiplicative type} over $S$, and $\pi$ is of {\em multiplicative type}, if there is an fppf covering $\{S_i\}$ of $S$ such that for all $i$ there are  finitely generated abelian groups $M_i$ and isomorphisms $\CG\times_SS_i \cong D(M_i)\times_{\Spec\BZ}S_i$.
\item
The group scheme $\CG$ is called a {\em torus} if it is of multiplicative type and if there is a covering as in (a) with all $M_i$ free over~$\BZ$.\footnote{Equivalently: $\CG$ is a torus if and only if  $\pi$ is smooth, connected and of multiplicative type.}
\item
The group scheme $\CG$ is called a {\em split torus} if $\CG\cong D(M) \times_{\Spec\BZ}S$ with $M$ a free finitely generated $\BZ$-module.
\end{enumerate}
\end{Def} 
\cite[Prop.~14.51(6)]{GoertzWedhorn} shows that group schemes of multiplicative type are affine.

\smallskip

\begin{Def}[{Reductivity and semisimplicity}] \label{App-4} 
The group scheme $\CG$\footnote{Recall that we assume that $\pi$ is affine.}  is called {\em reductive} or {\em semisimple} over $S$, and $\pi$ is called {\em reductive} or {\em semisimple}, if $\pi$ is smooth and if for all geometric points $\overline s$ of $S$ the fiber $\CG^o_{\overline s}$ is reductive or semisimple, respectively.
\end{Def} 
The definition of reductivity in SGA3 is more restrictive than \Cref{App-4}(b), as noted in  \cite[\S~3.1]{ConradSGA3}. It also requires $\CG$ to be connected. The more general definition above is justified by the following result:

\begin{Prop}[{\cite[Prop.~3.1.3]{ConradSGA3}}]
\label{App-5}
Suppose $\pi$ is smooth. Then $\CG$ is reductive if and only if $\CG^o$ is reductive. In this case $\CG^o$ is clopen in $\CG$, the quotient $\CG/\CG^o$ exists and it is \'etale over $S$ and of finite presentation.
\end{Prop} 

Since being smooth and being affine are preserved under any base change, and since geometric fibers of any base change are geometric fibers of a given scheme we also have:
\begin{Prop} \label{App-9}
If $\pi$ is reductive or semisimple, then the respective property is preserved under any base change.
\end{Prop}

Let $\Sch_S$ be the category of schemes over $S$ and $\Gps$ that of abstract groups.
\begin{Prop}[{Center, \cite[Rem.~2.2.5 and Thm.~3.3.4]{ConradSGA3}}]\label{App-1}
Suppose $\pi$ is smooth and has connected geometric fibers, i.e., for all geometric point $\overline s$ of $S$ the fiber $\CG_{\overline{s}}$ is connected. Then the functor
\[\Sch_S\to  \Gps, S'\mapsto \{g\in\CG(S') \mid \forall g'\in\CG(S'): gg'=g'g \} \]
is representable by a finitely presented closed subgroup scheme $Z_\CG$ of $\CG$ over $S$. 

If in addition $\pi$ is connected reductive, then $Z_\CG$ is of multiplicative type, affine and flat over~$S$, and the formation of $Z_\CG$ commutes with any base change.
\end{Prop} 
\begin{Def}[Center]
The group scheme $Z_\CG$ is called the {\em center} of~$\CG$.
\end{Def}
Even if $\pi$ is reductive, the scheme $Z_\CG$ need not be smooth over $S$, as is witnessed for instance by $\CG=\SL_p$ and $S=\Spec \BF_p$, in which case $Z_\CG$ is the finite flat group scheme $\mu_p$ which is not smooth over $\BF_p$.

\begin{Def}[Tori and splitness]\label{App-12}
Let $\pi$ be reductive and let $S$ and \hbox{$\CG$ be connected.}\footnote{For $S$ not connected, the definition is more complicated, but we do not need it.} 
\begin{enumerate}
\item A {\em maximal torus in $\CG$} is a closed $S$-subgroup $\CT\subset\CG$ such that for any geometric point $\overline s$ of $S$ the fiber $\CT_{\overline s}$ is a maximal torus in the reductive group $\CG_{\overline s}$.
\item One calls $\CG$ {\em split reductive over $S$} if it contains a maximal torus $\CT$ that is split over $S$ \textbf{and} if for each root $\alpha\in\Hom(\CT,\BG_{m,S})$ the root space $\Lie(\CG/S)_\alpha$ is free over $\CO_S$ of rank $1$.
\end{enumerate}
\end{Def}

\begin{Def} \label{App-13}
Suppose $\CG$ and $\CG'$ are connected semisimple over $S$.
\begin{enumerate}
\item A homomorphism $\phi\colon \CG \rightarrow \CG'$ is called an {\em isogeny} if $f$ is finite, flat and surjective, and it is called a {\em central isogeny} if in addition $\kernel\phi$ lies in the center of $\CG$.
\item $\CG$ is called of {\em adjoint type} if $Z_\CG=1$.
\item $\CG$ is called {\em simply connected} if all its geometric fibers are simply connected. 
\end{enumerate}
\end{Def}
Since for $\CG$ reductive the formation of $Z_\CG$ commutes with any base change, a connected semisimple group $\CG$ is of adjoint type if and only if all its geometric fibers are of adjoint type.

\begin{Thm}[{\cite[Exer.~6.5.2, Prop.~3.3.5]{ConradSGA3}}]
\label{App-15}
For~\hbox{$\pi$ semisimple the following hold:}
\begin{enumerate}
\item There exists a semisimple, simply connected group scheme $\wt\pi\colon\wt\CG\to S$ and a central isogeny $\phi^\sco\colon \wt\CG\to\CG$ over $S$, and the pair $(\wt\CG,\phi^\sco)$ is unique up to unique isomorphism.
\item The maps $\phi^\ad\colon \CG\to \CG/Z(\CG)$ and $\phi\colon\wt\CG\to\wt\CG/Z(\wt\CG)$ are central isogenies, the $S$-group schemes $\CG/Z(\CG)$ and $\wt\CG/Z(\wt\CG)$ are isomorphic and semisimple of adjoint type, and under any isomorphism, the map $\phi$ factors uniquely via $\phi^\ad$.
\item The map $\phi^\sco$ induces a short exact sequence of finite flat $S$-group schemes
\[1\to \kernel\phi^\sco\stackrel{\phi^\sco}\to Z(\wt\CG)\to Z(\CG)\to 1.\] 
\end{enumerate}
\end{Thm}

\begin{Prop}[{\cite[Prop.~3.3.5]{ConradSGA3}}]\label{App-20}
Let $\phi\colon \CG\to\CG'$ be a central isogeny of connected reductive groups. Then $\CT'\mapsto\phi^{-1}\CT$ defines a bijection between maximal tori of $\CG'$ and maximal tori of~$\CG$.
\end{Prop}

Given an $S$-scheme $X$ carrying a $\CG$-action, an important construction is that of a quotient $X/\CG$. The course taken in SGA3 is as follows: embed the category $\Sch_S$ via Yoneda into the category of functors $\Sch_S^\opp\to\Sets$ by sending $X$ to $h_X\colon T\mapsto \Hom_S(T,X)$. Equip the latter category with the fppf-topology; we write $\Sh(\Sch_S)_\fppf$. It turns out that that any $h_X$ lie in $\Sh(\Sch_S)_\fppf$. One calls $F\in\Sh(\Sch_S)_\fppf$ representable if it is isomorphic to $h_X$ for some $X\in\Sch_S$. Now given $X$, $\CG$ as above, consider the presheaf $h_X/h_\CG\colon T\mapsto h_X(T)/h_\CG(T)$ and let $(h_X/h_\CG)^\sh$ be the associated sheaf in $\Sh(\Sch_S)_\fppf$. 
\begin{Def}
One calls $Y\in\Sch_S$ an {\em fppf-quotient} of $X$ by $\CG$ if $h_Y\cong (h_X/h_\CG)^\sh$.
\end{Def}
In general $Y$ need not exist. If it exists, an important result of Raynaud gives a comparison with (universal) geometric quotients: if $\CG$ is smooth and affine, if $X$ is locally of finite type and if the action is strictly free, i.e., $\CG\times_SX\to X\times_SX,(g,x)\mapsto (gx,x)$ is an immersion, then geometric and fppf-quotients agree; see~\cite[Ch.~4]{Edixhofen-vdGeer-Moonen}.

A special case is when $X$ itself is an $S$-group scheme $\CH$ and $\CG$ is a closed normal subgroup of $\CH$. Then $h_\CH/h_\CG$ carries a group law, a unit section and an inversion. This passes to the fppf-sheafification. Hence if $\CH/\CG$ exists as an fppf-quotient, it is automatically an $S$-group scheme.

\begin{Prop} [{Derived subgroup, \cite[Thm.~5.3.1]{ConradSGA3}}]\label{App-6}
If $\pi$ is connected and reductive, then the following hold: 
\begin{enumerate}
\item The fppf-sheafification of the {\em commutator subfunctor} $S' \mapsto[\CG(S'), \CG(S')]$ on $Sch_S$ is representable by a semisimple closed normal $S$-subgroup $\CG^\der \subset \CG$.
\item The fppf-quotient $\CG/\CG^\der$ is representable by a torus. 
\item The quotient map $\CG \rightarrow \CG/\CG^\der$ is initial among all homomorphisms from $\CG$ to an abelian sheaf, and the formation of $\CG^{\der}$ commutes with any base change on $S$.
\end{enumerate}
\end{Prop}
Hence if $\pi$ is connected and semisimple, then $\CG=\CG^\der$, since a non-trivial torus quotient would violate the required semisimplicity of all fibers of $\CG$.

\begin{Def}
Suppose that $\pi$ is reductive. Then we define $\CG^\der$ as $(\CG^o)^\der$.
\end{Def}
If $\pi$ is reductive, then the $S$-subgroup scheme $\CG^\der$ is closed and normal in $\CG$: By  \Cref{App-3}, the group scheme $\CG^o$ is closed and normal in $\CG$. To conclude observe that $\CG^\der$ is closed in $\CG^o$ and that normality is a property of the underlying functor of points, so that one can use that for an abstract group $G$ and a normal subgroup $N$ of $G$ the commutator subgroup $[N,N]$ is normal in~$G$.

\begin{Prop} \label{App-8}
If $\pi$ is reductive and $\CG/\CG^o$ is finite \'etale, then the following hold:
\begin{enumerate}
\item The quotient $\CG/\CG^\der$ is fppf-representable by a smooth affine $S$-group scheme.
\item The $S$-group scheme $\CG/\CG^\der$ is an extension of the finite \'etale $S$-group scheme $\CG/\CG^o$ by the torus $\CG^o/\CG^\der$.
\end{enumerate}
\end{Prop}
\begin{proof}
Consider the fppf-sheaf $\overline\CG:=(h_\CG/h_{\CG^\der})^\sh$ on $\Sch_S$. To prove (a) we need to show that it is representable by a smooth affine group scheme. Let $S'\to S$ be an fppf-cover over which $\pi_0:=\CG/\CG^o$ becomes a finite constant group scheme and each component has an $S'$-point. Then $\CG_{S'}$ is a disjoint union $\bigsqcup_{g\in\pi_0}\hat g\CG^o_{S'}$ where each $\hat g$ is a representative in $\CG(S')$ of $g\in\pi_0$.  Now $\CG^o_{S'}/\CG^\der_{S'}$ exists as a smooth affine fppf-quotient by \Cref{App-6}. Hence so does $\CG_{S'}/\CG^\der_{S'}=\bigsqcup_{g\in\pi_0}\hat g\CG^o_{S'}/\CG^\der_{S'}$. It follows that $\overline\CG$ after restriction to $\Sch_{S'}$ is representable by a smooth affine $S'$-scheme. By descent, see the proof of \cite[Thm.~14, \S~5.4]{Stix-FFGS}, it follows that $\overline\CG$ is representable by a scheme over $S$, and this proves~(a).

For the proof of (b) observe that it suffices to show that the natural diagram 
\[
1\to \CG^o/\CG^\der \to \CG/\CG^\der\to \CG/\CG^o\to 1
\]
of fppf-presheaves is a short exact sequence of groups under any evaluation at $T\in\Sch_S$. This remains true under fppf-sheafification, and hence also for the representing schemes that exist by (a), \Cref{App-6} and the hypotheses. This proves~(b).
\end{proof}

\bibliographystyle{alpha}
\bibliography{DEM}

\end{document}